\documentclass{article}
\usepackage{amsmath}
\usepackage{amssymb}
\usepackage{amsthm}
\usepackage{algorithm}
\usepackage{algorithmic}
\usepackage{multirow}% added 2019/10/31
\usepackage{array}
\usepackage{comment}
\usepackage{ascmac}
\usepackage[dvipdfmx]{graphicx}
\usepackage{color}
\usepackage{lscape}
\usepackage[top=20truemm,bottom=20truemm,left=16truemm,right=16truemm]{geometry}
\usepackage{subcaption}
\usepackage{here}
\def\@themcountersep{}
\newtheorem{theorem}{Theorem}[section]

\newtheorem{proposition}[theorem]{Proposition}

\renewcommand{\algorithmicrequire}{\textbf{Input:}}
\renewcommand{\algorithmicensure}{\textbf{Output:}}
\begin{document}
\title{Centering ADMM for the Semidefinite Relaxation of the QAP\thanks{This research was supported by the Japan Society for the Promotion of Science through a Grant-in-Aid for Challenging Exploratory Research (17K18946) and a Grant-in-Aid for Scientific Research ((B)19H02373) of the Ministry of Education, Culture, Sports, Science and Technology of Japan.} }

\author{Shin-ichi Kanoh\thanks{
Graduate School of Systems and Information Engineering, University of Tsukuba, Tsukuba, Ibaraki 305-8573, Japan. email: s1930138@s.tsukuba.ac.jp 
}
and 
Akiko Yoshise\thanks{Corresponding author. Faculty of Engineering, Information and Systems, University of Tsukuba, Tsukuba, Ibaraki 305-8573, Japan. email: yoshise@sk.tsukuba.ac.jp
}     
}

\date{January 2020 \\
Revised April 2020}

\maketitle

% REQUIRED
\begin{abstract}
We propose a new method for solving the semidefinite (SD) relaxation of the quadratic assignment problem (QAP), called Centering ADMM. Centering ADMM is an alternating direction method of multipliers (ADMM) combining the centering steps used in the interior-point method. The first stage of Centering ADMM updates the iterate so that it approaches the central path by incorporating a barrier function term into the objective function, as in the interior-point method. If the current iterate is sufficiently close to the central path with a sufficiently small value of the barrier parameter, the method switches to the standard version of ADMM. We show that Centering ADMM (not employing a dynamic update of the penalty parameter) has global convergence properties. To observe the effect of the centering steps, we conducted numerical experiments with SD relaxation problems of instances in QAPLIB. The results demonstrate that the centering steps are quite efficient for some classes of instances.
\end{abstract}

% REQUIRED

% REQUIRED

\section{Introduction}
\label{sec: Introduction}

The quadratic assignment problem (QAP) in the trace formulation \cite{aEDWARDS1980} is given by 
\begin{equation}
\label{eq:QAP}
\begin{array}{lll}
\mbox{QAP} & \mbox{minimize} & \langle FXD - C , X \rangle, \\
                   & \mbox{subject to} & X \in \Pi_n, 
\end{array}
\end{equation}
where $F, D \in \mathbb{S}^n$ are $n \times n$ real symmetric matrices, $C$ is a real $n \times n$ matrix, $\langle \cdot, \cdot \rangle$ denotes the trace inner product $\langle A, B \rangle = \mbox{Tr}(A^TB)$,  and $\Pi_n$ denotes the set of all $n \times n$ permutation matrices. 
The QAP was initially introduced to describe a location problem where the task is to assign $n$ facilities to $n$ locations in a way that minimizes the total cost \cite{aKOOPMANS1957, aLAWLER1963}.
It has many applications in areas as divergent as network design, VSLI design, and image processing (see, e.g., \cite{bPARDALOS1993, aPARDALOS1994, bCELA1998}). 

It is known that the QAP is NP-hard (see, e.g., \cite{aSAHNI1976}) and it has remained difficult to solve even if the size of the problem is moderate, e.g., $n=30$ \cite{aANSTREICHER2003}.
This fact implies that finding better lower and upper bounds of the optimal value is quite important to solve it.
An efficient tool for finding such bounds is semidefinite (SD) relaxation (see, e.g., \cite{aZHAO1998, aPOVH2009, aRENDL2007, aDEKLERK2010, aLUO2010, aPONG2016, aOLIVEIRA2018}).
SD relaxation in \cite{aZHAO1998} uses facial reduction to guarantee strict feasibility for both the relaxed problem and its dual and simplifies the constraints by making many of them redundant. 

However, SD relaxation still often forces us to solve a large-scale semidefinite programs.
Recently, Oliveira, Workowicz, and Xu  \cite{aOLIVEIRA2018} performed computational experiments showing that their alternating direction method of multipliers (ADMM) is promising for solving the SD relaxation of the QAP.
The authors derived the update formula for solving the SD relaxation of the QAP proposed in \cite{aZHAO1998}.
They also derived upper and lower bounds of the QAP from the solution obtained by their ADMM and compared their bounds with existing bounds. 
%Besides, using their bounds, Liao \cite{tLIAO2016} provided a branch and bound method for the QAP and showed its computational results.

The ADMM is a first-order method, which requires less computation per iteration and is highly scalable. 
However, its computation often becomes unstable and requires a significant number of iterations to compute an accurate solution.
On the other hand, the interior-point method is a second-order method and computationally expensive. 
Still, it is known to be stable and to require only a small number of iterations due to the centering steps, which forces the current iterate closer to the central path.
Our motivation is to accelerate the ADMM proposed in \cite{aOLIVEIRA2018} by combining the benefits of the ADMM and the interior-point method.
In this paper, we devise a new algorithm, called Centering ADMM, which is an ADMM combining the centering steps of the interior-point method in its first stage.

A similar approach was taken by Lin et al. \cite{LIN2018}, where the authors proposed an ADMM-Based Interior Point Method (ABIP) for solving large-scale linear programs. 
Their ABIP and our Centering ADMM are similar in the sense that both methods add a centering effect in the search direction by using a barrier function. 
However, Centering ADMM is different from ABIP for the following reasons: 
Centering ADMM is limited to solving the SD relaxation of the QAP.
Also, it performs centering steps only in the primal problem and in its first stage.
It reverts to (the standard) ADMM if the current iterate is sufficiently close to the central path with a sufficiently small value of the barrier parameter.
On the other hand, ABIP employs centering steps in both primal and dual problems using the homogeneous self-dual form of linear programs throughout its iterations.

To observe the effect of the centering steps, we conducted numerical experiments using instances in QAPLIB \cite{QAPLIB} and compared the solutions obtained with the ADMM in \cite{aOLIVEIRA2018} and with Centering ADMM.
The results demonstrate that the centering steps are quite efficient for some classes of instances.

The organization of the paper is as follows: After giving a brief introduction of (the standard) ADMM in section \ref{sec: Standard ADMM}, we describe its in details for solving the SD relaxation of the QAP proposed in  \cite{aOLIVEIRA2018} in section \ref{sec: SD relaxation}.
We provide our new method, Centering ADMM, in section \ref{sec: Centering ADMM}.
Then, we show that Centering ADMM (not employing a dynamic update of the penalty parameter) has global convergence properties in section \ref{sec:global convergence}.
In section  \ref{sec: Numerical experiments}, we numerically compare these two methods in terms of their lower bounds of the QAP for instances in QAPLIB \cite{QAPLIB}.

\section{Standard ADMM}
\label{sec: Standard ADMM}

Here, we give a brief introduction of the (standard) alternating direction method of multipliers.
To contrast with Centering ADMM, we will refer to the method as the Standard ADMM in the following.

The Standard ADMM was proposed by Glowinski and Marrocco \cite{aGLOWINSKI1975} and Gabay and Mercier \cite{aGABAY1976} for solving the following type of optimization problem:

\begin{equation}
\label{eq:ADMM}
\begin{array}{ll}
\mbox{minimize} &  f(x) + g(z) \\
\mbox{subject to} &  Ax + Bz = c, 
\end{array}
\end{equation}
where $x \in\mathbb{R}^n$, $z \in\mathbb{R}^m$ , $A \in\mathbb{R}^{k \times n}$, $B \in\mathbb{R}^{k \times m}$, $c \in\mathbb{R}^k$ and $f: \mathbb{R}^n \rightarrow \mathbb{R} \cup \{+\infty\}$ and $g: \mathbb{R}^m \rightarrow \mathbb{R} \cup \{+\infty\}$ are closed proper convex functions.

By introducing a penalty parameter $\rho > 0$, the augmented Lagrangian function for problem (\ref{eq:ADMM}) is given by
\begin{equation}
\label{eq:augmented Lagrangian}
\notag
L_\rho(x,z,y) := f(x) + g(z) + \langle y , Ax+Bz-c \rangle + \frac{\rho}{2} \| Ax+Bz-c \|^2,
\end{equation}
where $y \in\mathbb{R}^k$ is the dual variable or Lagrange multiplier.
Using the augmented Lagrangian function, ADMM updates the variables $(x^k, y^k, z^k)$ as follows:
\begin{eqnarray*}
x^{k+1} & := & \mbox{argmin}_{x} L_\rho(x, z^k, y^k), \\
z^{k+1} & := & \mbox{argmin}_{z} L_\rho(x^{k+1}, z, y^k), \\
y^{k+1} & := & y^k + \rho(Ax^{k+1}+Bz^{k+1}-c).
\end{eqnarray*}
The third update formula is a feature of the Standard ADMM. It updates the dual variable $y$ in its gradient direction, i.e., in the direction that increases the objective function value of the dual problem. For this reason, the Standard ADMM is sometimes considered to be a method that solves the dual problem.

\section{Standard ADMM for the SD relaxation of the QAP}
\label{sec: SD relaxation}

As shown in, e.g., \cite{aPARDALOS1994}, the set $\Pi_n$ of all permutation matrices can be represented as 
\begin{equation}
\Pi_n = \mathcal{O}_n \cap \mathcal{E}_n \cap \mathcal{N}_n = \mathcal{O}_n \cap \mathcal{E}_n \cap \mathcal{Z}_n, \notag
\end{equation}
where 
$\mathcal{O}_n := \{X \in \mathbb{R}^{n \times n} | XX^T = X^TX = I\}$, 
$\mathcal{E}_n := \{X \in \mathbb{R}^{n \times n} | Xe = X^Te = e\}$, $\mathcal{N}_n := \{X \in \mathbb{R}^{n \times n} | X \geq 0\}$, 
$\mathcal{Z}_n := \{X \in \mathbb{R}^{n \times n} | X \circ X - X = O \}$, 
$I \in \mathbb{R}^{n \times n}$ is the identity matrix, $e \in \mathbb{R}^n$ is the vector of ones, $A \circ B $ is the Hadamard product of $A \in \mathbb{R}^{n \times n}$ and  $B \in \mathbb{R}^{n \times n}$.
Using this fact, Zhao et al. \cite{aZHAO1998} showed that the QAP  (\ref{eq:QAP}) is equivalent to the following problem $\mbox{QAP}_{\mathcal O}$:
\begin{equation}
\label{eq:QAP_O}
\notag
\begin{array}{lll}
\mbox{QAP}_{\mathcal O} & \mbox{minimize}_{X} & \langle FXD - C , X \rangle \\
                   & \mbox{subject to} & XX^T = I, \\
                   & & X^TX = I, \\
                   & & \|Xe - e\|^2 + \|X^Te - e\|^2 = 0, \\
                   & & X \circ X - X = O.  
\end{array}
\end{equation}
We also define $\mathcal{S}^n := \{ X \in \mathbb{R}^{n \times n} \mid X = X^T  \}$, and for any $A, B \in \mathcal{S}^n$, we say $A \succeq B$ if $A-B$ is positive semidefinite.
By considering the dual problem of the Lagrange dual of  $\mbox{QAP}_{\mathcal O}$ and projecting the dual problem onto the minimal face,  they also showed that the following problem $\mbox{QAP}_{\rm R1}$ gives an SDP relaxation problem for the QAP  (\ref{eq:QAP}):
\begin{equation}
\label{eq:QAP_R1}
\notag
\begin{array}{lll}
\mbox{QAP}_{\rm R1}  & \mbox{minimize}_{R}  & \mbox{tr}(\hat{V}^TL_Q\hat{V} R) \\
                                 & \mbox{subject to} & \mathcal{G}_J (\hat{V} R \hat{V}^T) = E_{00}, \\
                                 & & R \succeq 0,
\end{array}
\end{equation}
where $L_Q \in\mathcal{S}^{n^2+1}$, $\hat{V} \in\mathbb{R}^{n^2+1 \times (n-1)^2+1}$, $E_{00} \in \mathcal{S}^{n^2+1}$ is the matrix whose $(1,1)$-element is one and all other elements are zero, $J \subseteq \{(i,j) \mid 1 \leq i, j \leq n^2\}$ are given, $R \in\mathcal{S}^{(n-1)^2+1}$ is the variable matrix, and $\mathcal{G}_J : \mathcal{S}^{n^2+1} \rightarrow \mathcal{S}^{n^2+1}$ is the gangster operator defined by
\begin{equation}
\label{eq:G_J}
(\mathcal{G}_J(Y))_{ij} := 
\begin{cases}
Y_{ij} & (i,j) \in J \ \mbox{or} \ (j,i) \in J,\\
0 & \text{otherwise}.
\end{cases}
\end{equation}

More precisely, the matrix $L_Q$ is as follows: 
\begin{equation}
\label{eq:L_Q}
L_Q := 
\begin{pmatrix}
0 & -\frac{1}{2} \mbox{vec} (C)^T \\
-\frac{1}{2} \mbox{vec} (C) & D \otimes F
\end{pmatrix},
\end{equation}
where  $\mbox{vec}(C) \in \mathbb{R}^{n^2}$ is the vector formed by stacking the columns of  $C$ on top of
one another, and  $A \otimes B \in \mathbb{R}^{mp \times nq}$ is the Kronecker product of $A \in \mathbb{R}^{m\times n}$ and $B \in \mathbb{R}^{p\times q}$.
The matrix $\hat{V}$ is the normalized matrix of $\bar{V} \in \mathbb{R}^{n^2+1 \times (n-1)^2+1}$
given by 
\begin{equation}
\bar{V} := 
\begin{pmatrix}
1 & 0^T \\
\frac{1}{n} (e \otimes e) & V \otimes V 
\end{pmatrix},
\notag
\end{equation}
where $V \in \mathbb{R}^{n \times n-1}$ is a full-rank matrix
\begin{equation}
V := 
\begin{pmatrix}
I_{n-1} \\
-e^T
\end{pmatrix},
\notag
\end{equation}
and $I_{n-1}$ is the $(n-1) \times (n-1)$ element matrix. 
As a result, the matrices $V$ and $\hat{V}$ satisfy $V^Te=0$ and $\hat{V}^T\hat{V} = I$.

Note that $\mbox{QAP}_{\rm R1} $ has a relative interior feasible solution since it is obtained by projecting an SD relaxation problem to the minimal face \cite{aZHAO1998}.

In \cite{aOLIVEIRA2018}, the authors succeeded in deriving upper bounds and lower bounds of the QAP by applying the Standard ADMM to $\mbox{QAP}_{\rm R1} $.
In what follows, we show how we can solve  $\mbox{QAP}_{\rm R1} $ by using ADMM according to the descriptions in \cite{aOLIVEIRA2018}.

\subsection{Representation of the QAP for which the Standard ADMM is applicable}

In \cite{aOLIVEIRA2018}, the authors  represent $\mbox{QAP}_{\rm R1} $ as a problem $\mbox{QAP}_{\rm R2}$ having two variables $R$ and $Y$ by introducing a new variable $Y$ and adding the equation $Y=\hat{V}R\hat{V}^T$:
\begin{equation}
\label{eq:QAP_R2}
\begin{array}{lll}
\mbox{QAP}_{\rm R2} & \mbox{minimize}_{R,Y}   &  \langle L_Q,Y \rangle  \\
                          & \mbox{subject to} & \mathcal{G}_J (Y) = E_{00},  \\
                          &                            & Y = \hat{V}R\hat{V}^T, \\
                          &                            & R \succeq 0.
\end{array}
\end{equation}
Next, they consider the following optimization problem for which the Standard ADMM is applicable:
\begin{equation}
\notag
\begin{array}{lll}
\mbox{QAP}_{\rm R3} & \mbox{minimize}_{R,Y}   &  \langle L_Q,Y \rangle + \mathcal{I}(R) + \mathcal{I}(Y)  \\
                          & \mbox{subject to} &  Y = \hat{V}R\hat{V}^T,
\end{array}
\end{equation}
where the second and third constraints of  $\mbox{QAP}_{\rm R2}$ are combined in the objective function as the corresponding indicator functions below:
\begin{equation}
\label{eq:I}
\mathcal{I}(R) :=
\begin{cases}
0 & \text{if} \ R \succeq 0, \\
\infty & \text{otherwise},
\end{cases}
\ \ \ \
\mathcal{I}(Y) :=
\begin{cases}
0 &\text{if} \  \mathcal{G}_J (Y) = E_{00}, \\
\infty & \text{otherwise}.
\end{cases}
\end{equation}

\subsection{Update formula of the variables in the Standard ADMM}

The augmented Lagrangian function for $\mbox{QAP}_{\rm R3}$ is given by
\begin{equation}
\notag
L_\rho(R,Y,Z) := \langle L_Q,Y \rangle + \mathcal{I}(R) + \mathcal{I}(Y) + \langle Z , Y - \hat{V}R\hat{V}^T \rangle + \frac{\rho}{2} \|Y - \hat{V}R\hat{V}^T \|^2_F,
\end{equation}
where $Z$ is the dual variable matrix.
Using this function, the variables are updated as follows:
\begin{eqnarray}
\label{eq:ADMM-R}
R^{k+1} &:=& \mbox{argmin}_{R} L_\rho(R,Y^k,Z^k), \\
\label{eq:ADMM-Y}
Y^{k+1} &:=& \mbox{argmin}_{Y} L_\rho(R^{k+1},Y,Z^k), \\
\label{eq:ADMM-Z}
Z^{k+1} &=:&  Z^k + \rho(Y^{k+1} - \hat{V}R^{k+1}\hat{V}^T).
\end{eqnarray}

As shown in \cite{aOLIVEIRA2018}, the above updates can be explicitly calculated  by
\begin{eqnarray}
\notag \label{eq:ADMM-R2}
R^{k+1} &=& \mathcal{P}_{\mathcal{S}_+}(\hat{V}^T(Y^k+\frac{1}{\rho}Z^k)\hat{V}), \\
\label{eq:ADMM-Y2}
Y^{k+1} &=& E_{00} + \mathcal{G}_{J^C}(\hat{V} R^{k+1} \hat{V}^T - \frac{1}{\rho} (L_Q +Z^k )), \\
\label{eq:ADMM-Z2}
Z^{k+1} &=&  Z^k + \rho(Y^{k+1} - \hat{V}R^{k+1}\hat{V}^T),
\end{eqnarray}
where 
$\mathcal{P}_{\mathcal{S}_+}$ is the orthogonal projection onto $\mathcal{S}_+$ and $J^C$ is a set given by
\begin{equation}
\notag
J^C := \{ (i,j) \mid 1 \leq i,j, \leq n^2+1)\} \setminus J.
\end{equation}

In \cite{aOLIVEIRA2018}, the authors added the constraints $0 \leq Y_{ij} \leq 1(\forall i,j)$ if $0 \leq \hat{V}R\hat{V}^T\leq 1$ are satisfied and showed that adding these constraints has a profound effect in accelerating the convergence of the Standard ADMM.

\subsection{Stopping conditions of the Standard ADMM}

The stopping conditions are given by the optimal conditions of the Lagrangian function,
\begin{equation}
\notag 
L(R,Y,Z) := \langle L_Q,Y \rangle + \mathcal{I}(R) + \mathcal{I}(Y) + \langle Z , Y - \hat{V}R\hat{V}^T \rangle.
\end{equation}
Let us define $f(R) := \mathcal{I}(R)$ and $g(Y) :=\langle L_Q,Y \rangle+ \mathcal{I}(Y)$.
Then, the Karush-Khun-Tucker conditions of $\mbox{QAP}_{\rm R3}$ are given by
\begin{eqnarray}
Y - \hat{V}R\hat{V}^T & = & O, \label{zyoukensiki3} \\
\partial f(R) - \hat{V}^TZ\hat{V} & \ni & O, \label{zyoukensiki1} \\
\partial g(Y) + Z & \ni & O. \label{zyoukensiki2} 
\end{eqnarray}
Here, as in Section 3.3 of \cite{aBOYD2011}, we call  (\ref{zyoukensiki3}) the primal feasibility constraint, and (\ref{zyoukensiki1}) and (\ref{zyoukensiki2}) the dual feasibility constraints, respectively. 
We see that condition (\ref{zyoukensiki2}) is always satisfied at each iteration $(R^{k+1},Y^{k+1},Z^{k+1})$.
The update $Y^{k+1} = \mbox{argmin}_{Y} L_\rho (R^{k+1},Y,Z^k)$ implies that 
\begin{equation}
\notag
\partial g(Y^{k+1}) + Z^k + \rho (Y^{k+1} - \hat{V}R^{k+1}\hat{V}^T)
 \ni O,
\end{equation}
and the update $Z^{k+1} = Z^k + \rho (Y^{k+1} - \hat{V}R^{k+1}\hat{V}^T)$ implies that
\begin{equation}
\notag
\partial g(Y^{k+1}) + Z^{k+1} \ni O.
\end{equation}

Thus, condition (\ref{zyoukensiki2}), i.e., the dual feasibility of $Y$, is always satisfied. 
This implies that we only need to consider  the primal feasibility  (\ref{zyoukensiki3}) of $Z$ and the dual feasibility (\ref{zyoukensiki1}) of $R$ as stopping conditions. 
The primal feasibility  (\ref{zyoukensiki3}) of $Z$ can be measured by the Frobenius norm of the residual vector  $r_{\rm p} = \| Y^{k+1} - \hat{V}R^{k+1}\hat{V}^T \|_F$.
To measure the dual feasibility (\ref{zyoukensiki1}) of $R$, we focus on the update formula (\ref{eq:ADMM-R}) of $R^{k+1}$.
Since we set $R^{k+1} := \mbox{argmin}_{R} L_\rho (R,Y^k,Z^k)$, $R^{k+1}$ satisfies
\begin{equation}
\notag
\partial f(R^{k+1}) - \hat{V}^TZ^k\hat{V} + \rho (R^{k+1}-\hat{V}^TY^k\hat{V}) \ni O,
\end{equation}
and this implies that 
\begin{equation}
\notag
\partial f(R^{k+1}) - \hat{V}^TZ^{k+1}\hat{V} + \hat{V}^TZ^{k+1}\hat{V} - \hat{V}^TZ^k\hat{V} + \rho (R^{k+1}-\hat{V}^TY^k\hat{V}) \ni O,
\end{equation}
and hence,
\begin{equation}
\notag
\partial f(R^{k+1}) - \hat{V}^TZ^{k+1}\hat{V} \ni \rho \hat{V}^T (Y^k - Y^{k+1})\hat{V}.
\end{equation}
This implies that if $\rho \hat{V}^T (Y^k - Y^{k+1})\hat{V} = O$ holds, the dual feasibility (\ref{zyoukensiki1}) of $R$ is guaranteed, and hence, the Frobenius norm of the matrix $r_{\rm d} = \| \rho \hat{V}^T (Y^k - Y^{k+1})\hat{V} \|_F$ can be considered as the residual value of the dual problem.

If the values of $r_{\rm p}$ and $r_{\rm d}$ at the iterate $(R^{k+1},Y^{k+1},Z^{k+1})$ are sufficiently small, we stop the Standard ADMM.
In fact, in \cite{aOLIVEIRA2018}, the authors chose $10^{-5}$ or $10^{-12}$ as tolerances, for $r_{\rm p}$ and $r_{\rm d}$, respectively, whereby if  $r_{\rm p}$ and $r_{\rm d}$ become smaller than these tolerances, the update is stopped.

\section{Centering ADMM}
\label{sec: Centering ADMM}

We propose a new algorithm, called Centering ADMM, to solve the SDP relaxation problem of QAP by combining the path-following scheme employed by the interior point methods with the Standard ADMM described in the previous section.

As in the interior point method, we incorporate a barrier function term with a barrier parameter $\mu > 0$ in the objective function of problem $\mbox{QAP}_{\rm R2}$ (\ref{eq:QAP_R2}) as follows:
\begin{equation}
\label{eq:BQAP_R2}
\begin{array}{lll}
\mbox{BQAP}_{\rm R2}  & \mbox{minimize}_{R,Y}   &  \langle L_Q , Y \rangle - \mu\mbox{log}(\mbox{det}(R))  \\
                      & \mbox{subject to} &  \mathcal{G}_J (Y) = E_{00}, \\
                      &                            &  Y = \hat{V} R \hat{V}^T, \\
                      &                            &  R \succ  O.
\end{array}
\end{equation}
We consider the following problem for which ADMM is applicable:
\begin{equation}
\label{eq:BQAP_R3}
\begin{array}{lll}
\mbox{BQAP}_{\rm R3}  & \mbox{minimize}_{R,Y}   &  \langle L_Q , Y \rangle - \mu\mbox{log}(\mbox{det}(R))  + \mathcal{I}(Y) \\
                      & \mbox{subject to} & Y = \hat{V} R \hat{V}^T.
\end{array}
\end{equation}
In what follows, we derive the update formula for solving the $\mbox{BQAP}_{\rm R3}$ with ADMM.
First, let us consider the following augmented Lagrangian function:
\begin{equation}
\label{eq:BQAP-function}
L_\rho^{\rm BQAP} (R,Y,Z) :=\langle L_Q , Y \rangle - \mu\mbox{log}(\mbox{det}(R))  + \mathcal{I}(Y) + \langle Z , (Y - \hat{V} R \hat{V}^T) \rangle + \frac{\rho}{2} \| Y - \hat{V} R \hat{V}^T \|^2_F.
\end{equation}

The updating formulas for $Y$ and $Z$ are the same as in the Standard ADMM, defined by (\ref{eq:ADMM-Y}) and (\ref{eq:ADMM-Z}), and have the explicit forms  (\ref{eq:ADMM-Y2}) and (\ref{eq:ADMM-Z2}).
On the other hand, the updating formula for $R$ is defined by 
\begin{eqnarray}
\notag 
R^{k+1} & := & \mbox{argmin}_{R} L_\rho^{\rm BQAP} (R,Y^k,Z^k) \\
\label{eq:CADMM-R1}
           & = & \mbox{argmin}_{R} \left\{- \mu\mbox{log}(\mbox{det}(R)) + \langle Z , (Y - \hat{V} R \hat{V}^T) \rangle + \frac{\rho}{2} \| Y - \hat{V} R \hat{V}^T \|^2_F \right\}.
\end{eqnarray}

We can easily check that the function $F(R)=- \mu\mbox{log}(\mbox{det}(R)) + \langle Z , (Y - \hat{V} R \hat{V}^T) \rangle + \frac{\rho}{2} \| Y - \hat{V} R \hat{V}^T \|^2_F$ is strictly convex for any $R \in \mathcal{S}^n_{++}$, and hence, $F(R)$ has a unique minimum solution in $\mathcal{S}^n_{++}$.
Using this result, we obtain the following proposition.

\begin{proposition}
Suppose that we obtain a spectral decomposition of the matrix $\hat{V}^T Z \hat{V}\ + \rho\hat{V}^T Y \hat{V}$ into an orthogonal matrix $P$ and a diagonal matrix $D$, as
\begin{equation}
\label{eq:decomposition}
\hat{V}^T Z \hat{V}\ + \rho\hat{V}^T Y \hat{V}=PDP^T.
\end{equation}
Then, the new iterate $R^{k+1}$ in (\ref{eq:CADMM-R1}) is given by
\begin{equation}
\label{eq:CADMM-R2}
R^{k+1} = P\bar{R}P^T,
\end{equation}
where the matrix $\bar{R}$ is a diagonal matrix whose elements are 
\begin{equation}
\label{eq:barR1}
\bar{R}_{ii} = \frac{D_{ii} + \sqrt{D_{ii}^2 + 4 \rho \mu} }{2 \rho} \ (i=1,2,\ldots,n).
\end{equation}
\end{proposition}

\begin{proof}
$R^{k+1}$ can be calculated in a similar way to what is proposed in \cite{aBOYD2011}.
Since the function $L_\rho^{\rm BQAP} (R,Y,Z) $ in (\ref{eq:BQAP-function}) is strictly convex, the gradient of $L_\rho^{\rm BQAP} (R,Y,Z) $ at $R^{k+1}$ should be $O$, and hence, we have
\begin{equation}
\notag
- \mu R^{-1} - \hat{V}^T Z \hat{V}\ + \rho ( -\hat{V}^T Y \hat{V}  + R  ) = O.
\end{equation}
The above equation and the decomposition (\ref{eq:decomposition}) imply that 
\begin{eqnarray*}
\rho R - \mu R^{-1} & = & \hat{V}^T Z \hat{V}\ + \rho\hat{V}^T Y \hat{V} \\
                               & = & PDP^T,
\end{eqnarray*}
and, by setting  $\bar{R} = P^TRP$, we have
\begin{equation}
\notag \label{eq:barR2}
\rho \bar{R} - \mu \bar{R}^{-1} = D.
\end{equation}
Thus, it turns out that $\bar{R}$ is a diagonal matrix, and  $\bar{R}$ is given by (\ref{eq:barR1}), since $R$ and $\bar{R}$ should be positive semidefinite.
\qed
\end{proof}

The update formulas of  Centering ADMM consist of  (\ref{eq:CADMM-R2}) for $R$,  (\ref{eq:ADMM-Y2}) for $Y$ and  (\ref{eq:ADMM-Z2}) for $Z$.

For the Standard ADMM, it was shown  in \cite{aBOYD2011} that the following dynamic update of the penalty parameter $\rho$, depending on the residual values $r_{\rm p}$ and $r_{\rm d}$, $\tau^{\rm incr}>1,\tau^{\rm decr}>1$, and $\theta>1$, is an efficient way to accelerate convergence:
\begin{equation}
\label{eq:update rho}
\rho^{k+1} = 
\begin{cases}
\tau^{\rm incr} \rho^k & r_{\rm p} > \theta r_{\rm d}, \\
\rho^k / \tau^{\rm decr}& r_{\rm d} > \theta r_{\rm p}, \\
\rho^k & \mbox{otherwise}.
\end{cases}
\end{equation}
We use this update with $\tau^{\rm incr} = \tau^{\rm decr} = 2$ and $\theta = 10$ for both Centering ADMM and the Standard ADMM. 

If the residual values of the primal and dual problems become smaller than $0.1$ at iteration $k$, we consider that the point $(X^k, Y^k, Z^k)$ is sufficiently close to the central path, and update the barrier parameter $\mu^k$ by $\mu^{k+1} = 0.75 \mu^k$, where the ratio $0.75$ was determined from experience.

If the barrier parameter $\mu^k$ is sufficiently small, e.g., $\mu^ {k} < 10^{-3}$, we consider that the point $(X^k, Y^k, Z^k)$ is sufficiently close to the set of optimal solutions.
In that case, the centering effect of the barrier function is not needed, and we can switch to the Standard ADMM instead of Centering ADMM.

The full description of Centering ADMM is in Algorithm \ref{algo:CenteringADMM}. 
 \begin{algorithm}
  \caption{Centering ADMM}
  \label{algo:CenteringADMM}
 \begin{algorithmic}[1]
 \renewcommand{\algorithmicrequire}{\textbf{Input:}}
 \renewcommand{\algorithmicensure}{\textbf{Output:}}
 \STATE initialization
 \\ \textit{$Y^0,Z^0,\mu^0,\rho^0$}
 \WHILE{$\mu^k < 10^{-3}$}
 \STATE Compute $R^{k+1} =  \mbox{argmin}_{R} L_{\rho^k}^{\rm BQAP} (R,Y^k,Z^k)$ by  (\ref{eq:CADMM-R2})
 \STATE Compute $Y^{k+1}$ by  (\ref{eq:ADMM-Y2})
 \STATE Compute $Z^{k+1}$ by  (\ref{eq:ADMM-Z2}) 
 \STATE $r_{\rm p} = \|Y^{k+1} - \hat{V}R^{k+1}\hat{V}^T \|_F$
 \STATE $r_{\rm d} = \|\rho^k \hat{V}^T (Y^k - Y^{k+1})\hat{V}\|_F$
 \IF {$r_{\rm p} > 10 r_{\rm d}$}
 \STATE $\rho^{k+1} =  2\rho^k$
 \ELSE 
 \IF {$r_{\rm d} > 10 r_{\rm p}$}
 \STATE $\rho^{k+1} = \rho^k/2 $ 
 \ELSE
 \STATE $\rho^{k+1} = \rho^k$
 \ENDIF
 \ENDIF
 \STATE $r = \mbox{max} \left( r_{\rm p}, r_{\rm d} \right)$
 \IF {$r < 0.1$}
 \STATE $\mu^{k+1} = 0.75\mu^k$
 \ELSE
 \STATE $\mu^{k+1} = \mu^k$
 \ENDIF
 \ENDWHILE
 \STATE Start the Standard ADMM with the initial point $(Y^k,Z^k)$
 \end{algorithmic} 
 \end{algorithm}

%%%%%%%%%%%%%%%%%%%%%%%%%%%%%%%%
\section{Global convergence of Centering ADMM}
\label{sec:global convergence}

In section 3.2 of \cite{aBOYD2011}, it has been shown that ADMM (with a fixed penalty parameter $\rho > 0$) in section 2 has global convergence properties if the following  assumptions hold (see also Appendix A of  \cite{aBOYD2011}):
\begin{description}
\item[Assumption 1.]
The (extended-real-valued) functions $f: \mathbb{R}^n \rightarrow \mathbb{R} \cup \{+\infty\}$ and  $g: \mathbb{R}^n \rightarrow \mathbb{R} \cup \{+\infty\}$ are closed, proper, and convex.
\item[Assumption 2.]
The unaugmented Lagrangian $L_0$ (the augmented Lagrangian $L_{\rho}$ with $\rho =0$) has a saddle point.
\end{description}
Centering ADMM is an ADMM for problem $\mbox{BQAP}_{\rm R3}$  (\ref{eq:BQAP_R3}), 
and the functions $f$ and $g$ are given by
\begin{equation}
\notag
f(R) := - \mu\mbox{log}(\mbox{det}(R)), \ 
g(Y) := \langle L_Q , Y \rangle + \mathcal{I}(Y),
\end{equation}
where $\mu > 0$ and $L_Q$ is given by (\ref{eq:L_Q}),  and the augmented Lagrangian function $L_{\rho}^{\rm BQAP} (R,Y,Z)$ is defined by (\ref{eq:BQAP-function}). 
From the definition (\ref{eq:I}) of $\mathcal{I}(Y)$, $\mathcal{I}(Y)$ is the indicator function for an affine space,  and hence, we can see that Assumption 1 holds for Centering ADMM.
Thus, if Assumption 2 holds, i.e., the unaugmented Lagrangian function $L_{0}^{\rm BQAP} (R,Y,Z)$  has a saddle point, then Centering ADMM (with no update of $\rho$ , i.e., $\rho^k = \rho > 0$ for every $k=0,1,\ldots$) has global convergence properties.
In what follows, we show that the function $L_{0}^{\rm BQAP} (R,Y,Z)$ has a saddle point (Proposition \ref{prop:saddle point}). 

%Centering ADMM is based on problem $\mbox{BQAP}_{\rm R2}$ (\ref{eq:BQAP_R2}) by incorporating a barrier term in $\mbox{QAP}_{\rm R2}$ (\ref{eq:QAP_R2}).
The Karush-Khun-Tucker conditions of  $\mbox{BQAP}_{\rm R2}$  (\ref{eq:BQAP_R2}) turn out to be
\begin{eqnarray}
& & \hat{V}^T(L_Q + A)\hat{V}R = \mu I, \notag \\
& & \mathcal{G}_J (\hat{V}R\hat{V}^T) = E_{00}, \notag \\
& &  A \in \mathcal{S}^{n^2+1}_J,  \notag \\
& & R \succeq O, \notag \\
& & \hat{V}^T(L_Q + A)\hat{V} \succeq O, \notag
\end{eqnarray}
where
\begin{equation}
\label{eq:S_J}
\mathcal{S}^{n^2+1}_J= \{  X \in \mathcal{S}^{n^2+1} \mid \mathcal{G}_J (X) = X  \},
\end{equation}
and by setting $A = -(L_Q + Z)$, we obtain the following system:
\begin{eqnarray}
& & (-\hat{V}^TZ\hat{V})R = \mu I \label{eq:KKT1}, \\
& & \mathcal{G}_J (\hat{V}R\hat{V}^T) = E_{00} \label{eq:KKT2}, \\
& & -(L_Q + Z) \in \mathcal{S}^{n^2+1}_J,  \label{eq:KKT3} \\
& & R \succeq 0, \label{eq:KKT4} \\
& & -\hat{V}^TZ\hat{V} \succeq O. \label{eq:KKT5}
\end{eqnarray}
For any fixed $\mu > 0$, the objective function of problem $\mbox{BQAP}_{\rm R2}$ is strictly convex, and hence the above system has a unique solution.
Proposition \ref{prop:saddle point} guarantees that the solution is a saddle point of the function $L_{0}^{\rm BQAP} (R,Y,Z)$.

\begin{proposition}
\label{prop:saddle point}
For any fixed $\mu > 0$, the unique solution $(R^*,Z^*)$ of the system (\ref{eq:KKT1})-(\ref{eq:KKT5}) satisfies
\begin{equation}
\notag
\max_{Z} L_0^{\rm BQAP} (R^*,Y^*,Z) = L_0^{\rm BQAP} (R^*,Y^*,Z^*) = \min_{R,Y} L_0^{\rm BQAP} (R,Y,Z^*), 
\end{equation}
where $Y^* = \hat{V}R^*\hat{V}^T$.
% and $L_{0}^{\rm BQAP} (R,Y,Z)$ is the augmented barrier function $L_{\rho}^{\rm BQAP} (R^*,Y^*,Z) $ defined by (\ref{eq:BQAP-function}) with $\rho = 0$.
That is,  $(R^*,Y^*, Z^*)$ is a saddle point of  the function $L_{0}^{\rm BQAP} (R,Y,Z)$.
\end{proposition}

\begin{proof}
From the definition (\ref{eq:I}) of $\mathcal{I}(Y)$, (\ref{eq:KKT2}) and setting $Y^* =\hat{V}R^*\hat{V}^T$, we can easily see that 
\begin{eqnarray*}
\max_{Z} L_0^{\rm BQAP} (R^*,Y^*,Z) 
& = & \max_{Z} \left\{ \langle L_Q , Y^* \rangle + \mathcal{I}(Y^*)  - \mu\mbox{log}(\mbox{det}(R^*)) + \langle Z , (Y^* - \hat{V} R^* \hat{V}^T) \rangle  \right\} \\
& = & \max_{Z} \left \{ \langle L_Q , Y^* \rangle  - \mu\mbox{log}(\mbox{det}(R^*)) \right \} \\
& = & L_0^{\rm BQAP} (R^*,Y^*,Z^*).
\end{eqnarray*}
Let us show that $L_0^{\rm BQAP} (R^*,Y^*,Z^*) = \min_{R,Y} L_0^{\rm BQAP} (R,Y,Z^*)$ holds.
Before doing so, we define the set $\mathcal{Y}$ as follows:
\begin{equation}
\label{eq:Y}
\mathcal{Y} := \{ Y \in \mathcal{S}^{n^2+1} \mid \mathcal{G}_J (Y) = E_{00}  \} .
\end{equation}
Then, from the definition (\ref{eq:I}) of $\mathcal{I}(Y)$, we see that
\begin{eqnarray}
\notag
\min_{R,Y}L_0^{\rm BQAP} (R,Y,Z^*) 
\notag
& = & \min_{R,Y} \left\{  \langle L_Q , Y \rangle + \mathcal{I}(Y)  - \mu\mbox{log}(\mbox{det}(R)) + \langle Z^* , (Y - \hat{V} R \hat{V}^T) \rangle \right\} \\
\notag
& = & \min_{R,Y}  \left\{  \langle L_Q + Z^*  , Y \rangle + \mathcal{I}(Y)  - \mu\mbox{log}(\mbox{det}(R)) + \langle Z^* , - \hat{V} R \hat{V}^T \rangle \right\} \\
\notag
& = & \min_{Y} \left\{  \langle L_Q + Z^*  , Y \rangle + \mathcal{I}(Y) \right\} + \min_{R} \left \{ - \mu\mbox{log}(\mbox{det}(R)) + \langle Z^* , - \hat{V} R \hat{V}^T \rangle \right\} \\
\label{eq:BQAP2}
& = & \min_{Y \in \mathcal{Y}} \left\{  \langle L_Q + Z^*  , Y \rangle  \right\} + \min_{R} \left \{ - \mu\mbox{log}(\mbox{det}(R)) + \langle Z^* , - \hat{V} R \hat{V}^T \rangle \right\}.
\end{eqnarray}

Since $Z^*$ satisfiers  (\ref{eq:KKT3}), the definition (\ref{eq:S_J}) of $\mathcal{S}^{n^2+1}_J$, the definition (\ref{eq:G_J}) of $\mathcal{G}_J(\cdot)$, and the definition (\ref{eq:Y}) of $\mathcal{Y}$ imply that 
\begin{eqnarray}
\notag
\min_{Y \in \mathcal{Y}} \left\{  \langle L_Q + Z^*  , Y \rangle  \right\} 
\notag
& = & \min_{Y \in \mathcal{Y}} \left\{  \langle \mathcal{G}_J(L_Q + Z^*)  , Y \rangle  \right\} \\
\notag
& = & \min_{Y \in \mathcal{Y}} \left\{  \langle L_Q + Z^*  , \mathcal{G}_J(Y) \rangle  \right\} \\
\notag
& = & \min_{Y \in \mathcal{Y}} \left\{  \langle L_Q + Z^*  , E_{00} \rangle  \right\} \\
\label{eq:BQAP3}
& = & \langle L_Q + Z^*  , E_{00} \rangle.
\end{eqnarray}

For any fixed $\mu > 0$, the function $- \mu\mbox{log}(\mbox{det}(R)) + \langle Z^* , - \hat{V} R \hat{V}^T \rangle$ is strictly convex at any $R \succ O$,  and the second term of (\ref{eq:BQAP2}) has a unique minimum solution $R$ satisfying
\begin{equation}
- \hat{V}^T Z^* \hat{V} - \mu R^{-1} = 0. \notag
\end{equation}
Thus, the fact that $R^*$ satisfies (\ref{eq:KKT1}) implies that $R^*$ is the minimum solution of the second term of (\ref{eq:BQAP2}) and we have
\begin{equation}
\label{eq:BQAP4}
 \min_{R} \left \{ - \mu\mbox{log}(\mbox{det}(R)) + \langle Z^* , - \hat{V} R \hat{V}^T \rangle \right\}
 = - \mu\mbox{log}(\mbox{det}(R^*)) + \langle Z^* , - \hat{V} R^* \hat{V}^T \rangle.
\end{equation}
Equations (\ref{eq:BQAP2}), (\ref{eq:BQAP3}), and (\ref{eq:BQAP4}) imply that
 \begin{eqnarray}
 \notag
\min_{R,Y}L_0^{\rm BQAP} (R,Y,Z^*) 
\notag
& = & \min_{Y \in \mathcal{Y}} \left\{  \langle L_Q + Z^*  , Y \rangle  \right\} + \min_{R} \left \{ - \mu\mbox{log}(\mbox{det}(R)) + \langle Z^* , - \hat{V} R \hat{V}^T \rangle \right\}  \\
\label{eq:BQAP5}
& = & \langle L_Q + Z^*  , E_{00} \rangle - \mu\mbox{log}(\mbox{det}(R^*)) + \langle Z^* , - \hat{V} R^* \hat{V}^T\rangle. 
\end{eqnarray}
By a discussion similar to derive (\ref{eq:BQAP3}), we also see that 
\begin{eqnarray}
\notag
L_0^{\rm BQAP} (R^*,Y^*,Z^*) 
\notag
& = & \langle L_Q , Y^* \rangle + \mathcal{I}(Y^*)  - \mu\mbox{log}(\mbox{det}(R^*)) + \langle Z^* , (Y^* - \hat{V} R^* \hat{V}^T) \rangle \\
\notag
& =  & \langle L_Q + Z^* , Y^* \rangle  - \mu\mbox{log}(\mbox{det}(R^*)) + \langle Z^* , - \hat{V} R^* \hat{V}^T \rangle \\
\label{eq:BQAP6}
& = & \langle L_Q + Z^* , E_{00} \rangle  - \mu\mbox{log}(\mbox{det}(R^*)) + \langle Z^* , - \hat{V} R^* \hat{V}^T \rangle.
\end{eqnarray}
Therefore,  (\ref{eq:BQAP5}) and  (\ref{eq:BQAP6}) guarantee that $L_0^{\rm BQAP} (R^*,Y^*,Z^*) = \min_{R,Y} L_0^{\rm BQAP} (R,Y,Z^*)$ holds.
\end{proof}

%%%%%%%%%%%%%%%%%%%%%%%%%%%%%%%%
\section{Numerical experiments}
\label{sec: Numerical experiments}

We conducted numerical experiments to examine the performance of Centering ADMM in comparison with the Standard ADMM on the QAPLIB instances with symmetric matrices in \cite{aBURKARD1997,QAPLIB}.
We used MATLAB R2018b on an Intel (R) Core (TM) i7-6700 CPU @ 3.40GHz 3.41GHz machine. 
For the sake of limiting the computational time, we only dealt with instances of size $n \leq 40$.

We set the initial points and the accuracy parameters of Centering ADMM and the Standard ADMM, as 
\begin{equation}
\notag
Y^0 = I, \  Z^0 = -I, \ \mu^0 =1, \ \rho^0 = n, \ \epsilon_{\rm r} = 0.1
\end{equation}
for all instances.

For both methods, we limited the number of iterations to 10000 and outputted the obtained lower bounds  every 100 iterations.
Figures \ref{fig:chr} -- \ref{fig:Tai-a} are plots of the difference between the lower bounds obtained by the two methods, i.e., (the value of the lower bound obtained by Centering ADMM) -  (the value of the lower bound obtained by the Standard ADMM) every 100 iterations for each class of instances.

The horizontal axis shows the number of iterations, and the vertical axis shows the difference between the obtained lower bounds. 
If the difference is positive (negative), it means that Centering ADMM (the Standard ADMM) computes a better lower bound. 

The results allow us to make the following observations for each class of instances.

%-----------------------------------------------------------
\subsection{Observations for the class of ``chr'' instances (Figure \ref{fig:chr}) }

Except for instance ``hr25a,'' the difference is positive when the number of iterations becomes larger than 1000, which implies that Centering ADMM obtains a better lower bound than the Standard ADMM at each iteration for most  instances of this class.

%-----------------------------------------------------------
\subsection{Observations for the class of ``Had'' instances (Figure \ref{fig:Had}) }

For ``Had'' instances, a significant increase in the lower bound occurs during the first few iterations, and the Standard ADMM obtains better results than Centering ADMM for every instance. 

Figure \ref {fig:Had12_difference} shows the results of solving ``Had12'' by using the Standard ADMM as the difference between the lower bound and the optimal value every 100 iterations. 
We can see that the lower bound is sufficiently close to the optimal value at 300 iterations.

Table \ref {tb:tableHad} compares the results obtained by the Standard ADMM and those obtained by Centering ADMM. 
The table lists the problem name (Prob.), its optimal value (Opt.), and for each ADMM, the pair of the lower bound (LB) and the number of iterations (\#Iter) for which the difference from the optimum value becomes less than or equal to 0.5. 

In every case, the lower bound is sufficiently close to the optimal value within 2000 iterations, and this suggests that the centering effect is not required.

%-----------------------------------------------------------
\subsection{Observations for the class of ``Kra'' instances (Figure \ref{fig:Kra}) }

The difference is positive once the number of iterations becomes larger than 1100, which implies that Centering ADMM obtains a better lower bound than the Standard ADMM at each iteration for all instances of this class.

%-----------------------------------------------------------
\subsection{Observations for the classes of ``Rou'' and ``Scr''  instances (Figure \ref{fig:Rou}) }

The difference is always positive, which implies that Centering ADMM obtains a better lower bound than the Standard ADMM at each iteration for all instances of these classes.

%-----------------------------------------------------------
\subsection{Observations for the class of ``Nug'' instances (Figure \ref{fig:Nug}) }

Similary to the results for the ``Had'' instances, a significant increase in the lower bound occurs at the beginning of the iterations for the ``Nug'' instances, and the Standard ADMM obtains better results than Centering ADMM for every instance. 
A difference from the results for the ``Had'' instances is that the upper bound is not attained at 10,000 iterations. 
At an early stage, the lower bound increases rapidly to a certain value, but after that, the increase becomes quite small.

Table \ref {tb:tableNug1} lists the differences between the lower bounds obtained by the Standard ADMM and by  Centering ADMM. 
The table shows the problem name (Prob.), the value at which the increase in the ratio of the differences starts to slow down (Slow Down LB), the number of iterations at which the difference becomes less than or equal to 1 (Small Diff. \#Iter.), and the lower bound obtained by the Standard ADMM at 10000 iterations (LB at 10000 Iter.). 
We omit the lower bound obtained by Centering ADMM since it is quite close to the value of ``LB at 10000 Iter.''

Table \ref {tb:tableNug2} compares the results obtained by the Standard ADMM and by Centering ADMM. 
The table shows the problem name (Prob.), its optimal value (Opt.), and for each ADMM, the pair of the lower bound (LB) and the number of iterations for which the increase in the ratio of the lower bound values starts to slow down (\#Iter.). 

Similarly to the results for the the ``Had'' instances, the lower bound is sufficiently close to the optimal value within 3000 iterations, and this this suggests that the centering effect is not required in any of the cases.

%-----------------------------------------------------------
\subsection{Observations for the class of ``Tai-a''  instances (Figure \ref{fig:Tai-a}) }

Except for instance ``Tai35a,'' Centering ADMM obtains a better lower bound than the Standard ADMM at almost every iteration for all instances of this class.

%-----------------------------------------------------------
\subsection{Observations for the classes of ``Els19'' and ``Tho30'' instances (Figure \ref{fig:ElsandTho}) }

For instance ``Els19,'' Centering ADMM obtains better lower bounds than the Standard ADMM. 
For instance ``Tho30,'' the difference is positive at almost every iterations.
Thus, Centering ADMM is better than the Standard ADMM for this instance as well.

\section{Concluding remarks}
\label{sec: Concluding}

We devised a new method for solving a semidefinite (SD) relaxation of the quadratic assignment problem (QAP), called Centering ADMM. 
Centering ADMM is an alternating direction method of multipliers (ADMM) combining the centering steps used in the interior-point method.
The first stage of Centering ADMM updates the iterate such that it approaches the central path by incorporating a barrier function term in the objective function, as in the interior-point method. 
If the current iterate is sufficiently close to the central path with a sufficiently small value of the barrier parameter, the method then proceeds to the Standard ADMM.
We showed that Centering ADMM  (not employing a dynamic update of the penalty parameter) has global convergence properties.
To observe the effect of the centering steps, we conducted numerical experiments with SD relaxation problems of the instances in QAPLIB \cite{QAPLIB}.
The results demonstrate that Centering ADMM is quite efficient for some instances, e.g.,  all instances  in ``chr,'' ``Kra,'' ``Rou,'' and ``Scr, '' and instances ``Els19'' and ``Tho30.''

Our future research will include further discussions on convergence of Centering ADMM,  providing a way of determining valid initial parameters, and an extension of the method to general semidefinite programs.

\begin{figure}[H]
\begin{tabular}{ccc}
%----------------------------------
\begin{minipage}{0.3\columnwidth}
\centering
\includegraphics[width = 50mm,pagebox=cropbox,clip]{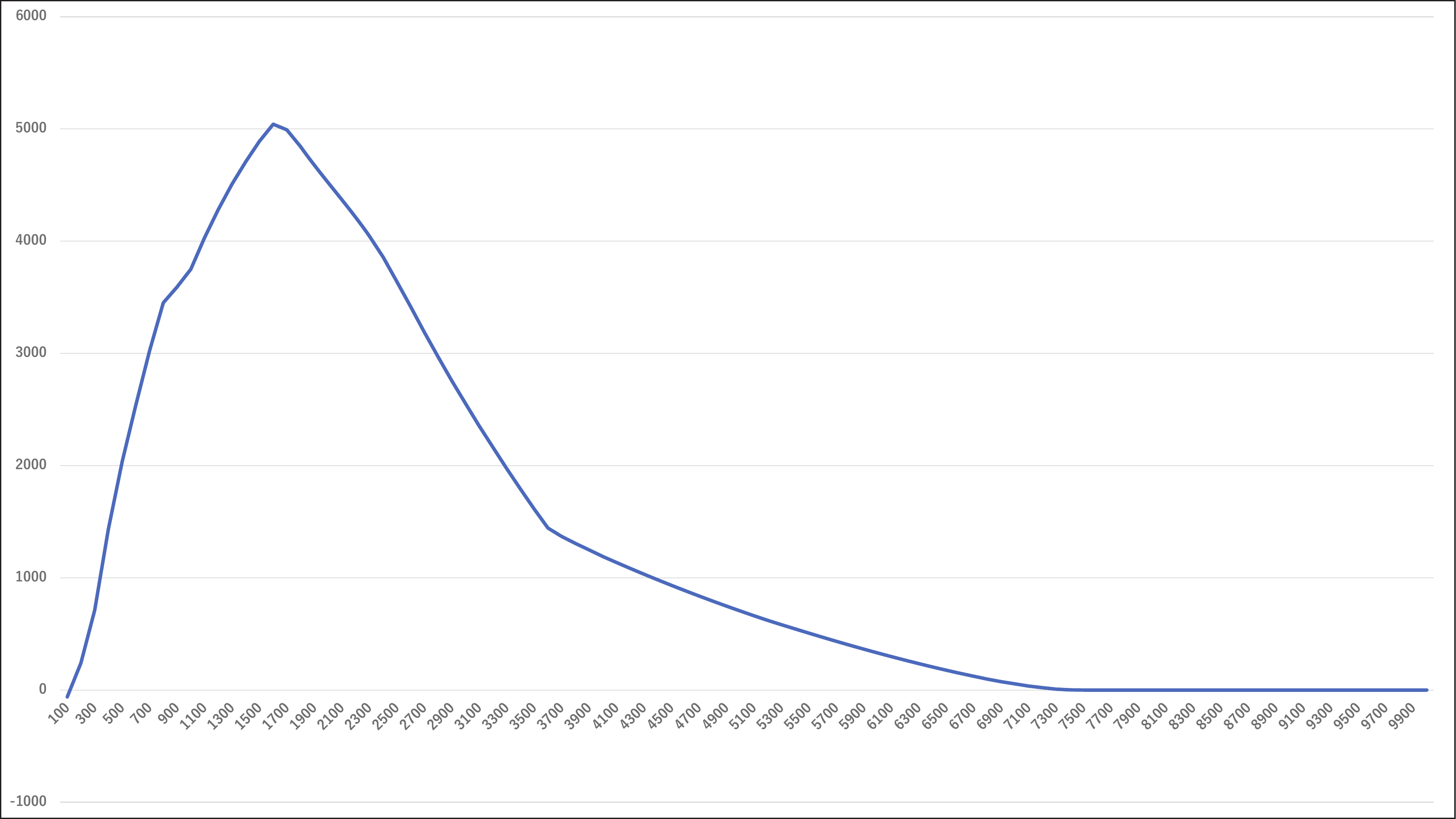}
\subcaption{chr12a}
\label{fig:chr12a}
\end{minipage} &
\begin{minipage}{0.3\columnwidth}
\centering
\includegraphics[width = 50mm,pagebox=cropbox,clip]{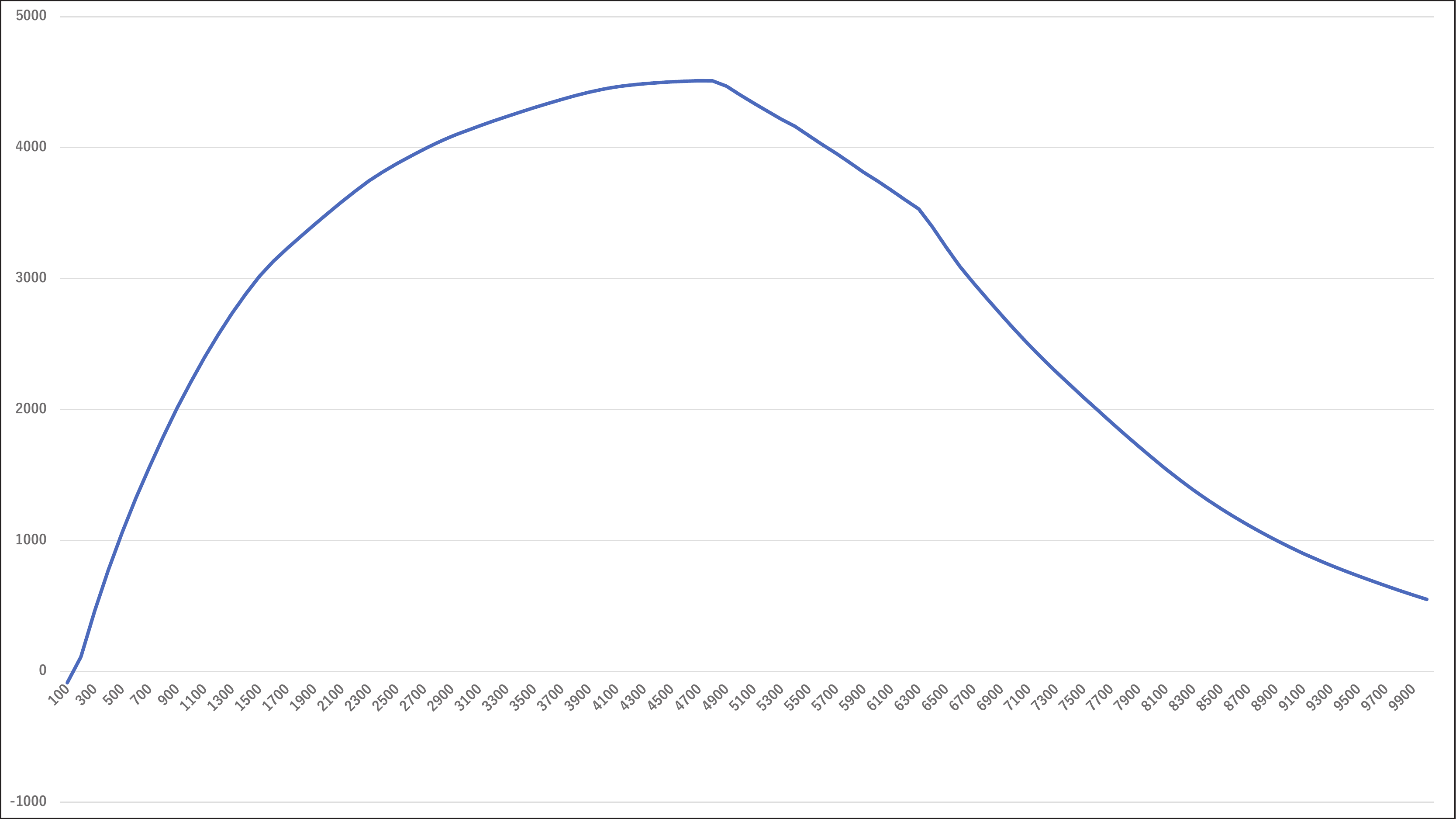}
\subcaption{chr15a}
\label{fig:chr15a}
\end{minipage} &
\begin{minipage}{0.3\columnwidth}
\centering
\includegraphics[width = 50mm,pagebox=cropbox,clip]{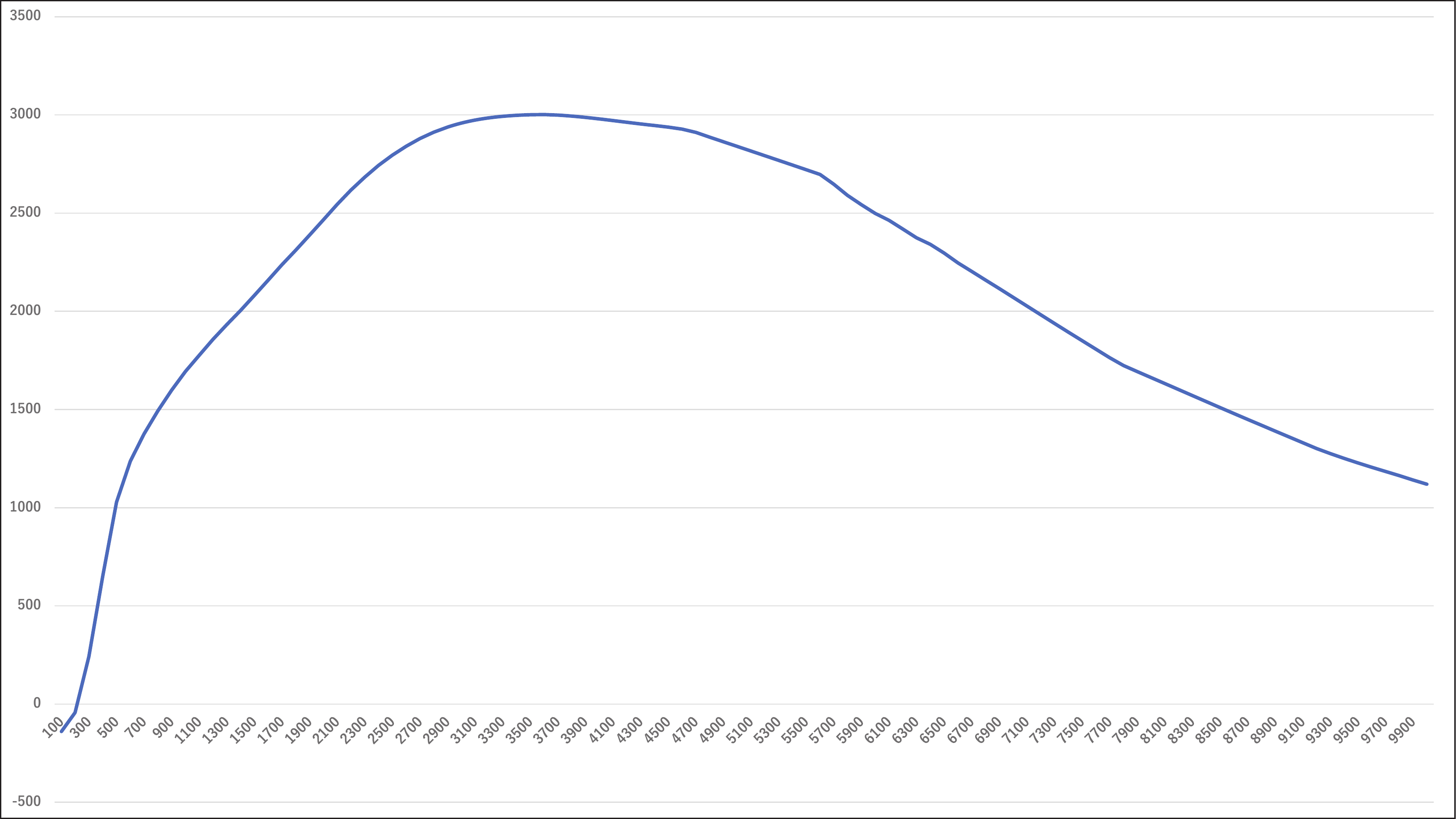}
\subcaption{chr18a}
\label{fig:chr18a}
\end{minipage} \\
%----------------------------------
\begin{minipage}{0.3\columnwidth}
\centering
\includegraphics[width = 50mm,pagebox=cropbox,clip]{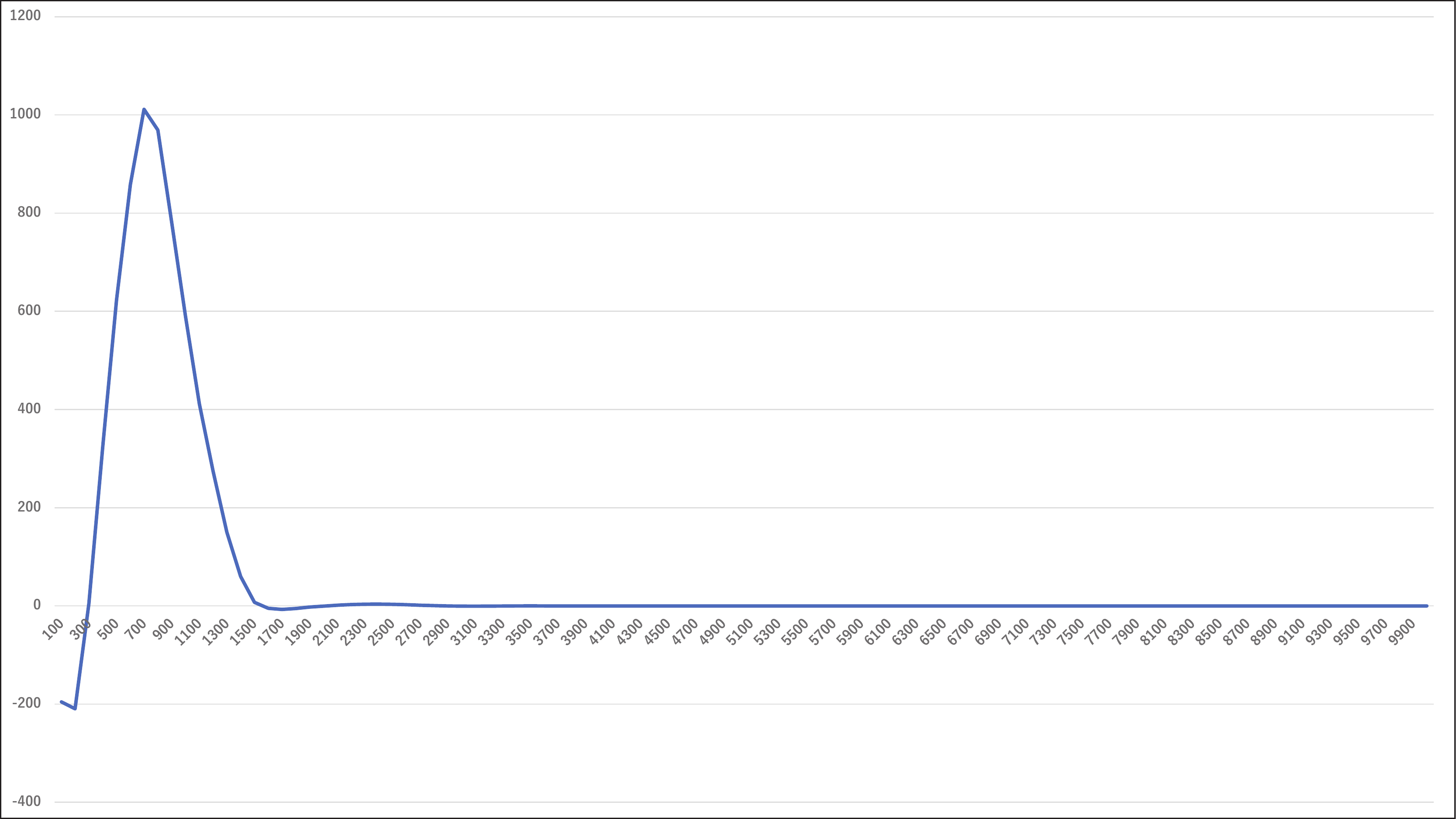}
\subcaption{chr20a}
\label{fig:chr20a}
\end{minipage} &
\begin{minipage}{0.3\columnwidth}
\centering
\includegraphics[width = 50mm,pagebox=cropbox,clip]{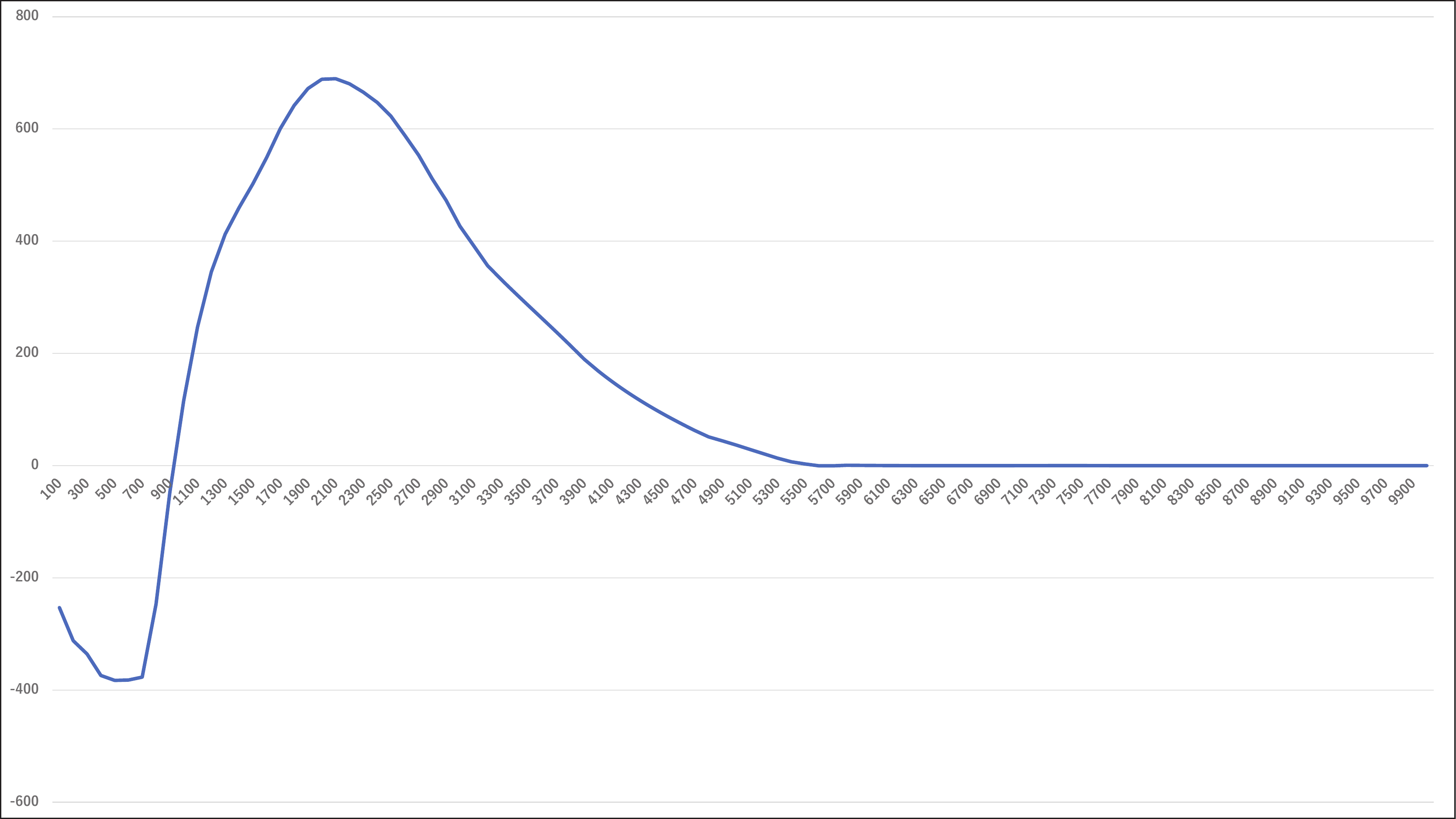}
\subcaption{chr22a}
\label{fig:chr22a}
\end{minipage} &
\begin{minipage}{0.3\columnwidth}
\centering
\includegraphics[width = 50mm,pagebox=cropbox,clip]{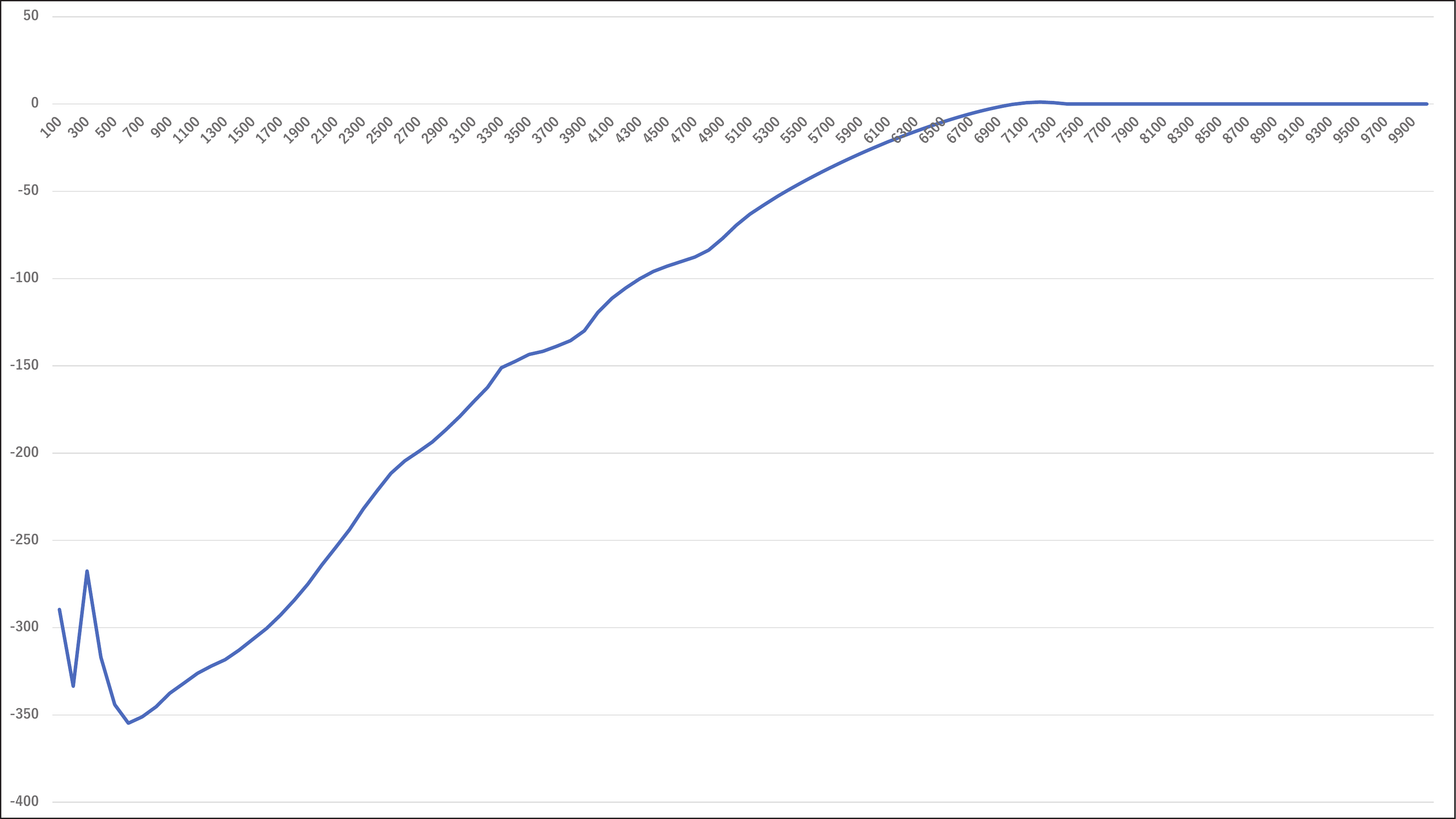}
\subcaption{chr25a}
\label{fig:chr25a}
\end{minipage} \\
%----------------------------------
\begin{minipage}{0.3\columnwidth}
\centering
\includegraphics[width = 50mm,pagebox=cropbox,clip]{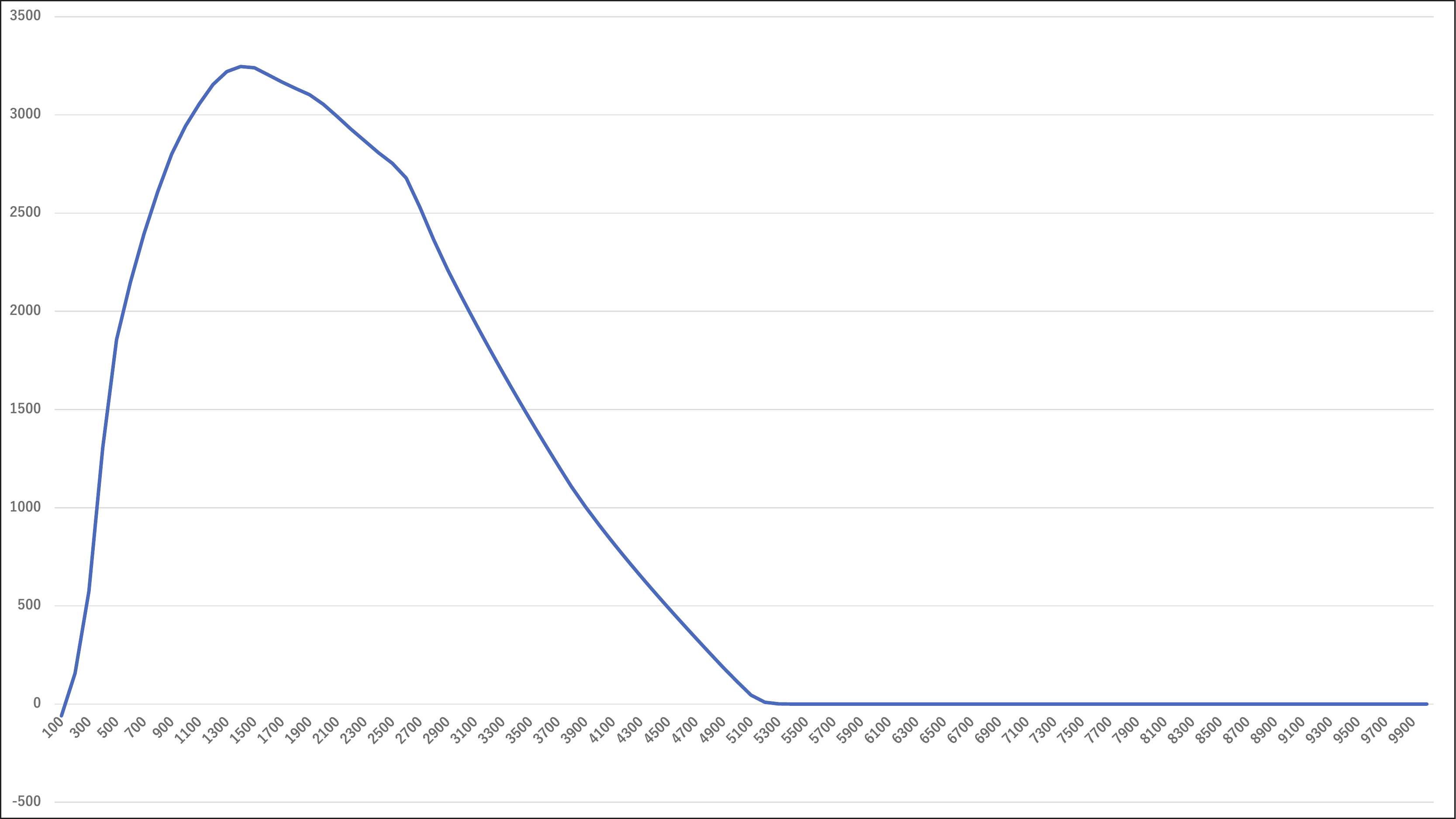}
\subcaption{chr12b}
\label{fig:chr12b}
\end{minipage} &
\begin{minipage}{0.3\columnwidth}
\centering
\includegraphics[width = 50mm,pagebox=cropbox,clip]{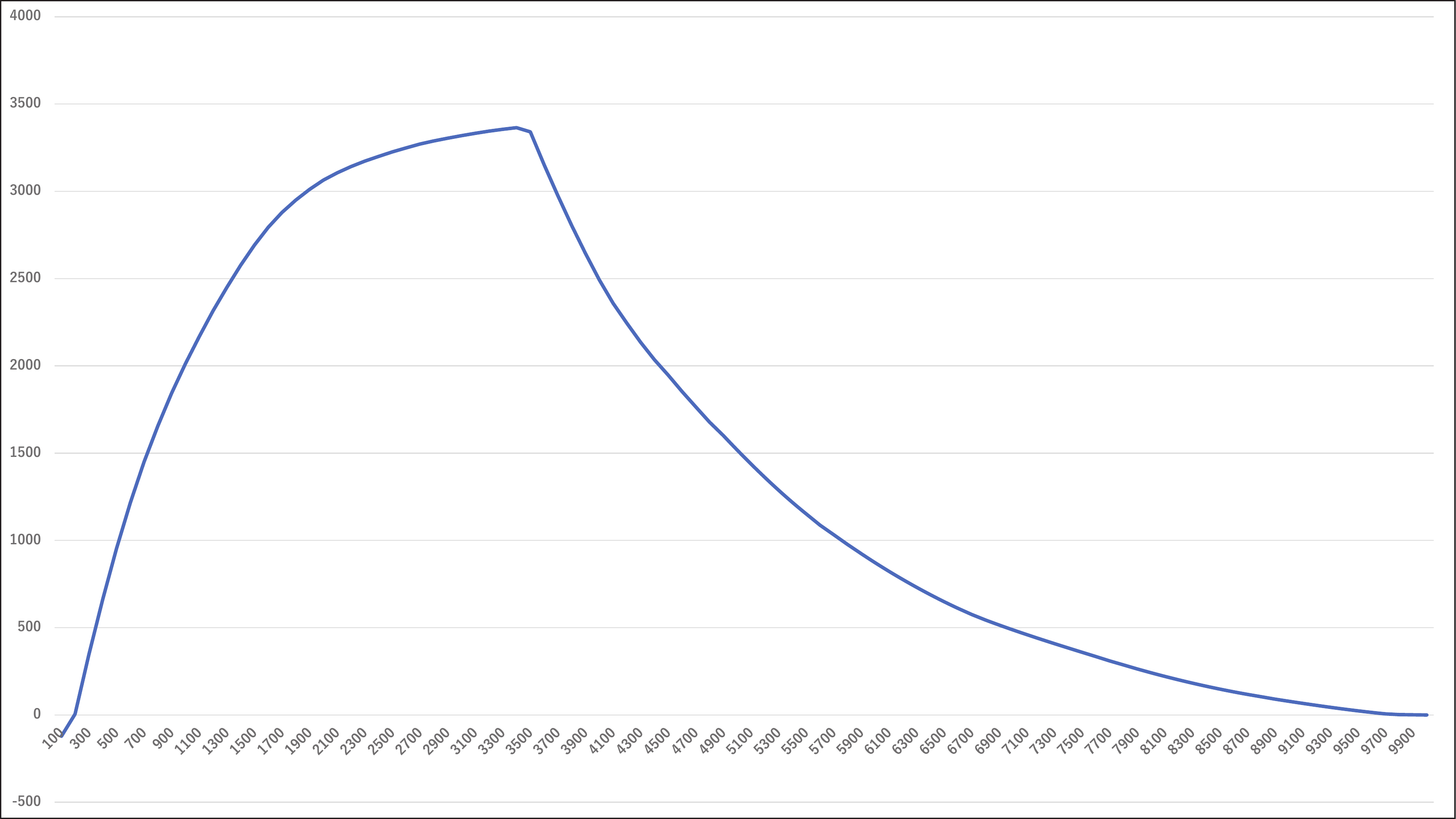}
\subcaption{chr15b}
\label{fig:chr15b}
\end{minipage} &
\begin{minipage}{0.3\columnwidth}
\centering
\includegraphics[width = 50mm,pagebox=cropbox,clip]{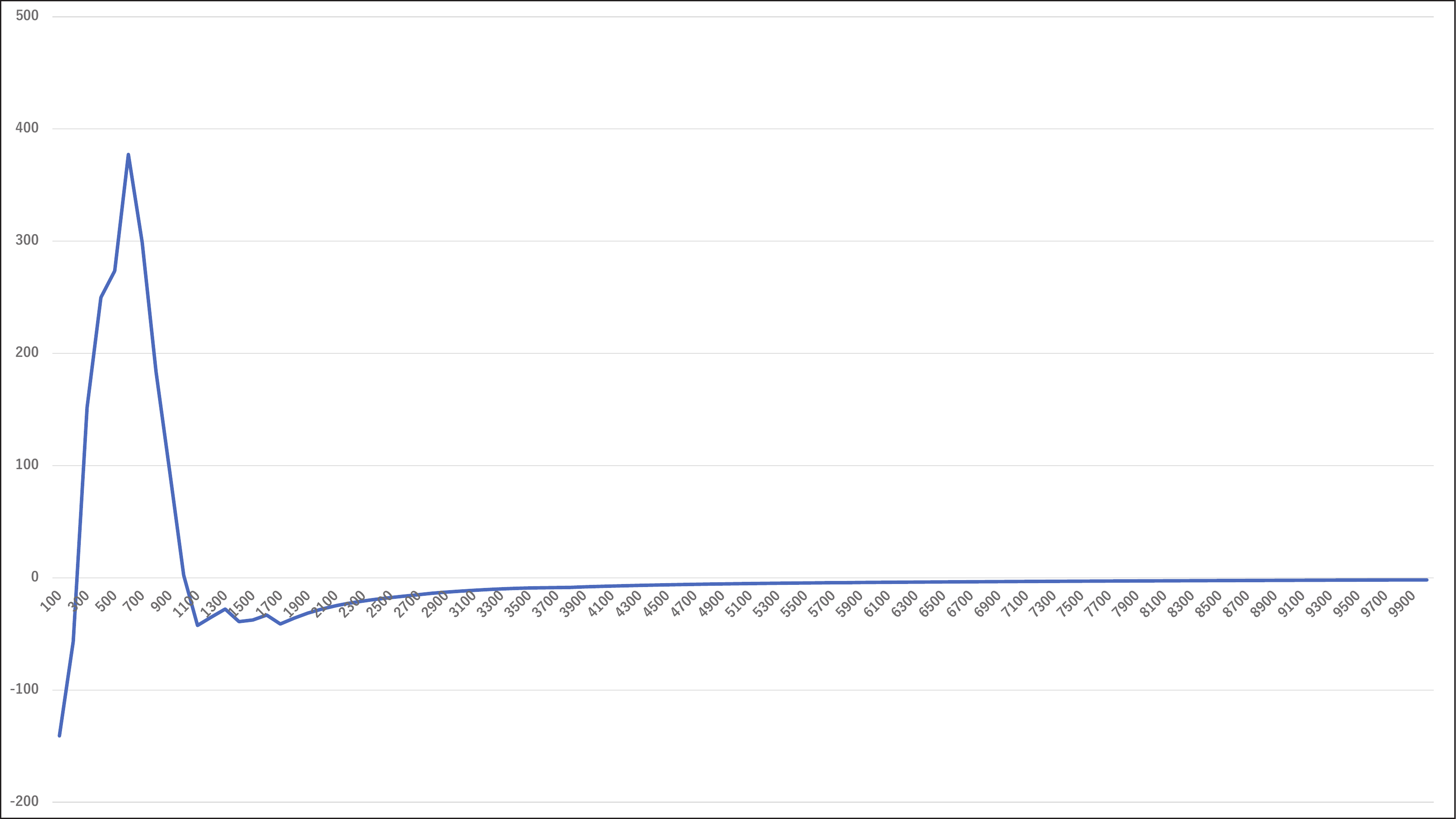}
\subcaption{chr18b}
\label{fig:chr18b}
\end{minipage} \\
%----------------------------------
\begin{minipage}{0.3\columnwidth}
\centering
\includegraphics[width = 50mm,pagebox=cropbox,clip]{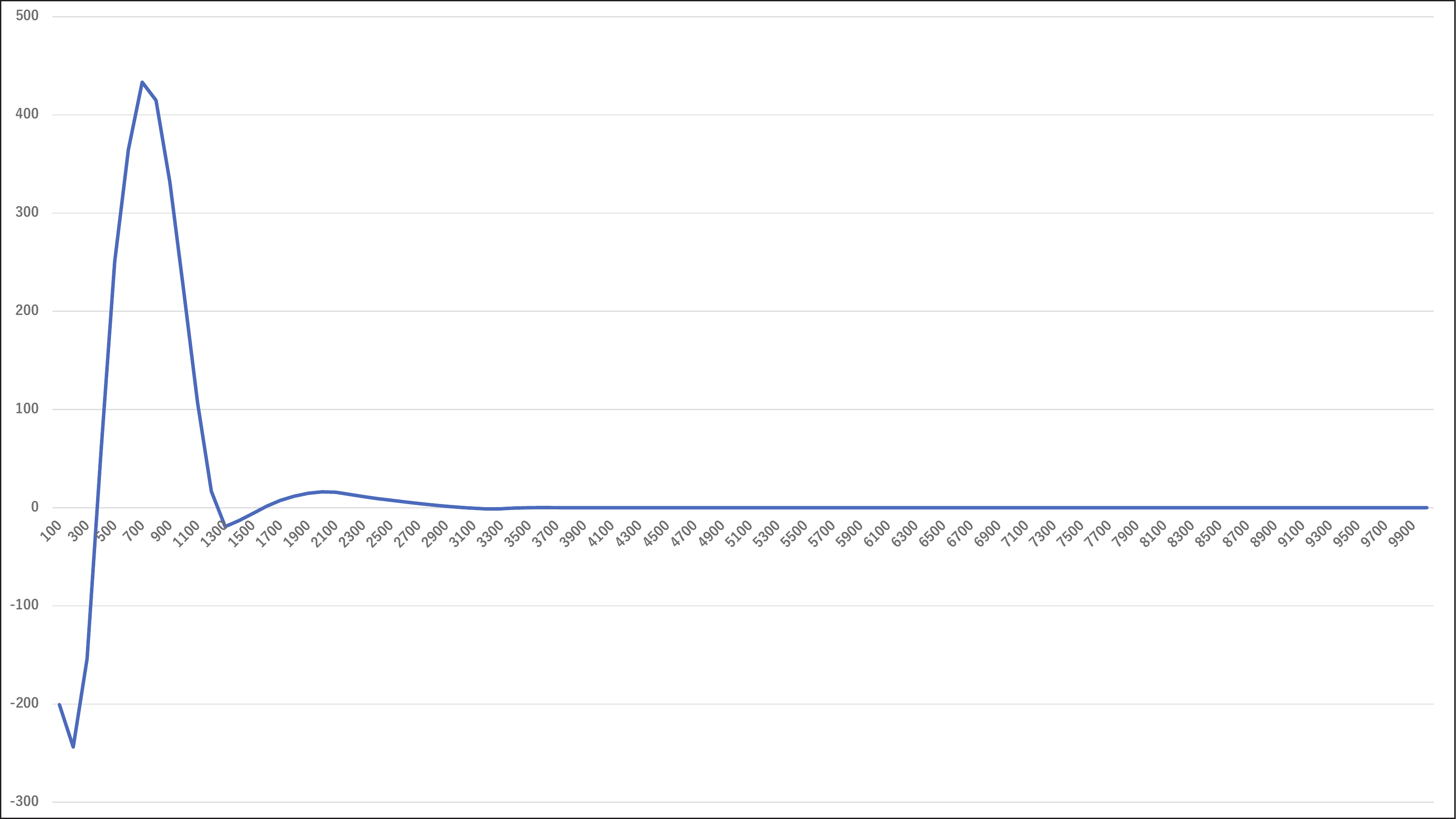}
\subcaption{chr20b}
\label{fig:chr20b}
\end{minipage} &
\begin{minipage}{0.3\columnwidth}
\centering
\includegraphics[width = 50mm,pagebox=cropbox,clip]{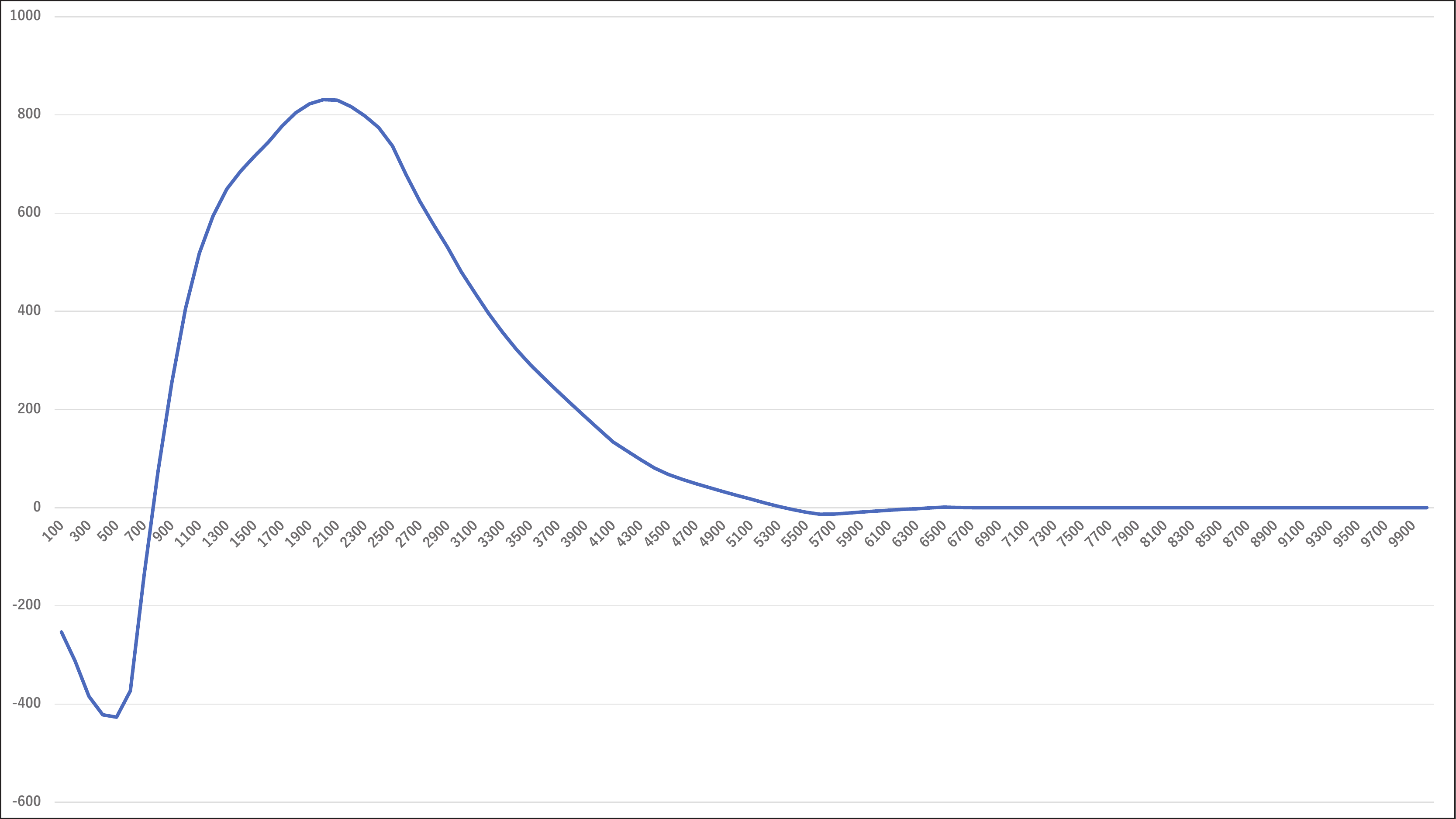}
\subcaption{chr22b}
\label{fig:chr22b}
\end{minipage} &
\begin{minipage}{0.3\columnwidth}
\centering
\includegraphics[width = 50mm,pagebox=cropbox,clip]{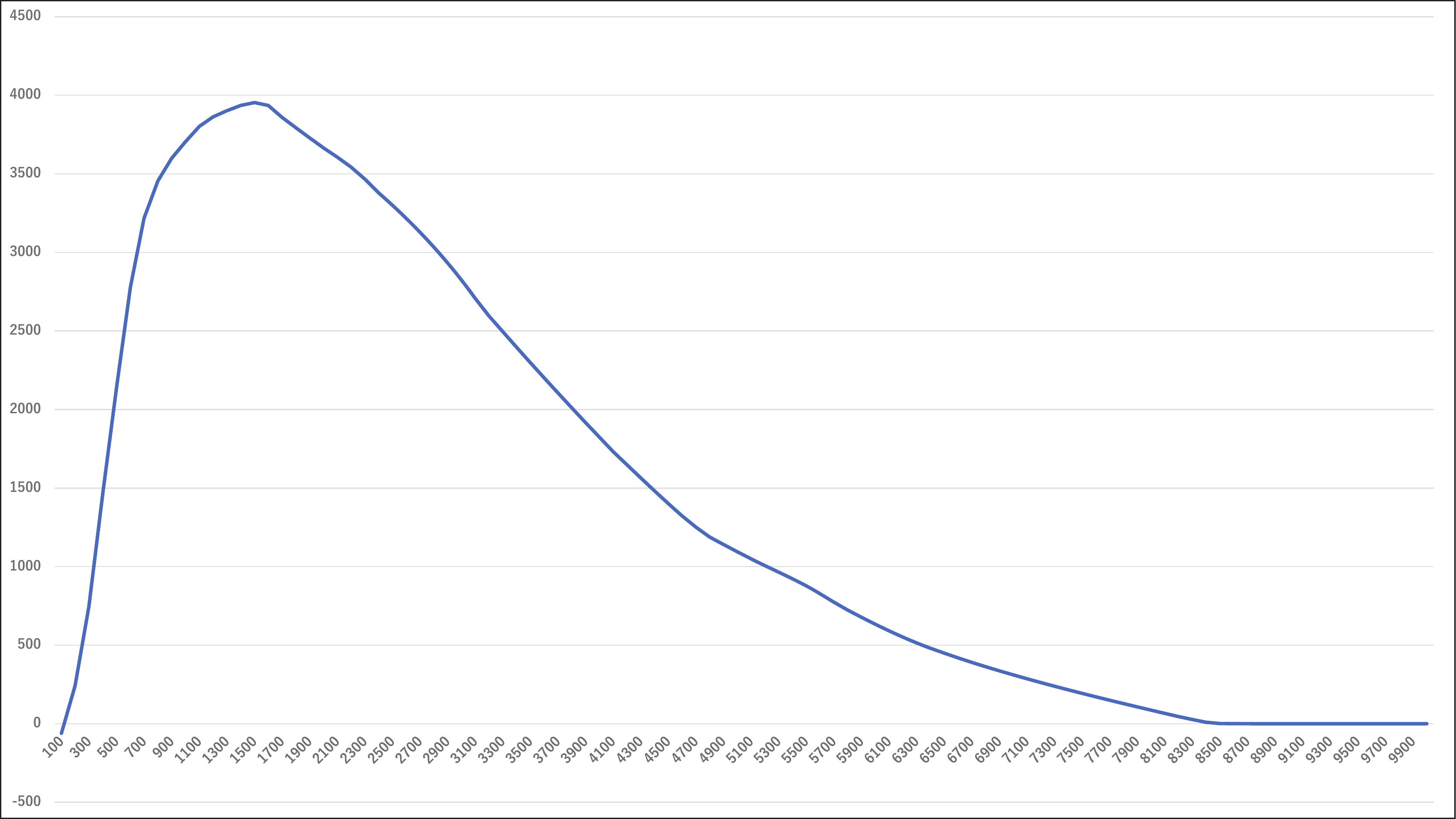}
\subcaption{chr12c}
\label{fig:chr12c}
\end{minipage} \\
%----------------------------------
\begin{minipage}{0.3\columnwidth}
\centering
\includegraphics[width = 50mm,pagebox=cropbox,clip]{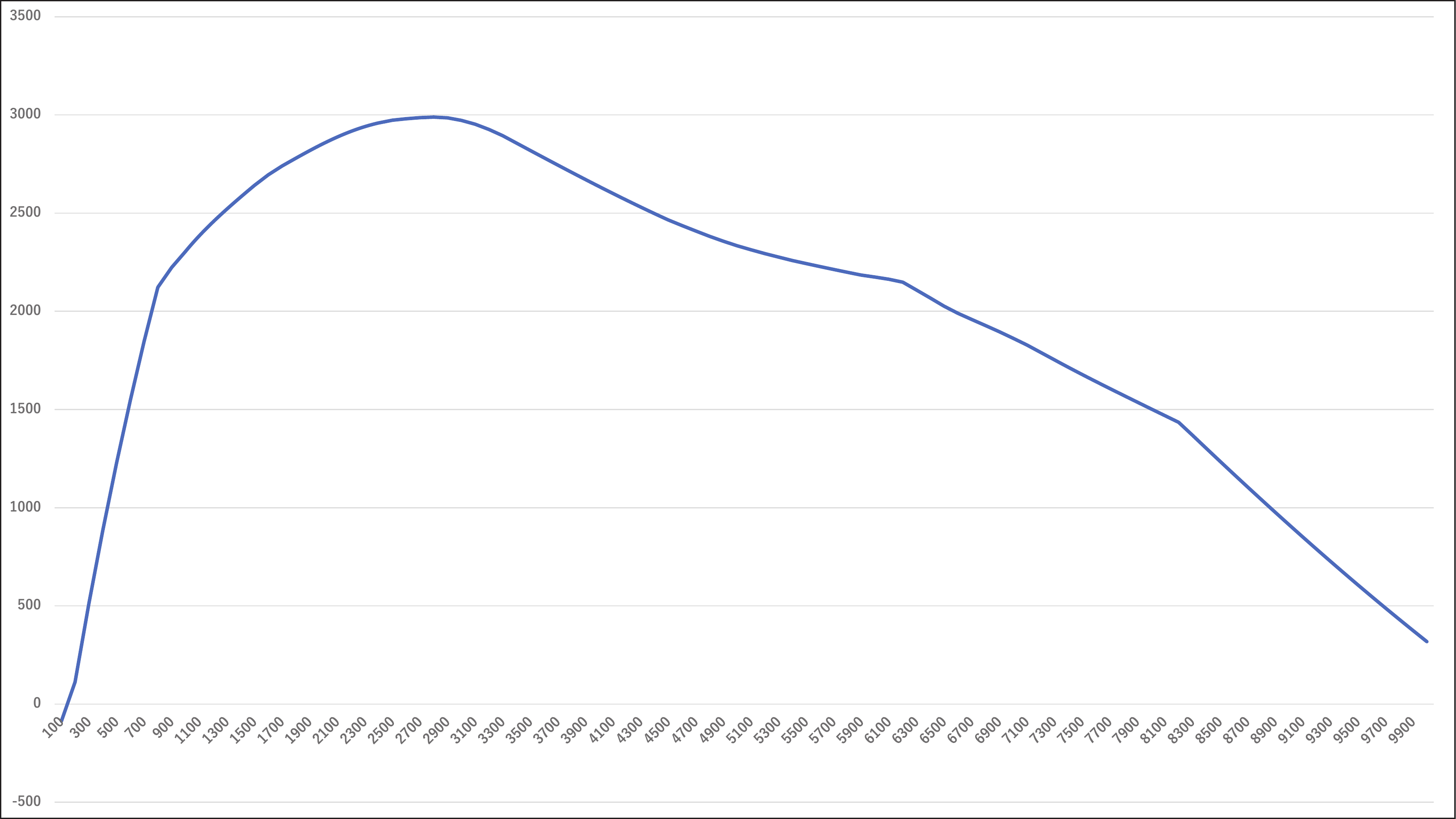}
\subcaption{chr15c}
\label{fig:chr15c}
\end{minipage} &
\begin{minipage}{0.3\columnwidth}
\centering
\includegraphics[width = 50mm,pagebox=cropbox,clip]{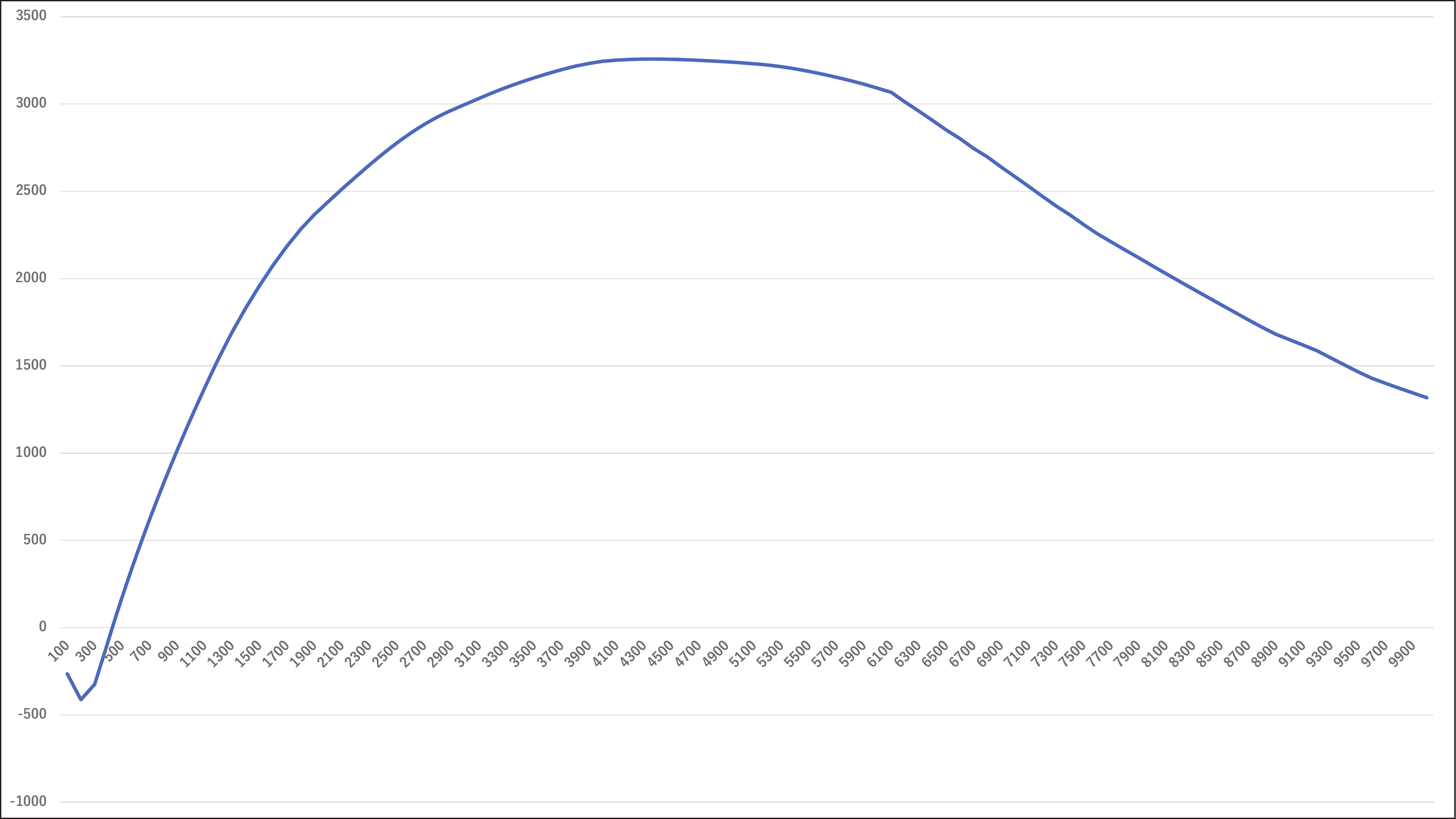}
\subcaption{chr20c}
\label{fig:chr20c}
\end{minipage} & \\
%----------------------------------
\end{tabular}
\caption{Results for   ``chr'' instances}
\label{fig:chr}
\end{figure}

\begin{figure}[H]
\begin{tabular}{ccc}
%----------------------------------
\begin{minipage}{0.3\columnwidth}
\centering
\includegraphics[width = 50mm,pagebox=cropbox,clip]{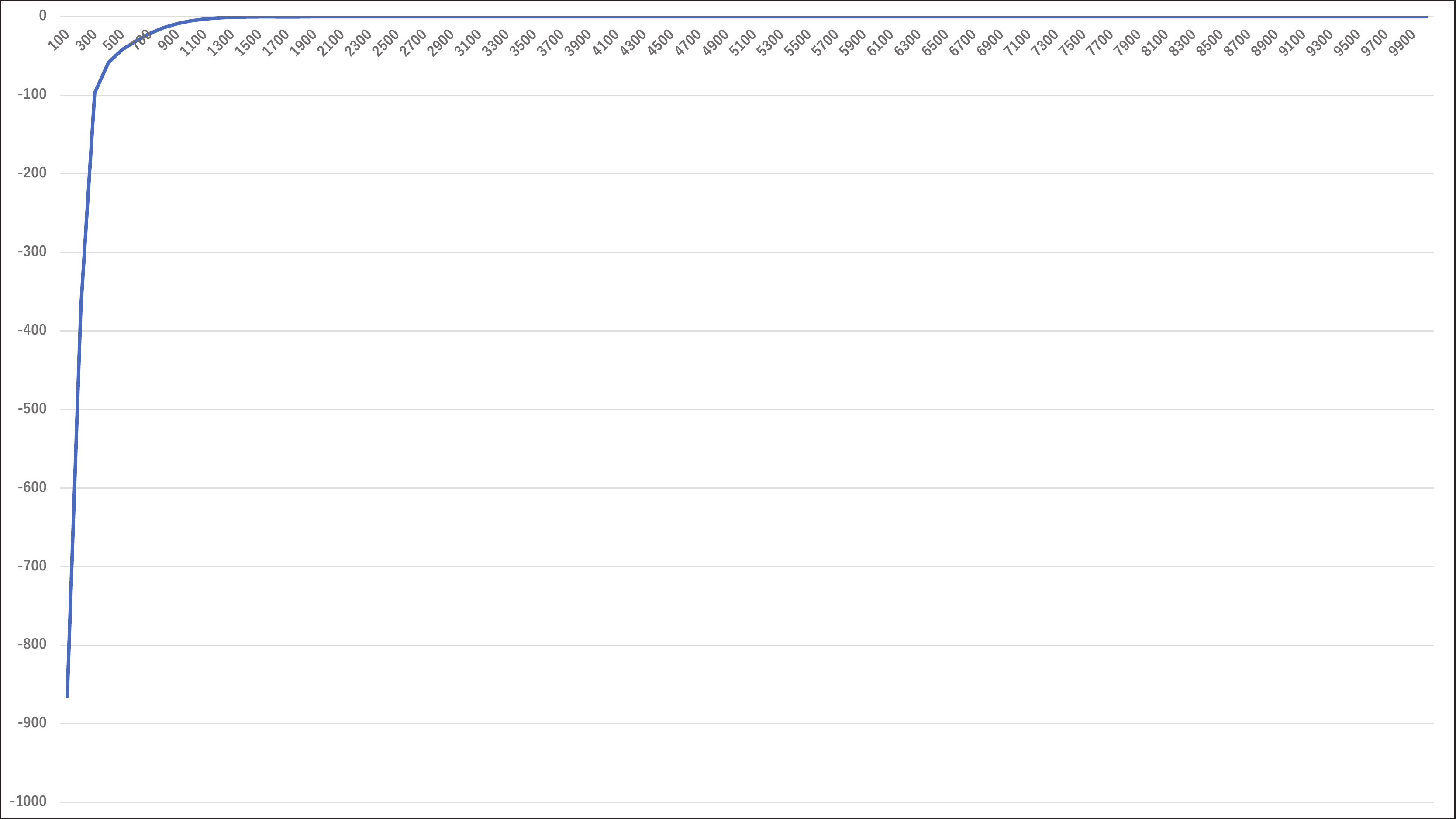}
\subcaption{Had12}
\label{fig:Had12}
\end{minipage} &
\begin{minipage}{0.3\columnwidth}
\centering
\includegraphics[width = 50mm,pagebox=cropbox,clip]{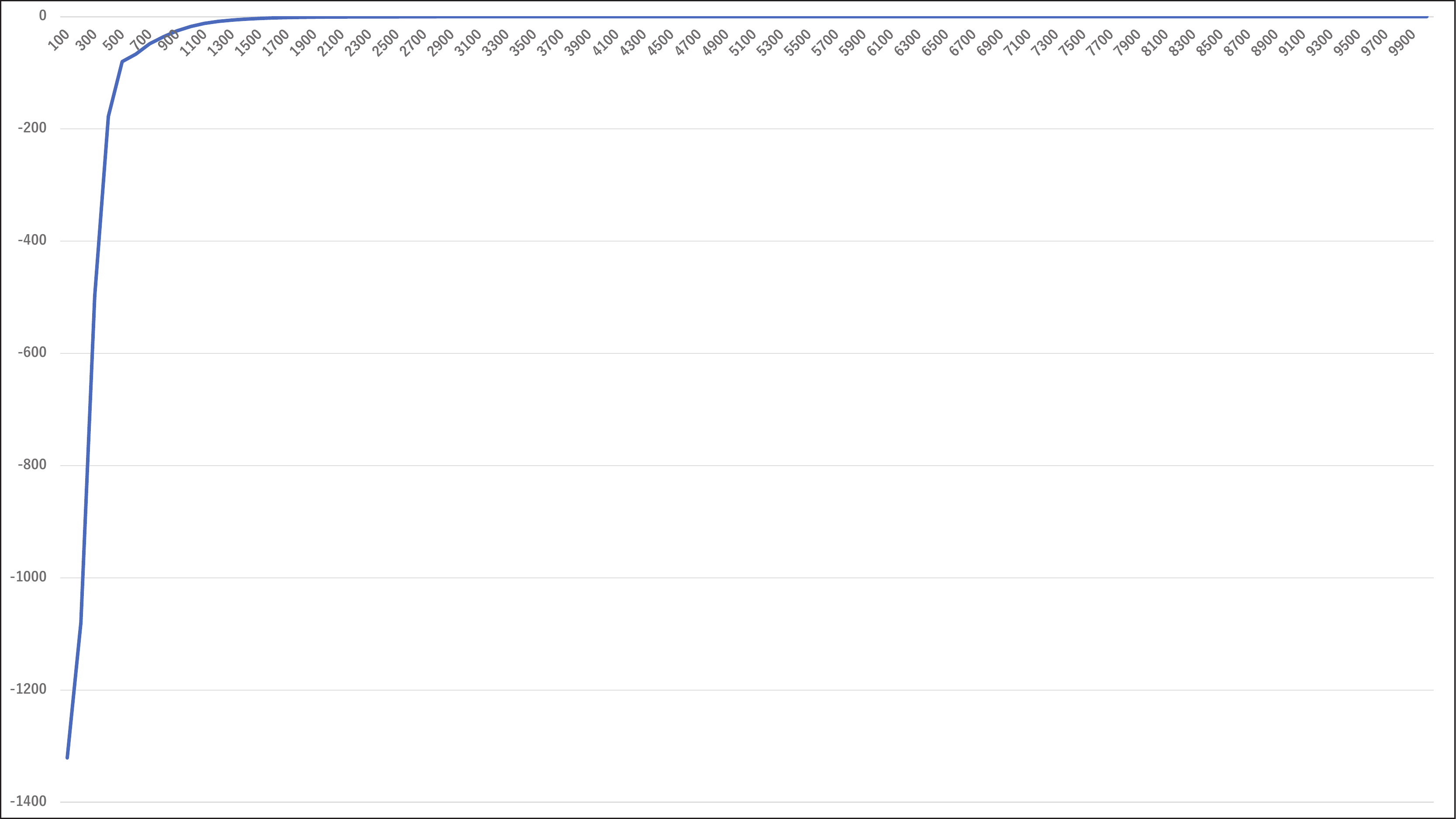}
\subcaption{Had14}
\label{fig:Had14}
\end{minipage} &
\begin{minipage}{0.3\columnwidth}
\centering
\includegraphics[width = 50mm,pagebox=cropbox,clip]{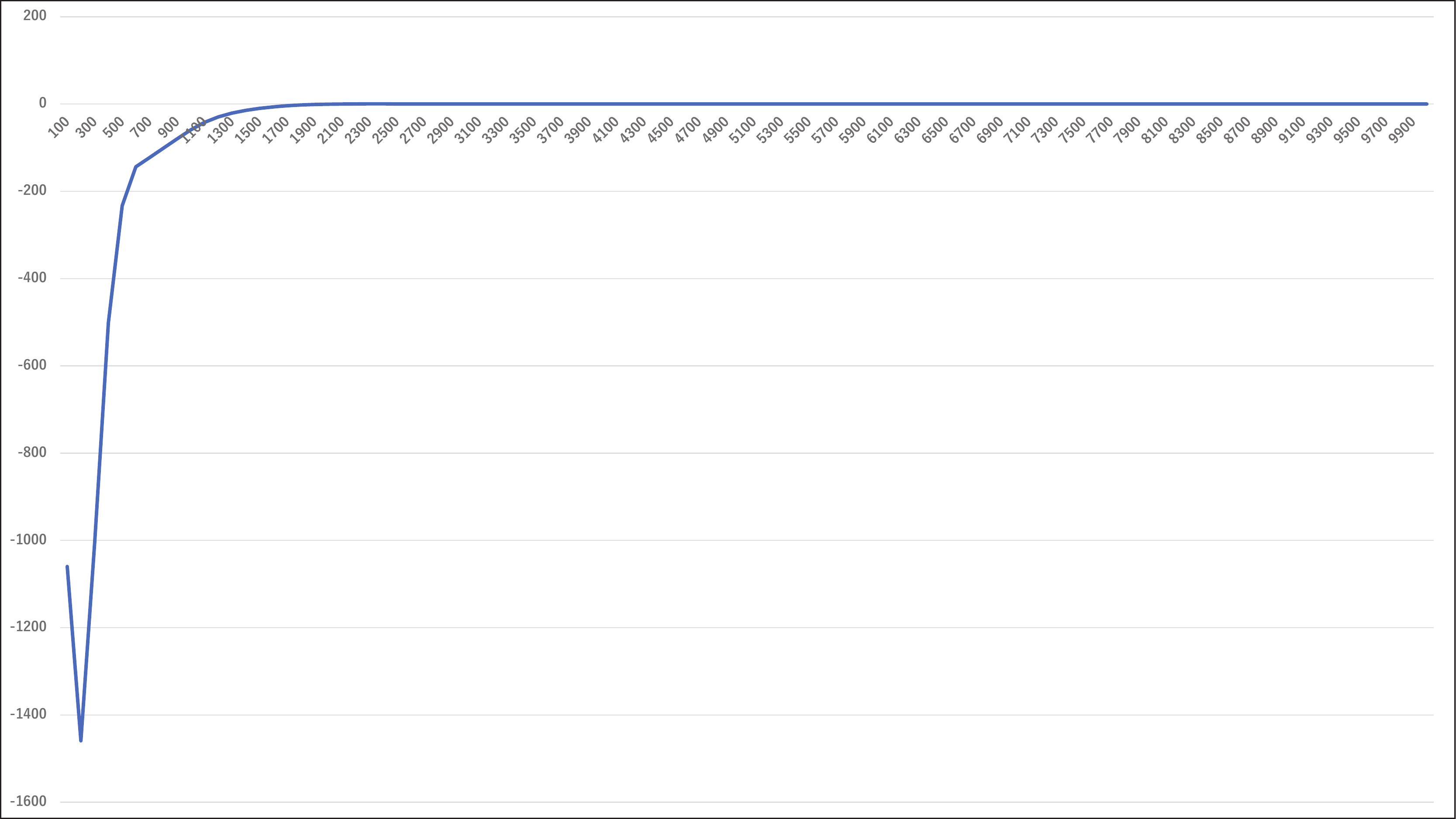}
\subcaption{Had16}
\label{fig:Had16}
\end{minipage} \\
%----------------------------------
\begin{minipage}{0.3\columnwidth}
\centering
\includegraphics[width = 50mm,pagebox=cropbox,clip]{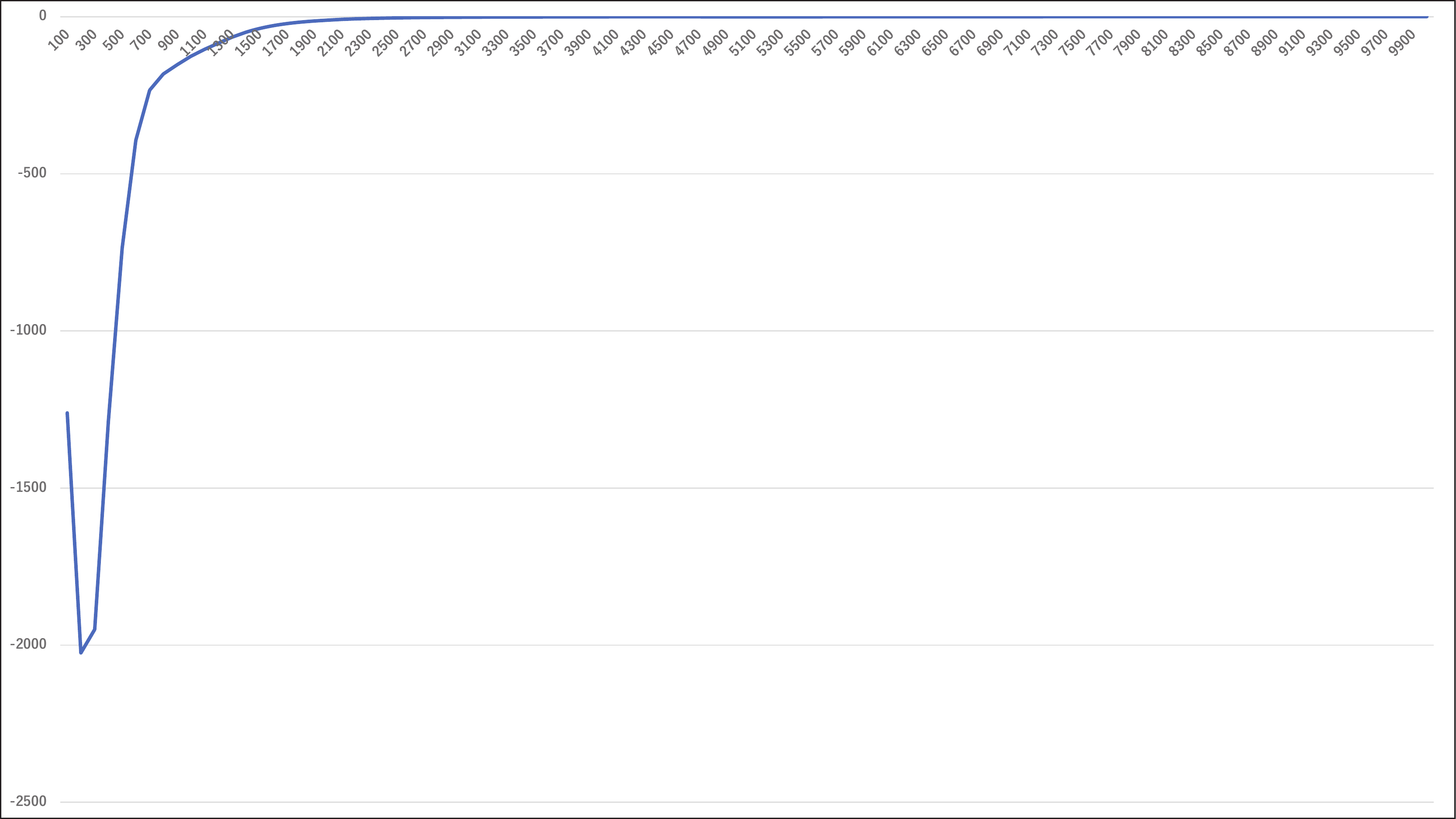}
\subcaption{Had18}
\label{fig:Had18}
\end{minipage} &
\begin{minipage}{0.3\columnwidth}
\centering
\includegraphics[width = 50mm,pagebox=cropbox,clip]{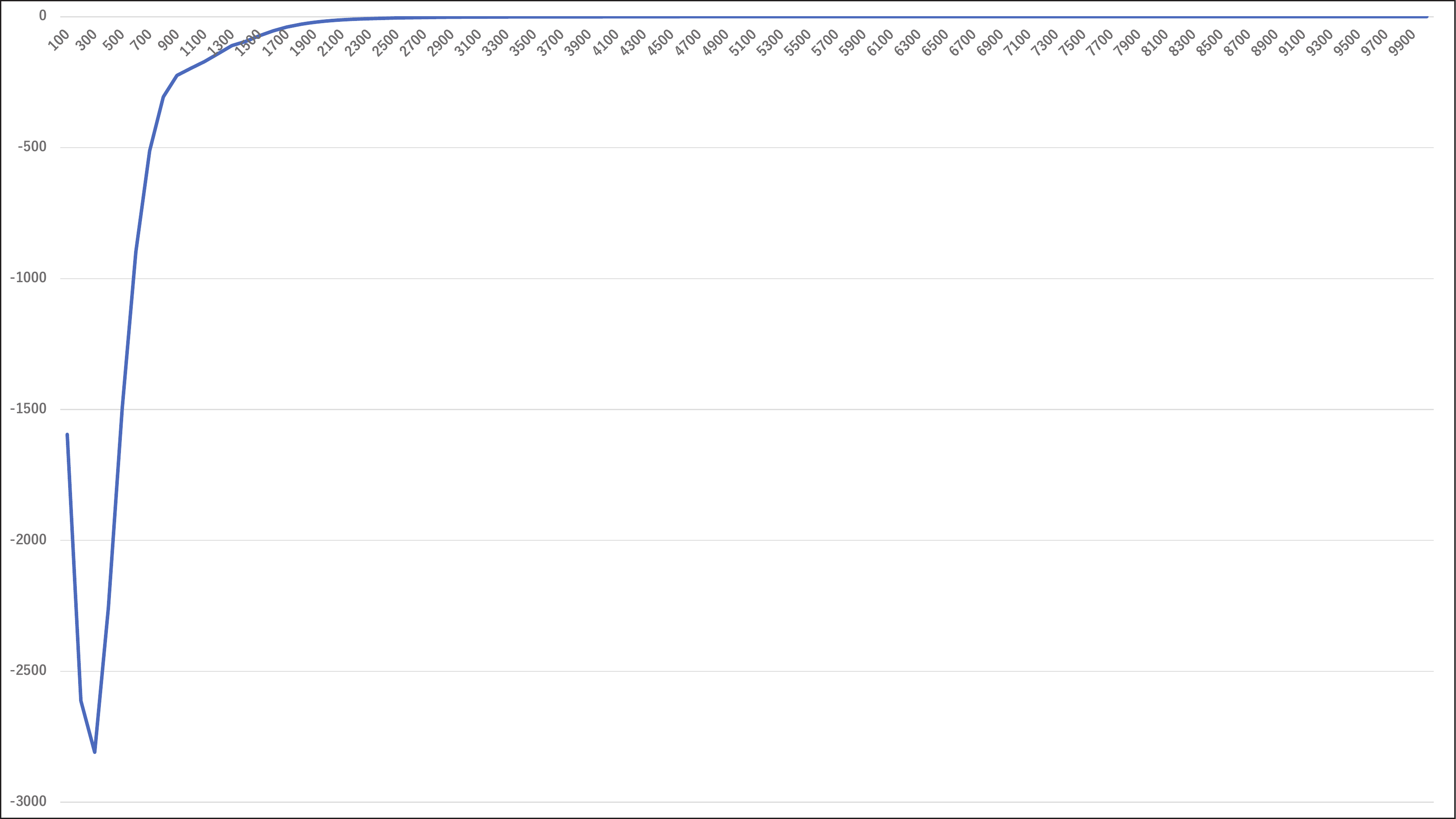}
\subcaption{Had20}
\label{fig:Had20}
\end{minipage} & \\
%----------------------------------
\end{tabular}
\caption{Results for  ``Had'' instances}
\label{fig:Had}
\end{figure}

\begin{figure}[H]
\begin{tabular}{ccc}
%----------------------------------
\begin{minipage}{0.3\columnwidth}
\centering
\includegraphics[width = 50mm,pagebox=cropbox,clip]{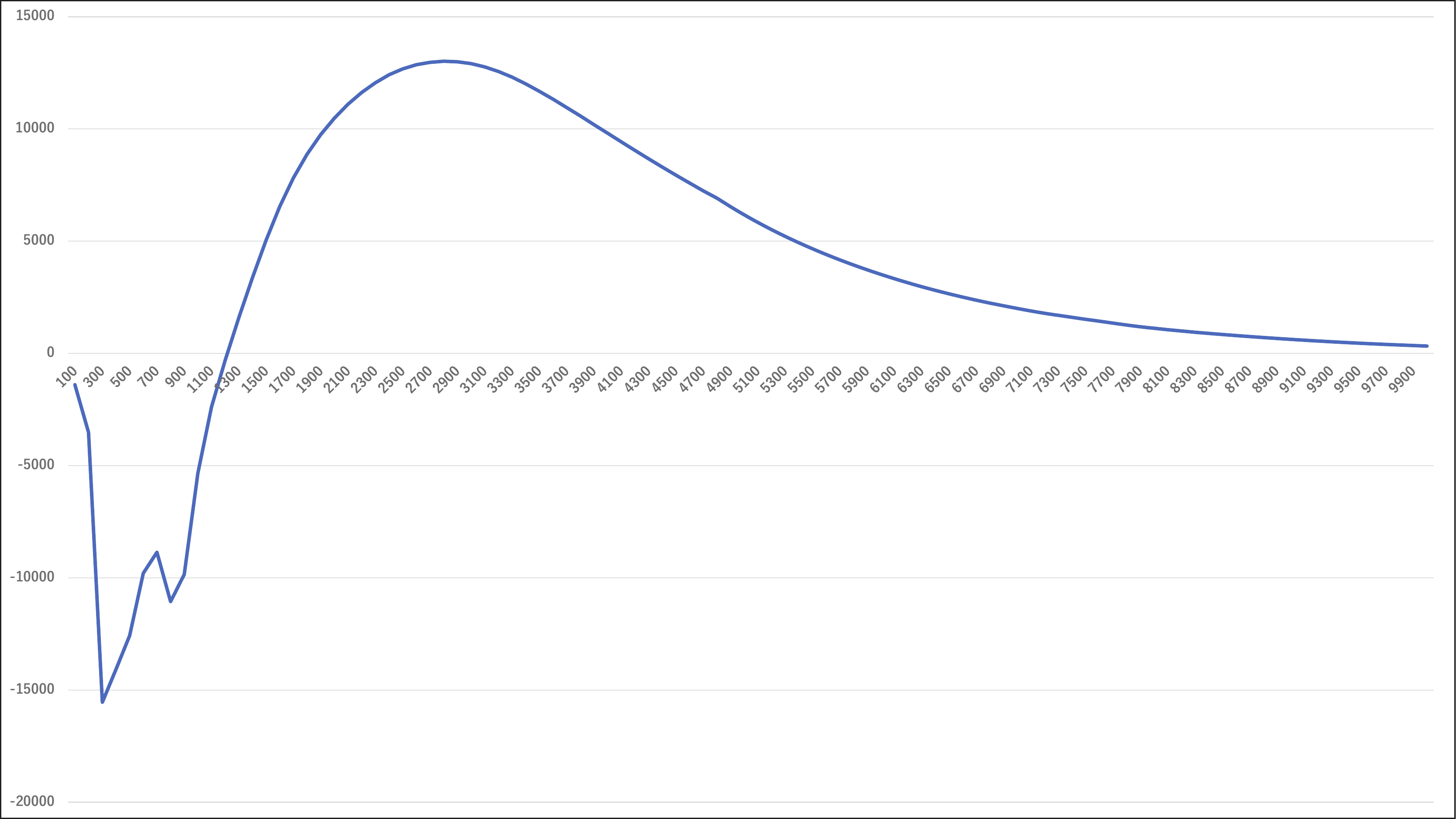}
\subcaption{Kra30a}
\label{fig:Kra30a}
\end{minipage} & 
\begin{minipage}{0.3\columnwidth}
\centering
\includegraphics[width = 50mm,pagebox=cropbox,clip]{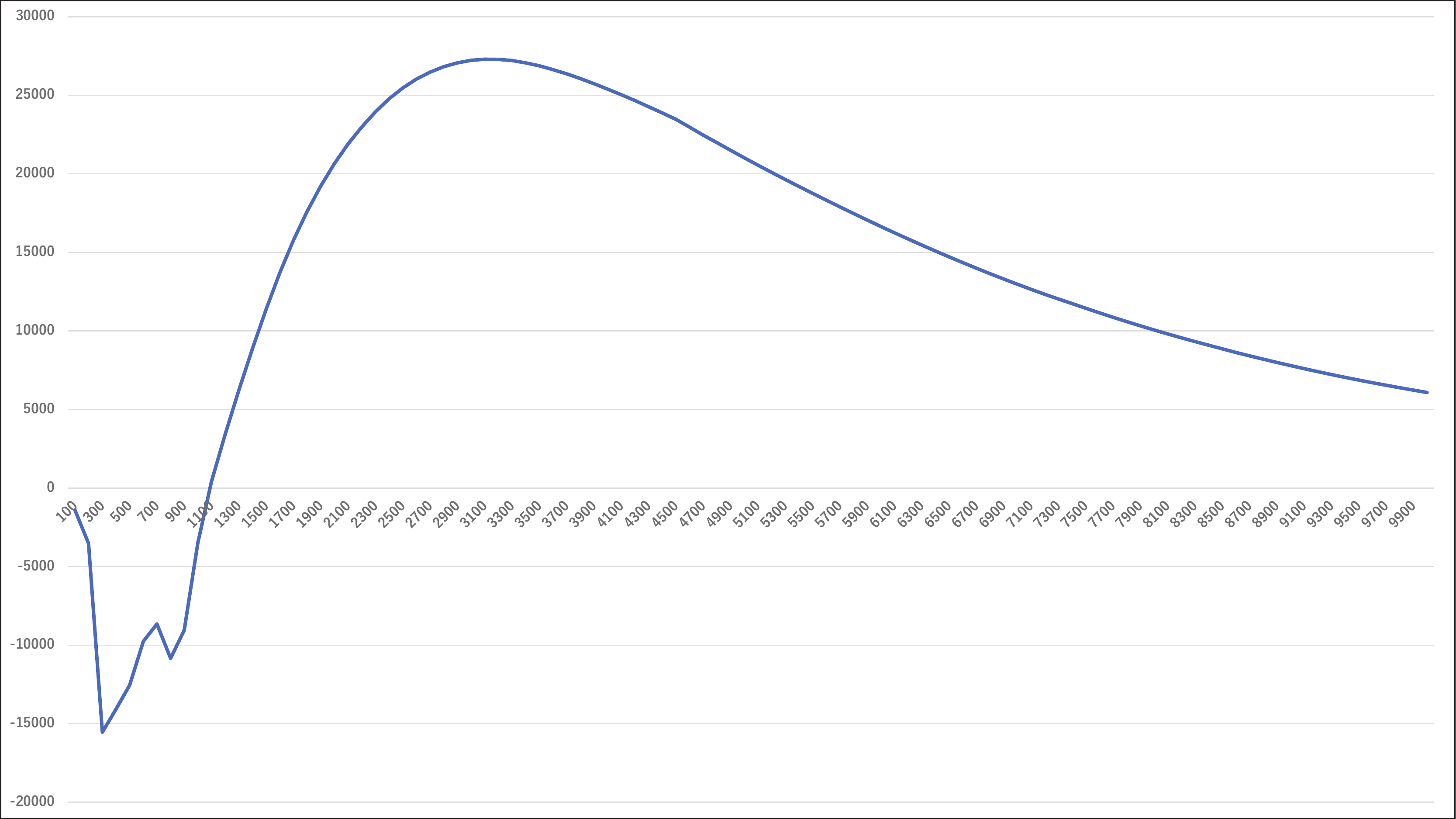}
\subcaption{Kra30b}
\label{fig:Kra30b}
\end{minipage}  &
\begin{minipage}{0.3\columnwidth}
\centering
\includegraphics[width = 50mm,pagebox=cropbox,clip]{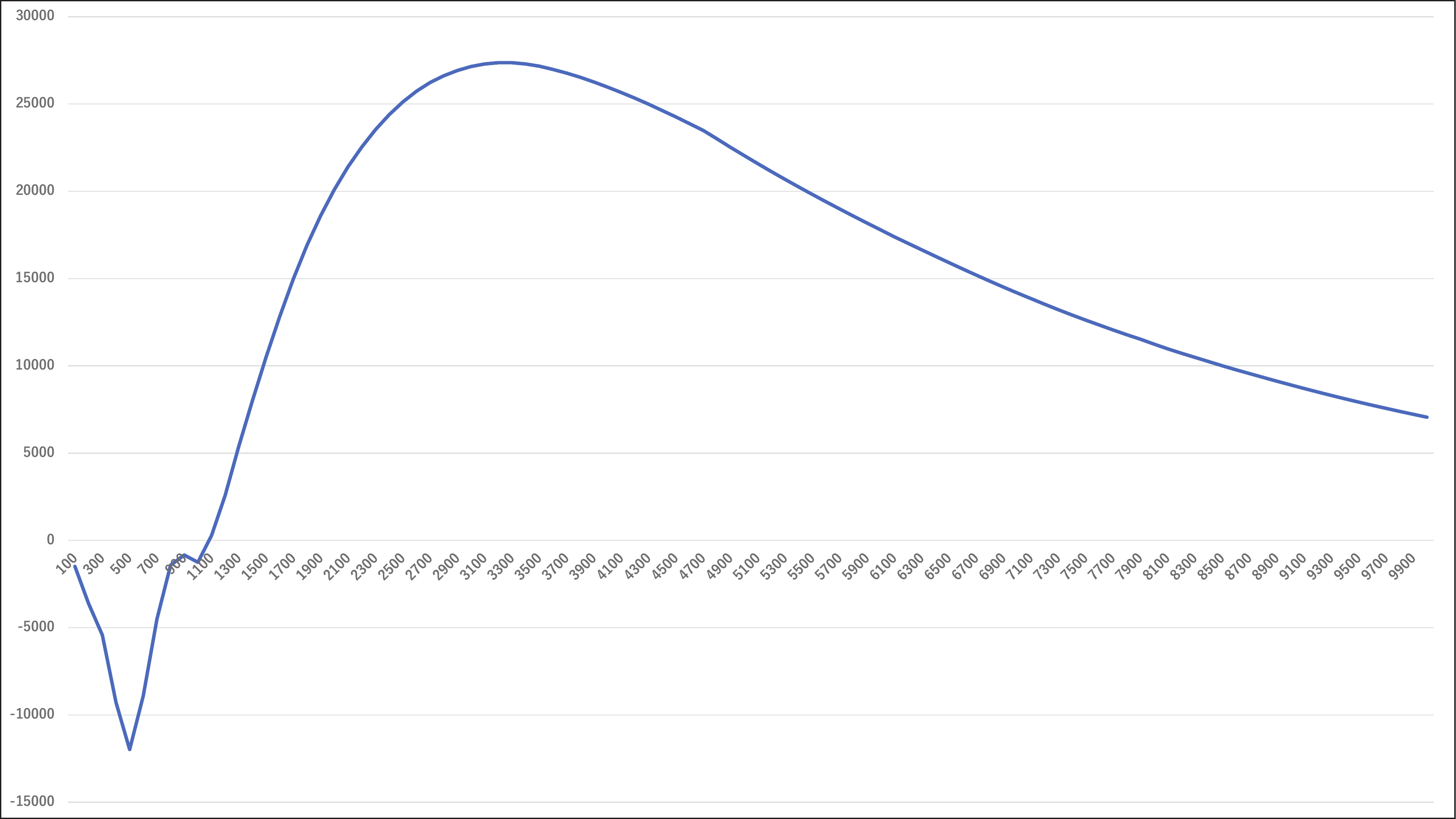}
\subcaption{Kra32}
\label{fig:Kra32}
\end{minipage} 
%----------------------------------
\end{tabular}
\caption{Results for  ``Kra'' instances}
\label{fig:Kra}
\end{figure}

\begin{figure}[H]
\begin{tabular}{ccc}
%----------------------------------
\begin{minipage}{0.3\columnwidth}
\includegraphics[width = 50mm,pagebox=cropbox,clip]{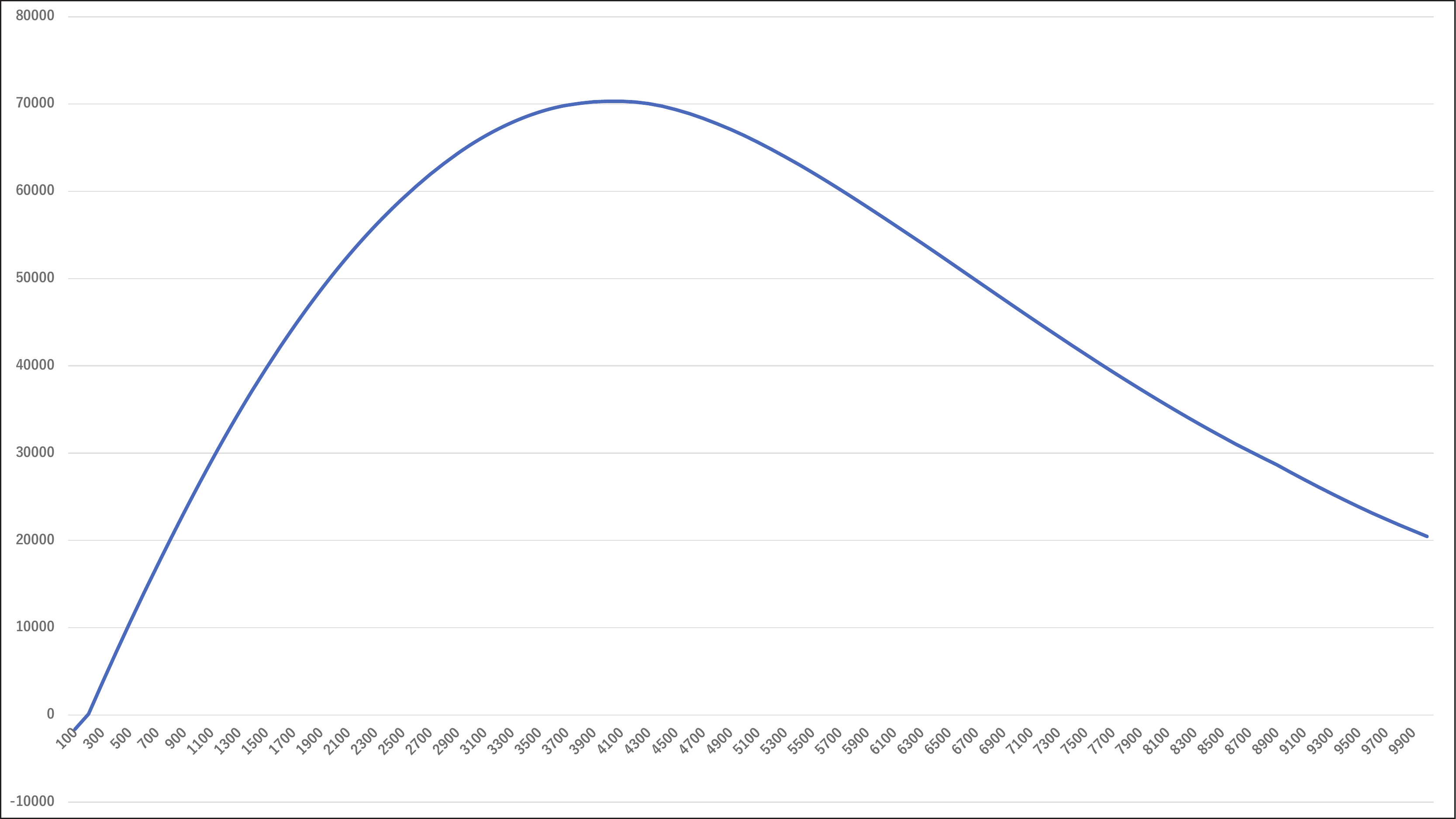}
\subcaption{Rou12}
\label{fig:Rou12}
\end{minipage} &
\begin{minipage}{0.3\columnwidth}
\centering
\includegraphics[width = 50mm,pagebox=cropbox,clip]{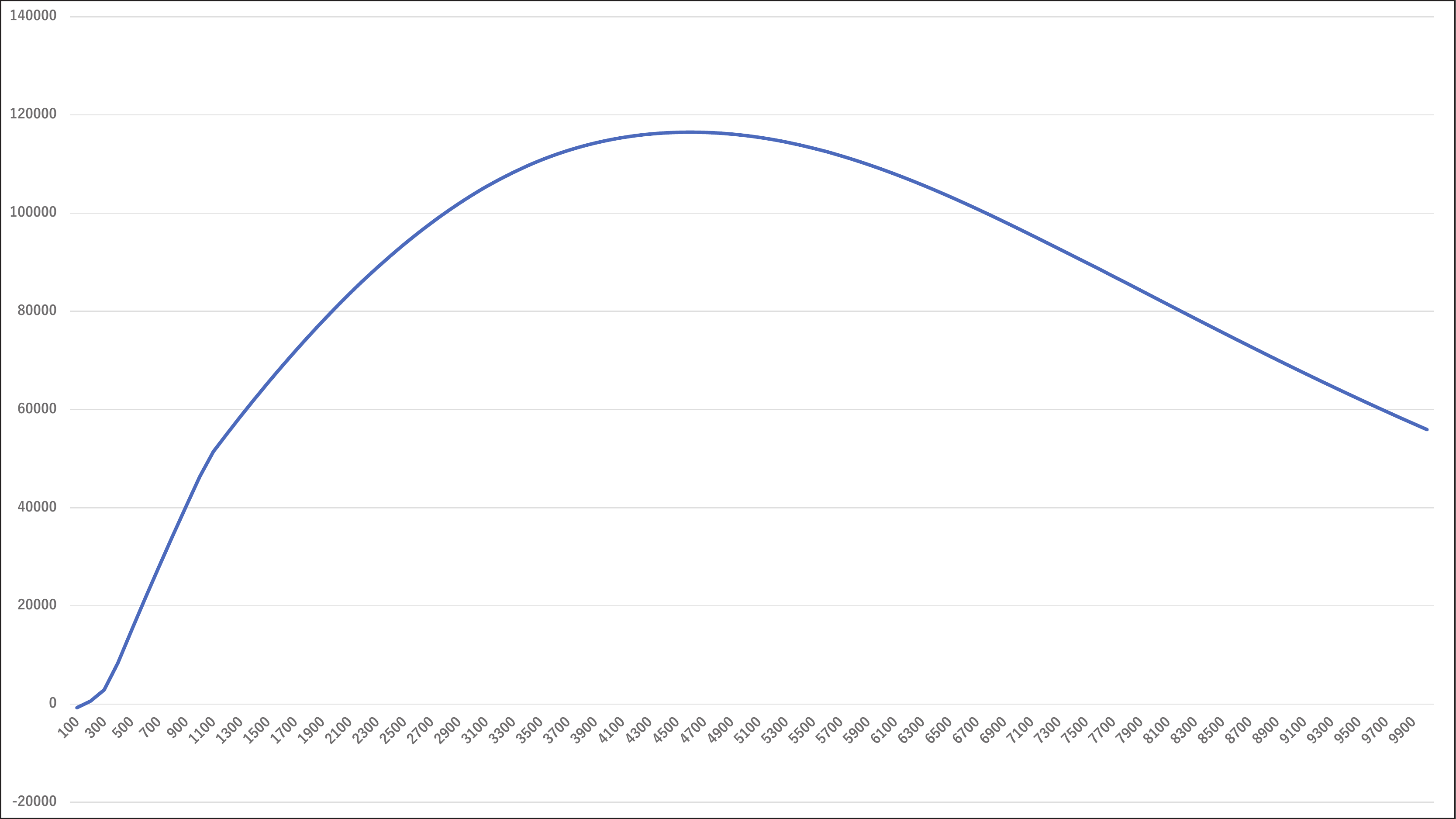}
\subcaption{Rou15}
\label{fig:Rou15}
\end{minipage} &
\begin{minipage}{0.3\columnwidth}
\centering
\includegraphics[width = 50mm,pagebox=cropbox,clip]{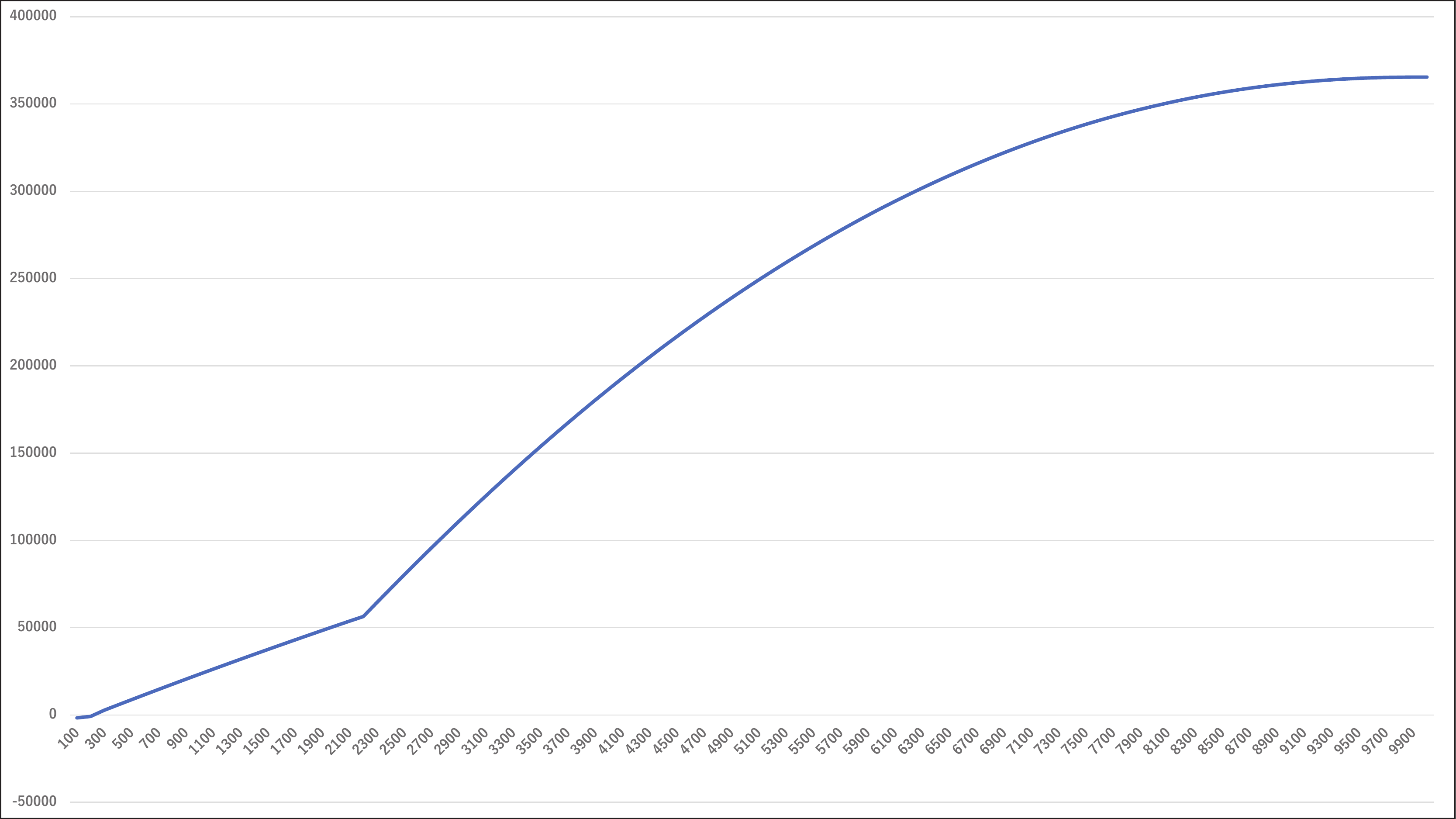}
\subcaption{Rou20}
\label{fig:Rou20}
\end{minipage}
%----------------------------------
\end{tabular}
\caption{Results for   ``Rou'' instances}
\label{fig:Rou}
\end{figure}

\begin{figure}[H]
\begin{tabular}{ccc}
%----------------------------------
\begin{minipage}{0.3\columnwidth}
\centering
\includegraphics[width = 50mm,pagebox=cropbox,clip]{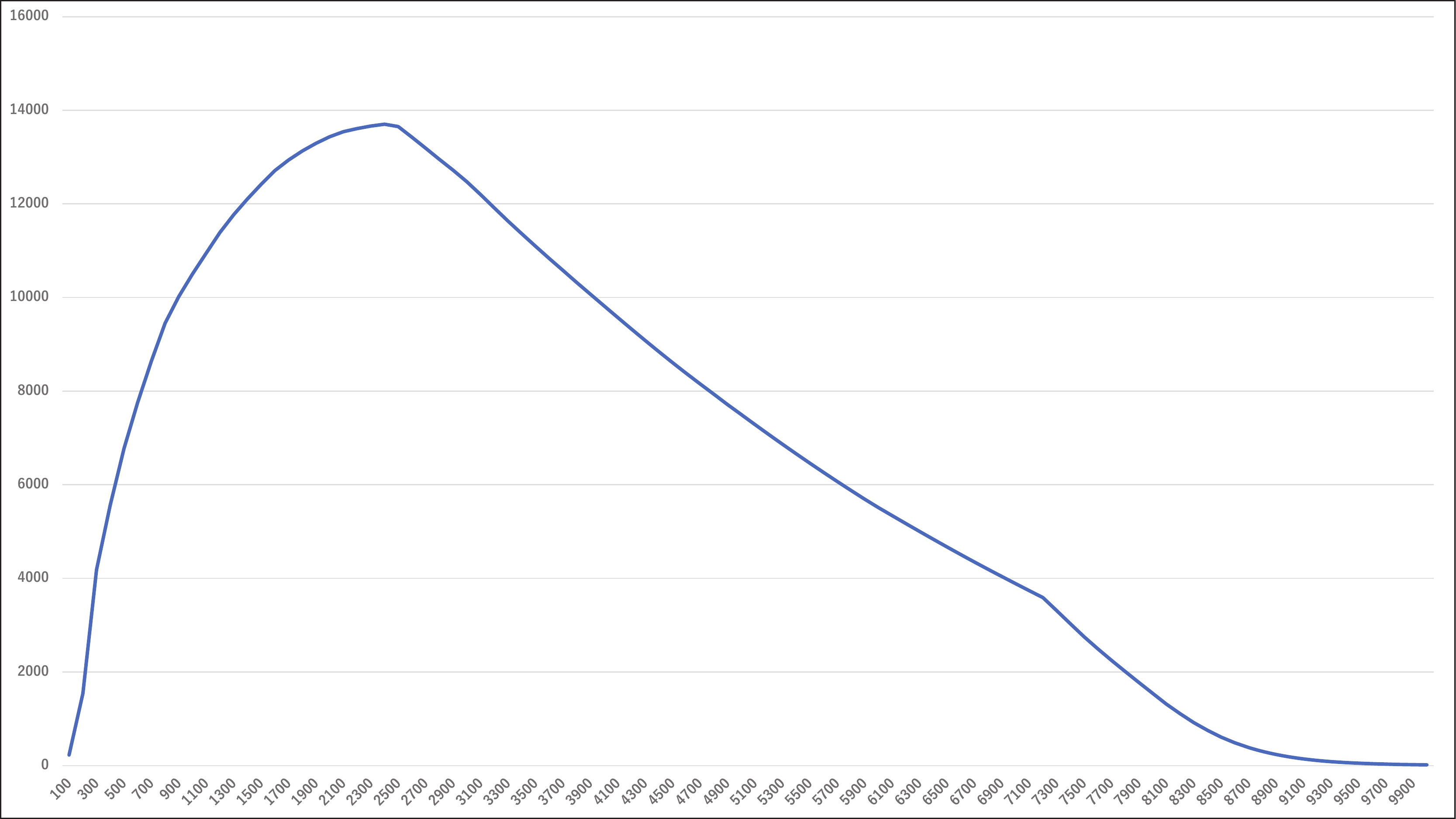}
\subcaption{Scr12}
\label{fig:Scr12}
\end{minipage} &
\begin{minipage}{0.3\columnwidth}
\centering
\includegraphics[width = 50mm,pagebox=cropbox,clip]{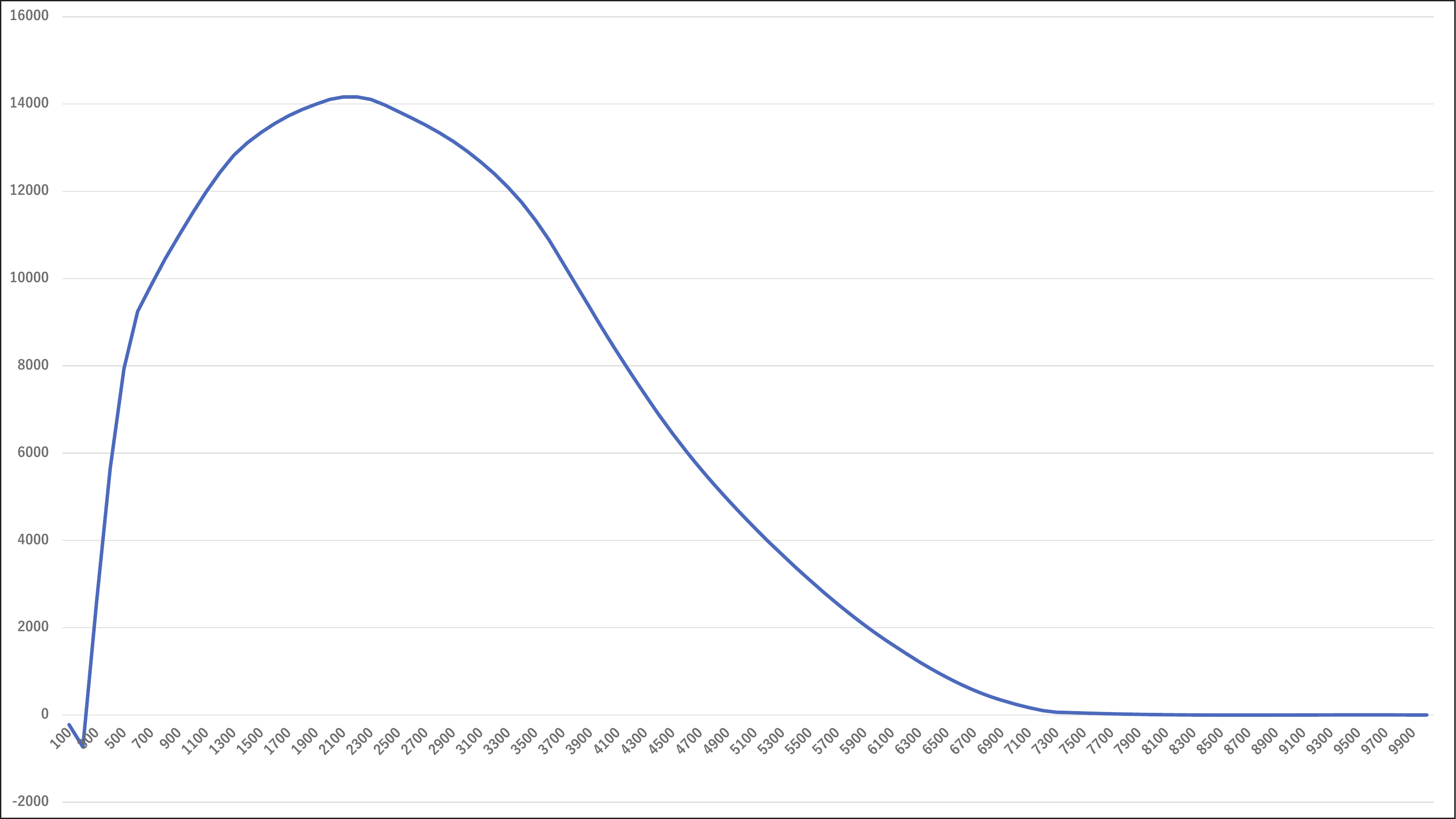}
\subcaption{Scr15}
\label{fig:Scr15}
\end{minipage} &
\begin{minipage}{0.3\columnwidth}
\centering
\includegraphics[width = 50mm,pagebox=cropbox,clip]{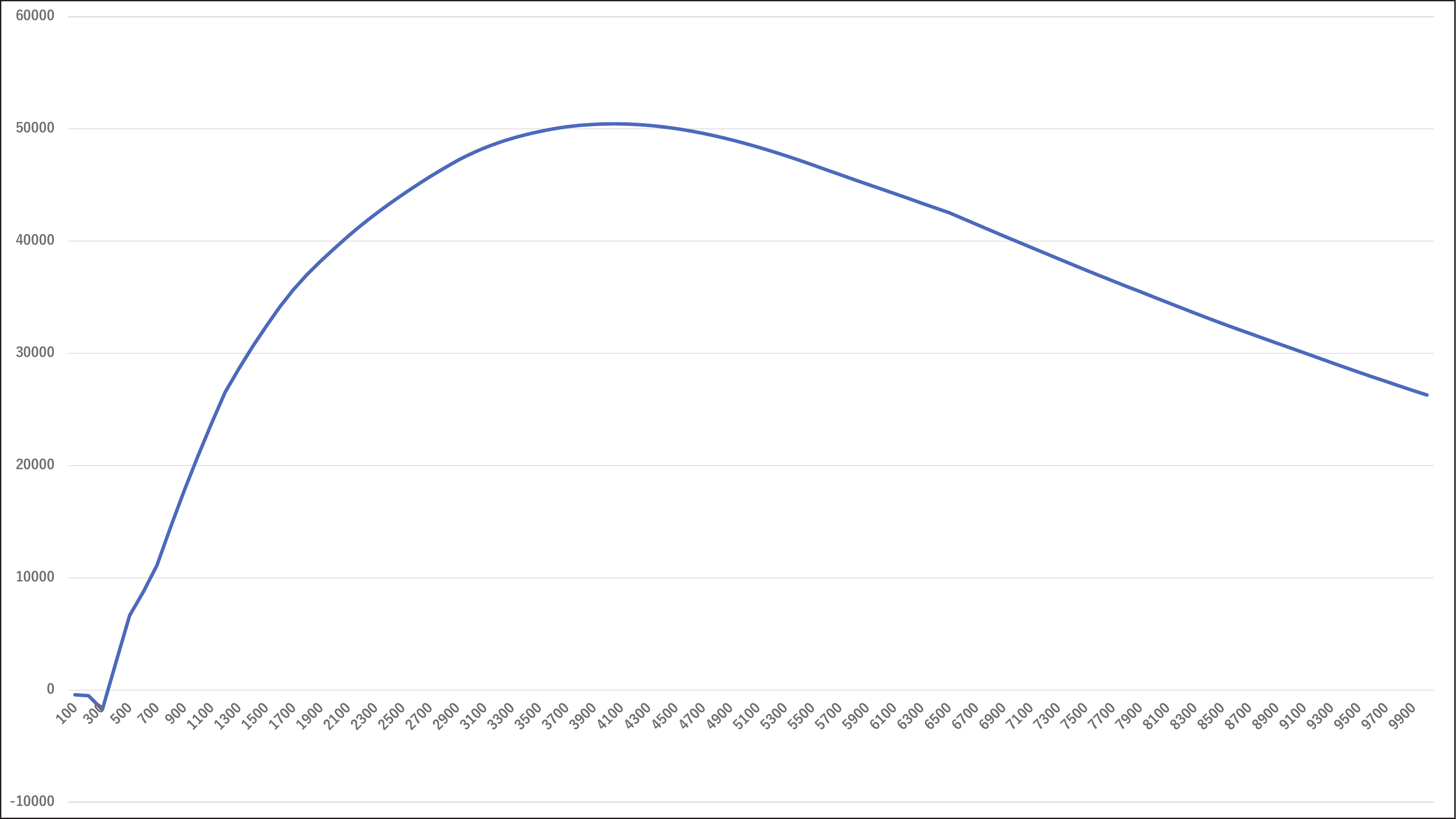}
\subcaption{Scr20}
\label{fig:Scr20}
\end{minipage}
%----------------------------------
\end{tabular}
\caption{Results for   ``Scr'' instances}
\label{fig:Scr}
\end{figure}

\begin{figure}[H]
\begin{tabular}{ccc}
%----------------------------------
\begin{minipage}{0.3\columnwidth}
\centering
\includegraphics[width = 50mm,pagebox=cropbox,clip]{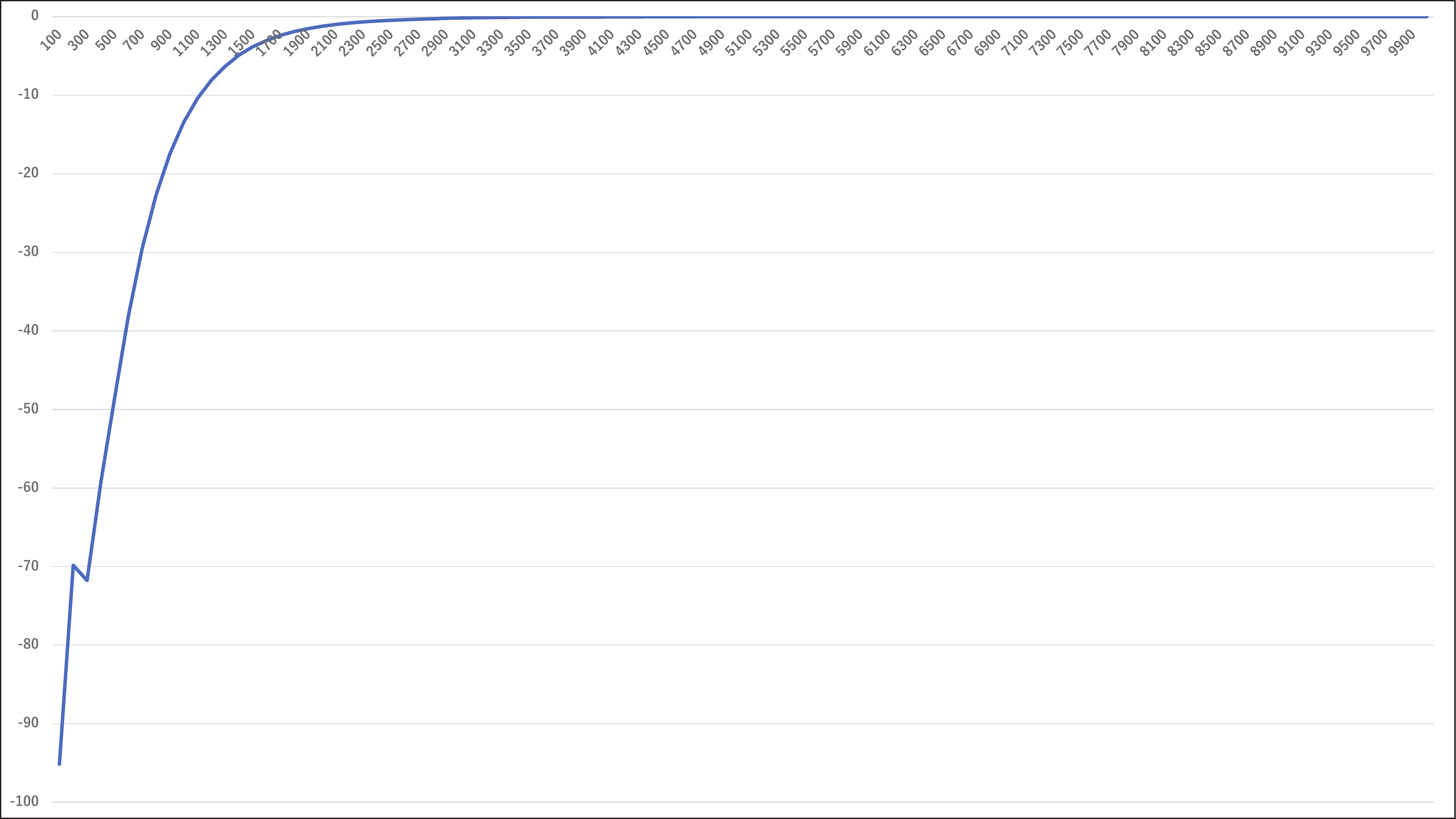}
\subcaption{Nug12}
\label{fig:Nug12}
\end{minipage} &
\begin{minipage}{0.3\columnwidth}
\centering
\includegraphics[width = 50mm,pagebox=cropbox,clip]{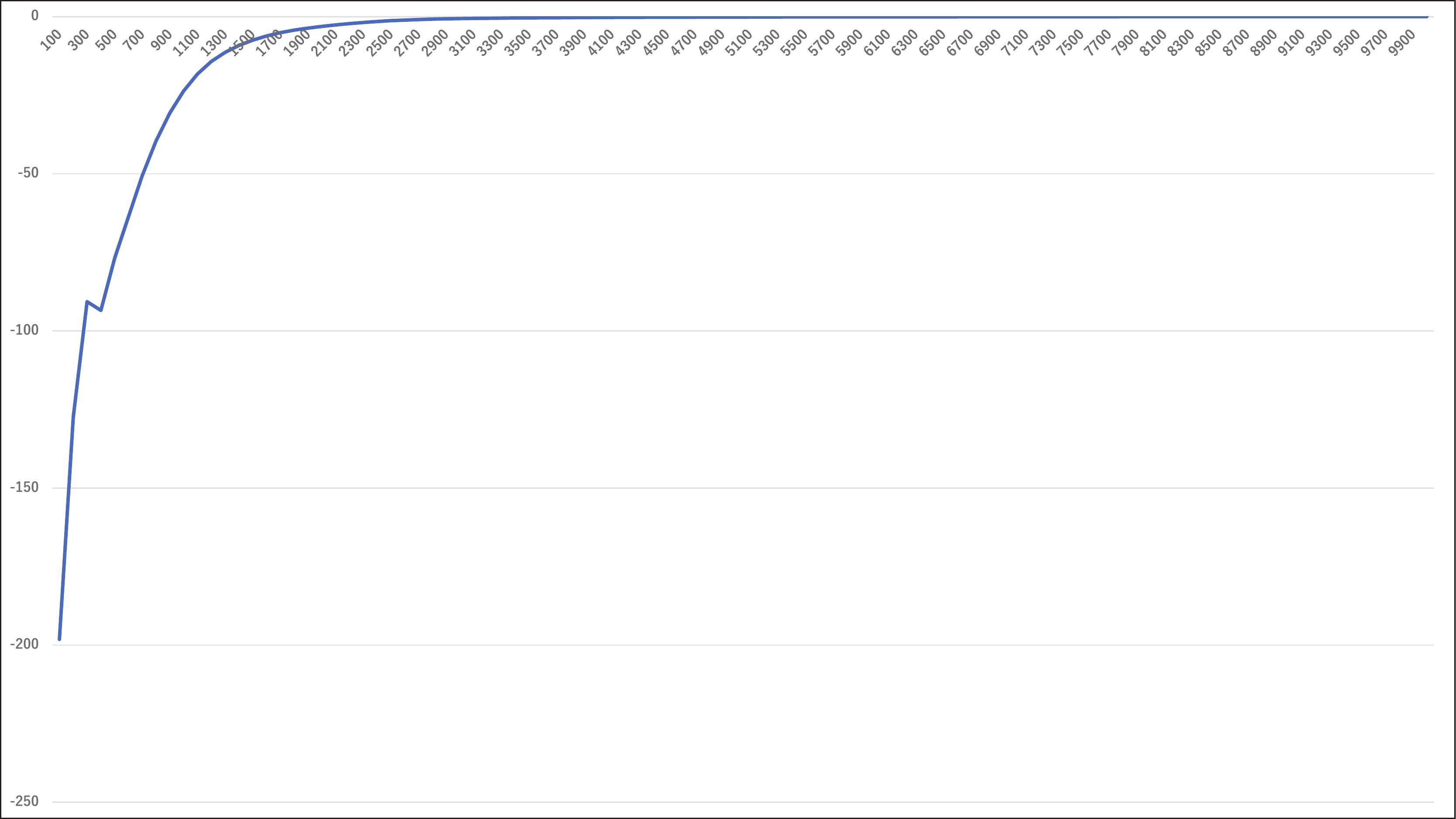}
\subcaption{Nug14}
\label{fig:Nug14}
\end{minipage} &
\begin{minipage}{0.3\columnwidth}
\centering
\includegraphics[width = 50mm,pagebox=cropbox,clip]{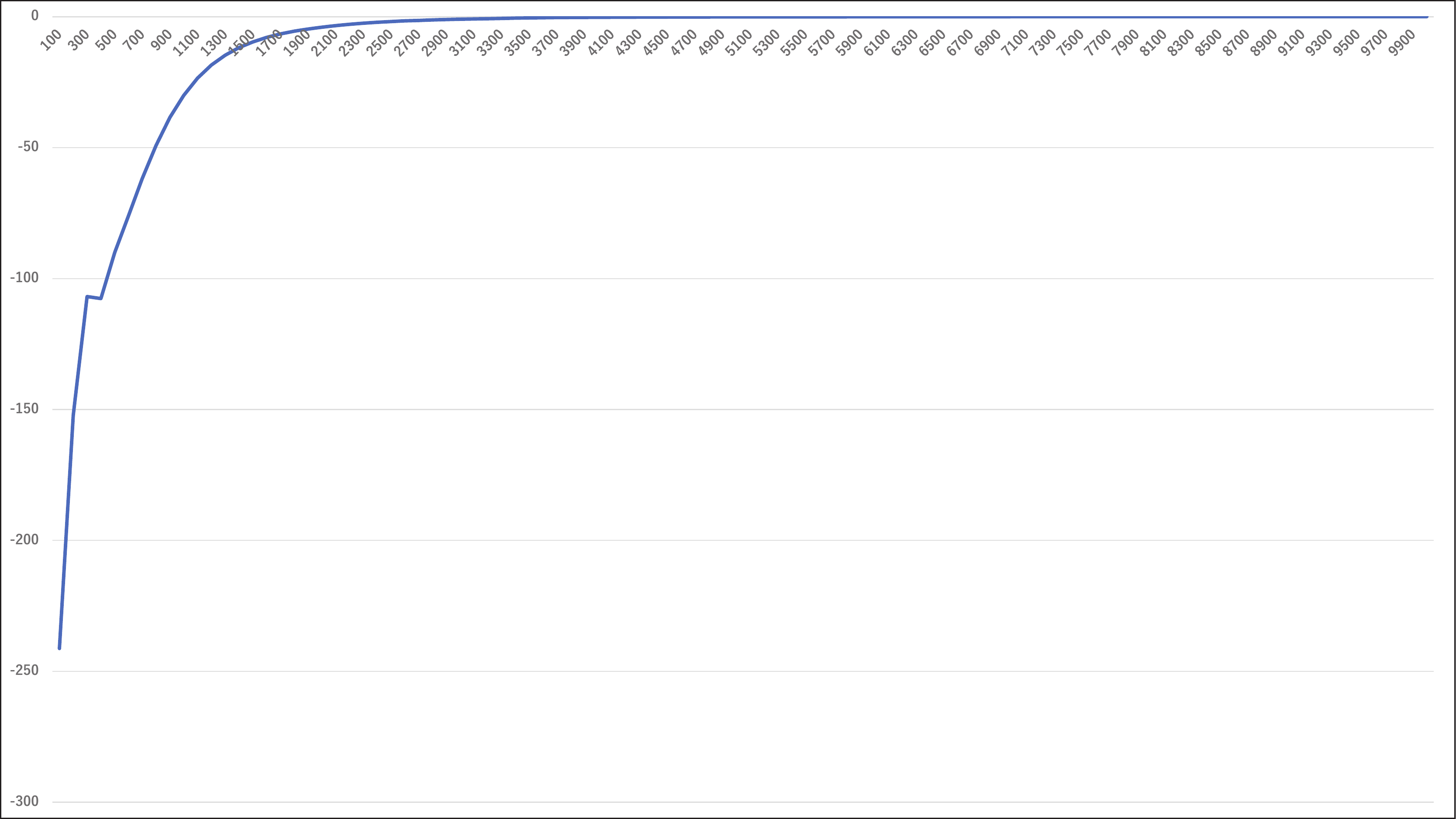}
\subcaption{Nug15}
\label{fig:Nug15}
\end{minipage} \\
%----------------------------------
\begin{minipage}{0.3\columnwidth}
\centering
\includegraphics[width = 50mm,pagebox=cropbox,clip]{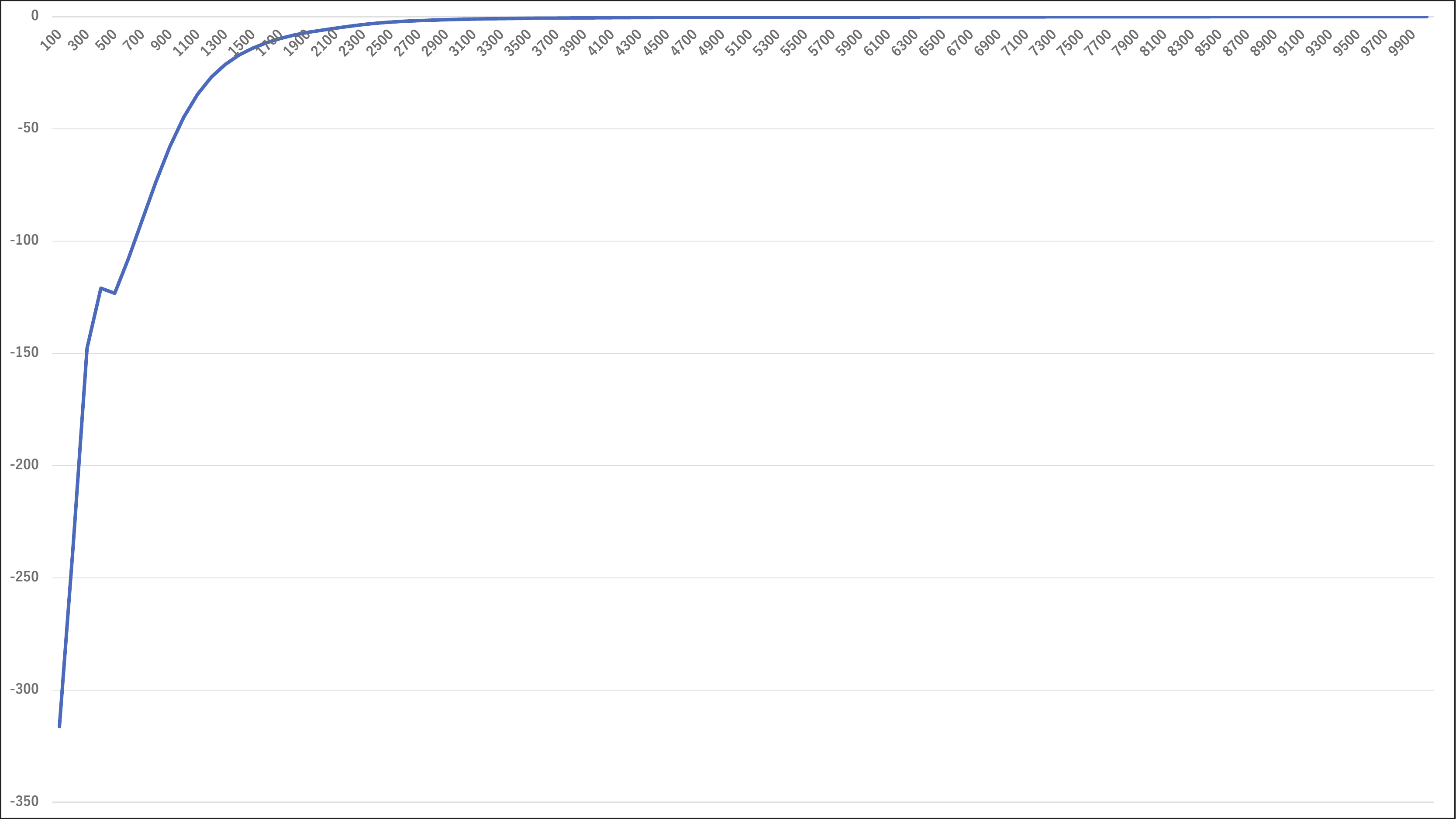}
\subcaption{Nug16a}
\label{fig:Nug16a}
\end{minipage} &
\begin{minipage}{0.3\columnwidth}
\centering
\includegraphics[width = 50mm,pagebox=cropbox,clip]{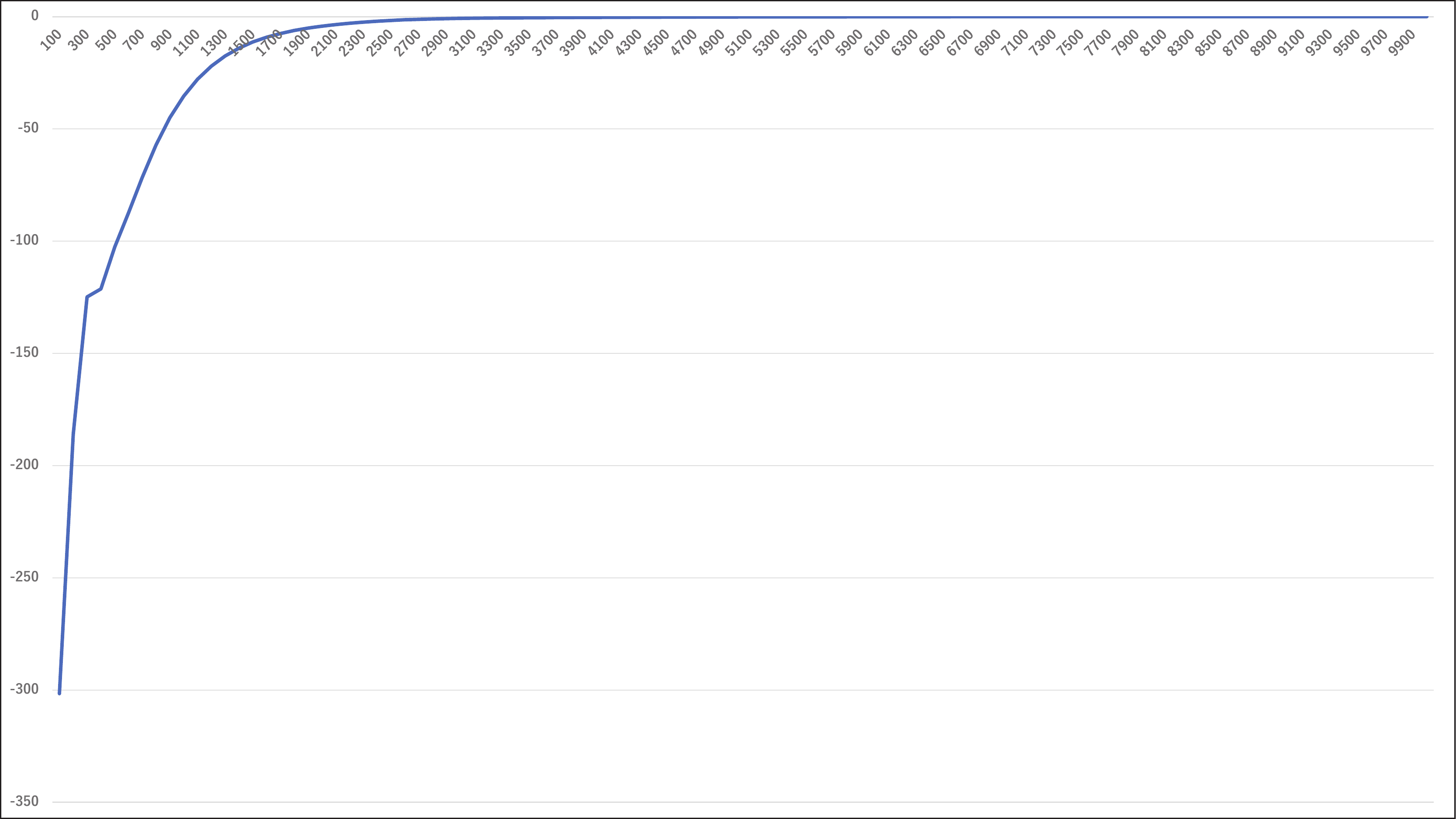}
\subcaption{Nug16b}
\label{fig:Nug16b}
\end{minipage} &
\begin{minipage}{0.3\columnwidth}
\centering
\includegraphics[width = 50mm,pagebox=cropbox,clip]{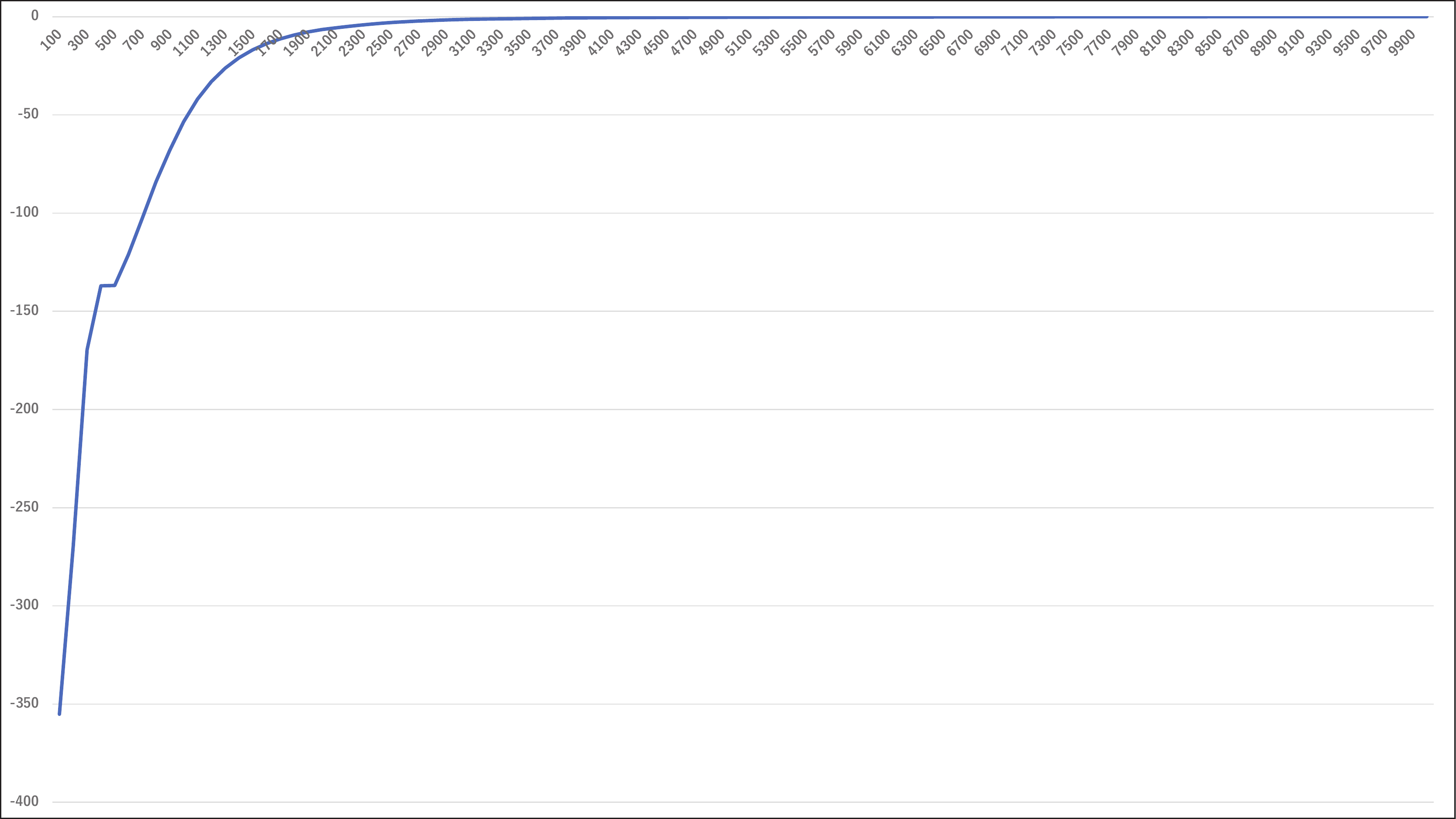}
\subcaption{Nug17}
\label{fig:Nug17}
\end{minipage} \\
%----------------------------------
\begin{minipage}{0.3\columnwidth}
\centering
\includegraphics[width = 50mm,pagebox=cropbox,clip]{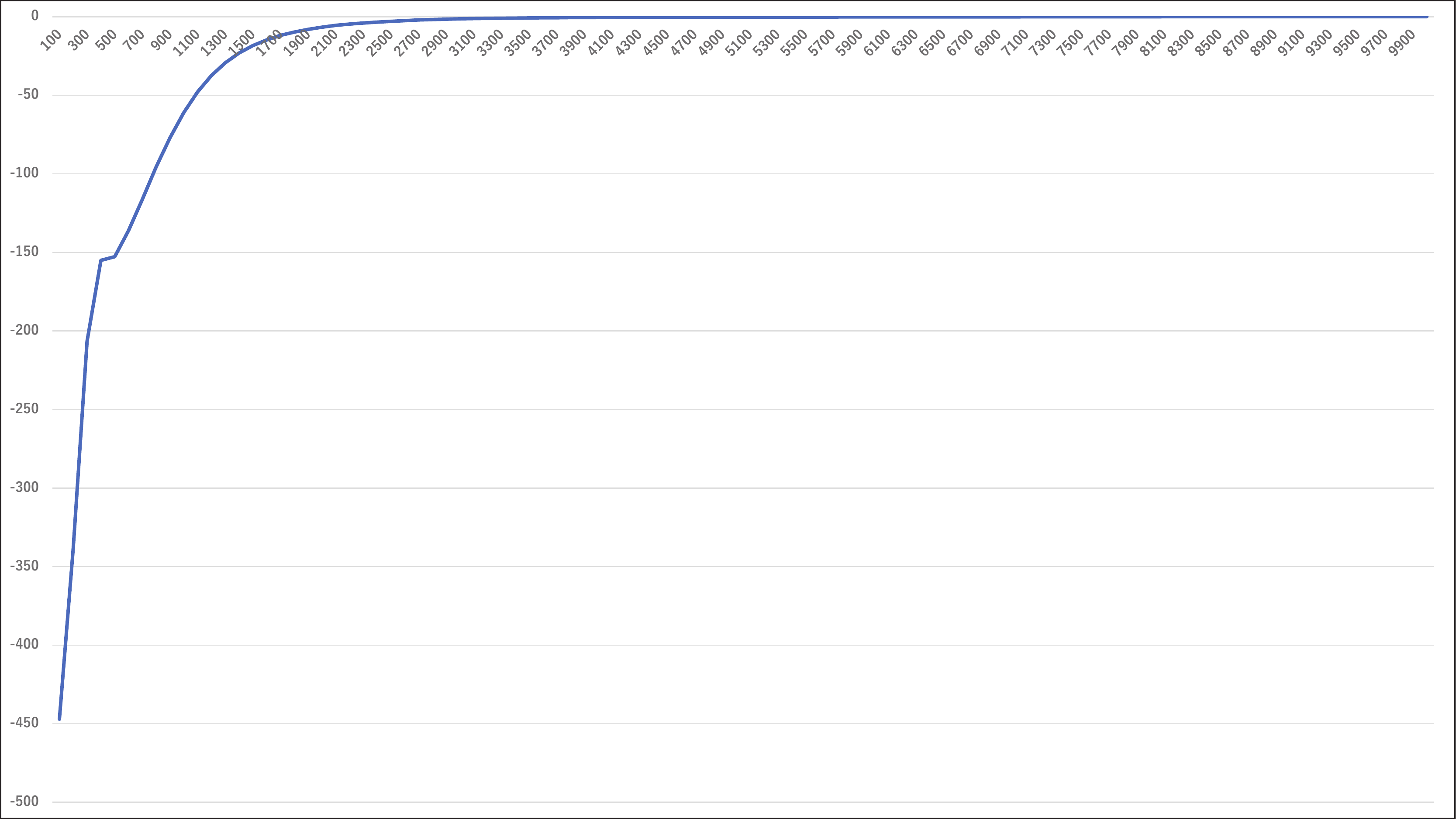}
\subcaption{Nug18}
\label{fig:Nug18}
\end{minipage} &
\begin{minipage}{0.3\columnwidth}
\centering
\includegraphics[width = 50mm,pagebox=cropbox,clip]{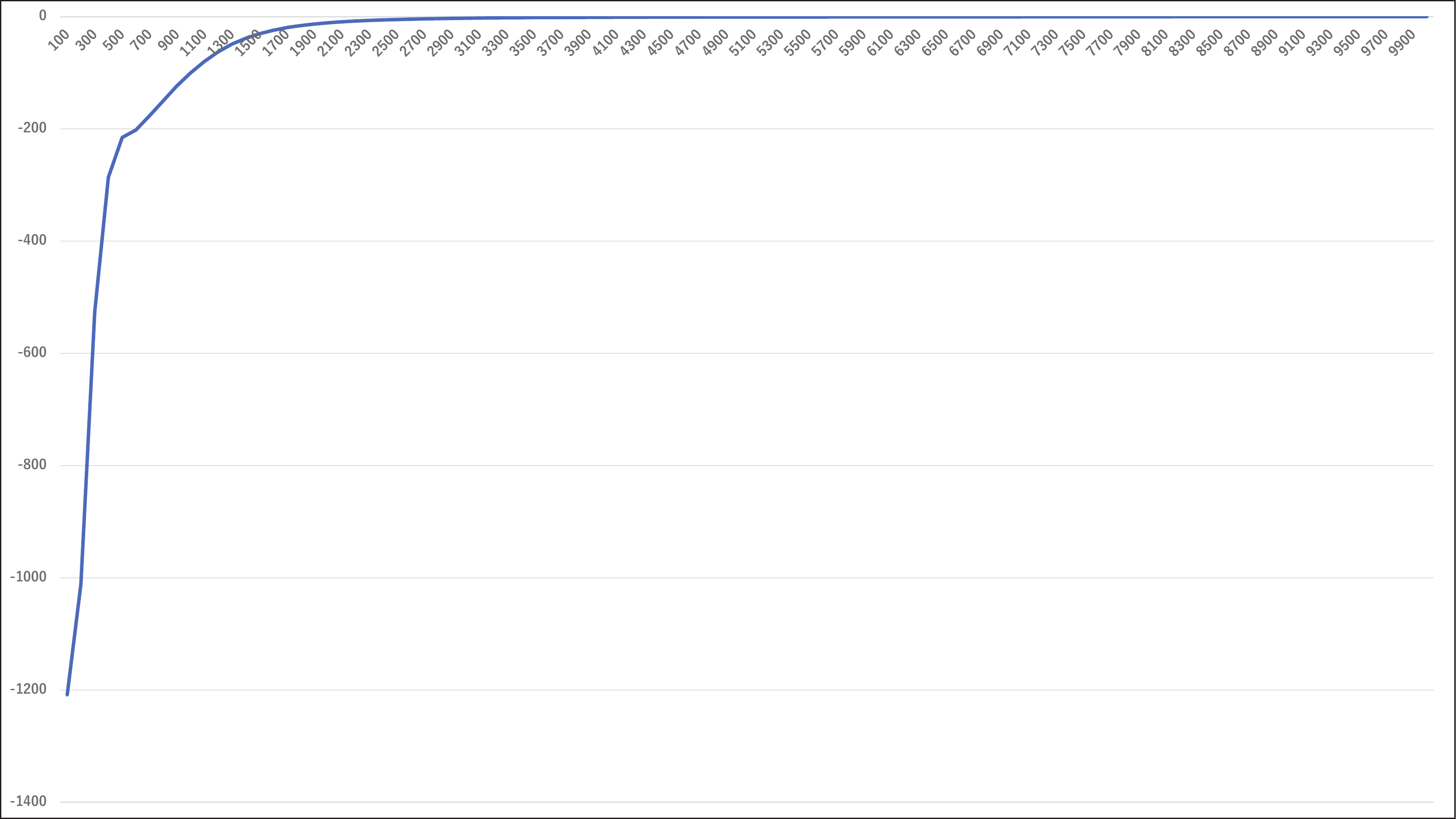}
\subcaption{Nug20}
\label{fig:Nug20} 
\end{minipage} &
\begin{minipage}{0.3\columnwidth}
\centering
\includegraphics[width = 50mm,pagebox=cropbox,clip]{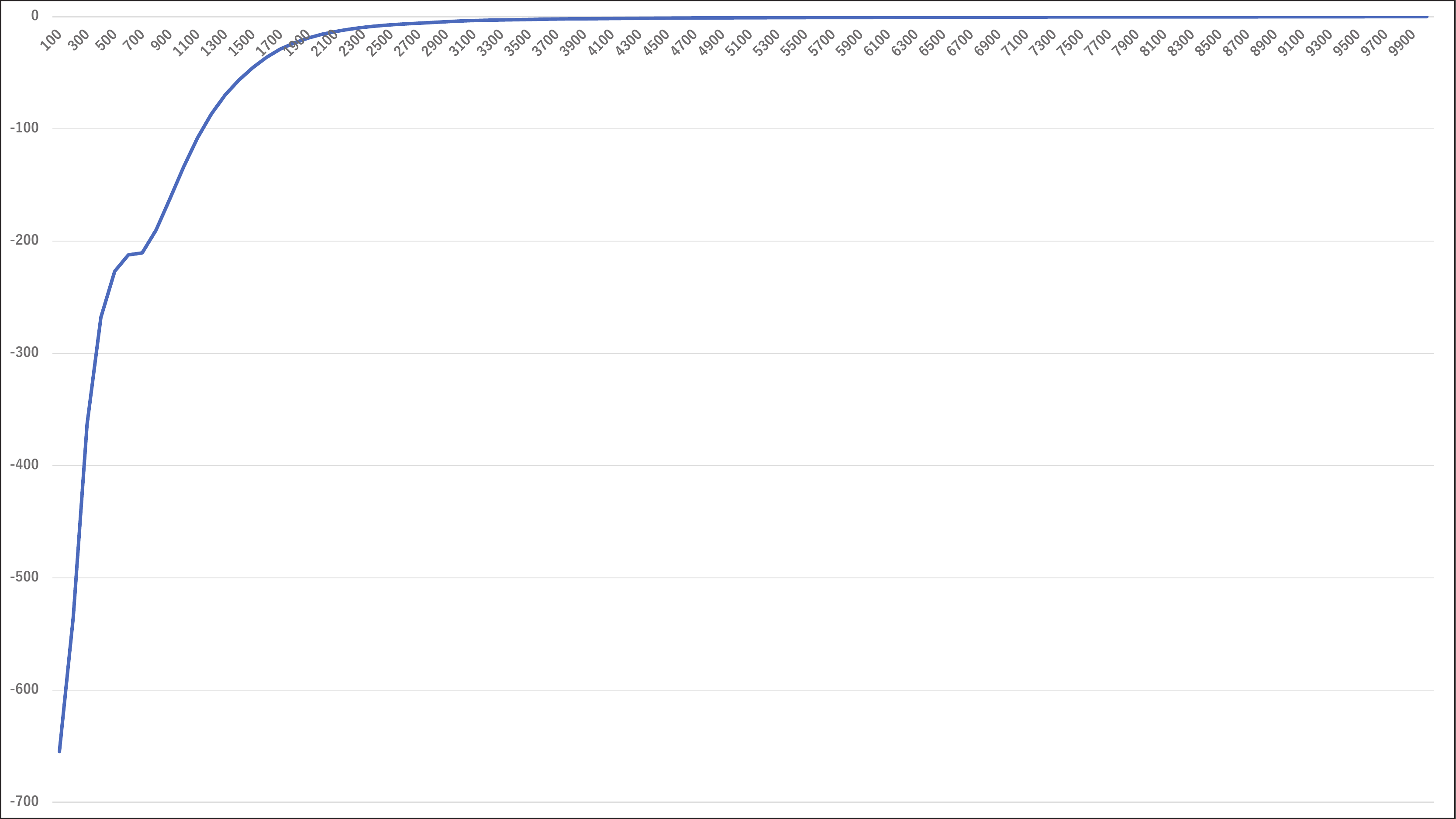}
\subcaption{Nug21}
\label{fig:Nug21} 
\end{minipage} \\
%----------------------------------
\begin{minipage}{0.3\columnwidth}
\centering
\includegraphics[width = 50mm,pagebox=cropbox,clip]{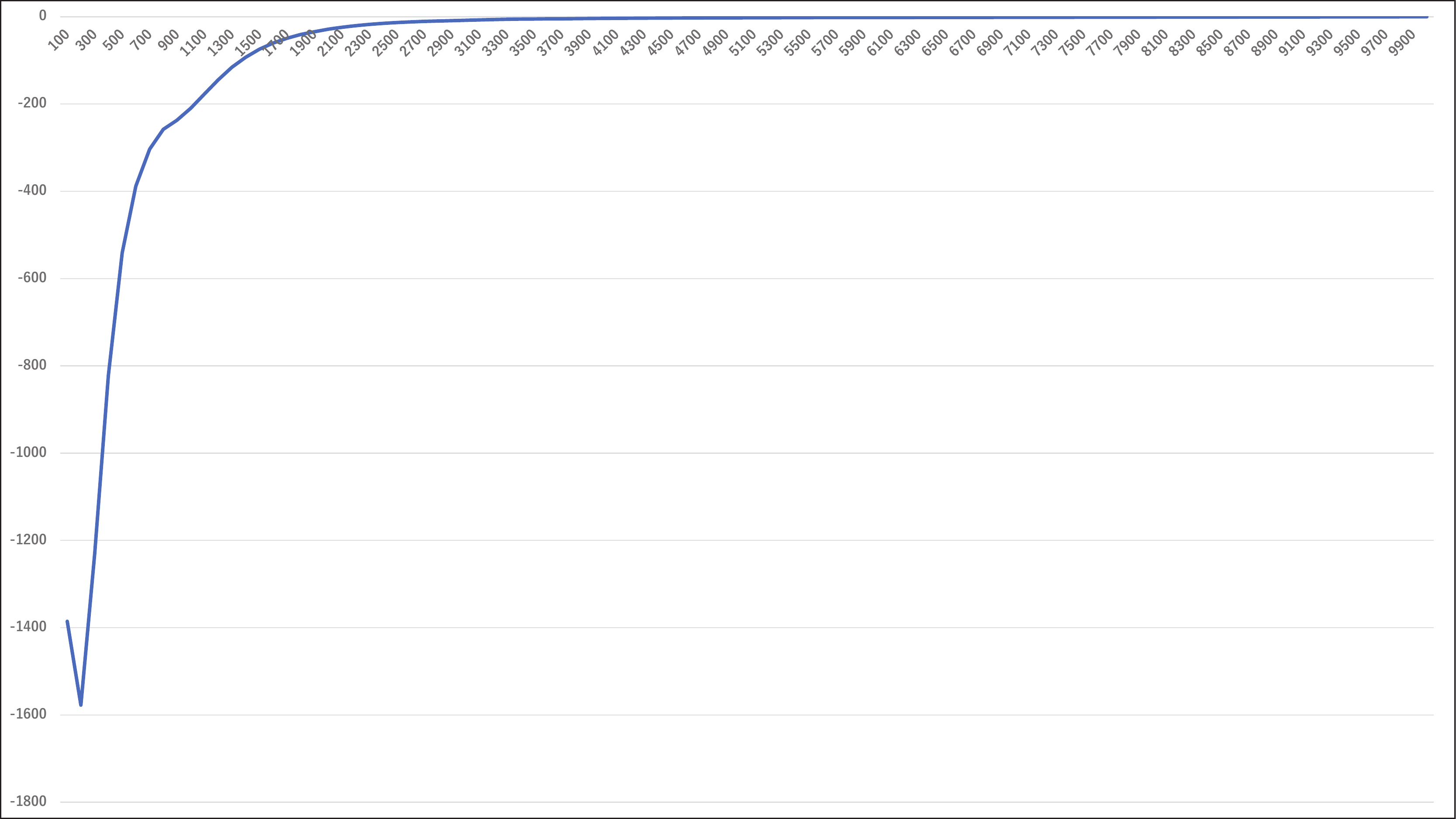}
\subcaption{Nug22}
\label{fig:Nug22}
\end{minipage} &
\begin{minipage}{0.3\columnwidth}
\centering
\includegraphics[width = 50mm,pagebox=cropbox,clip]{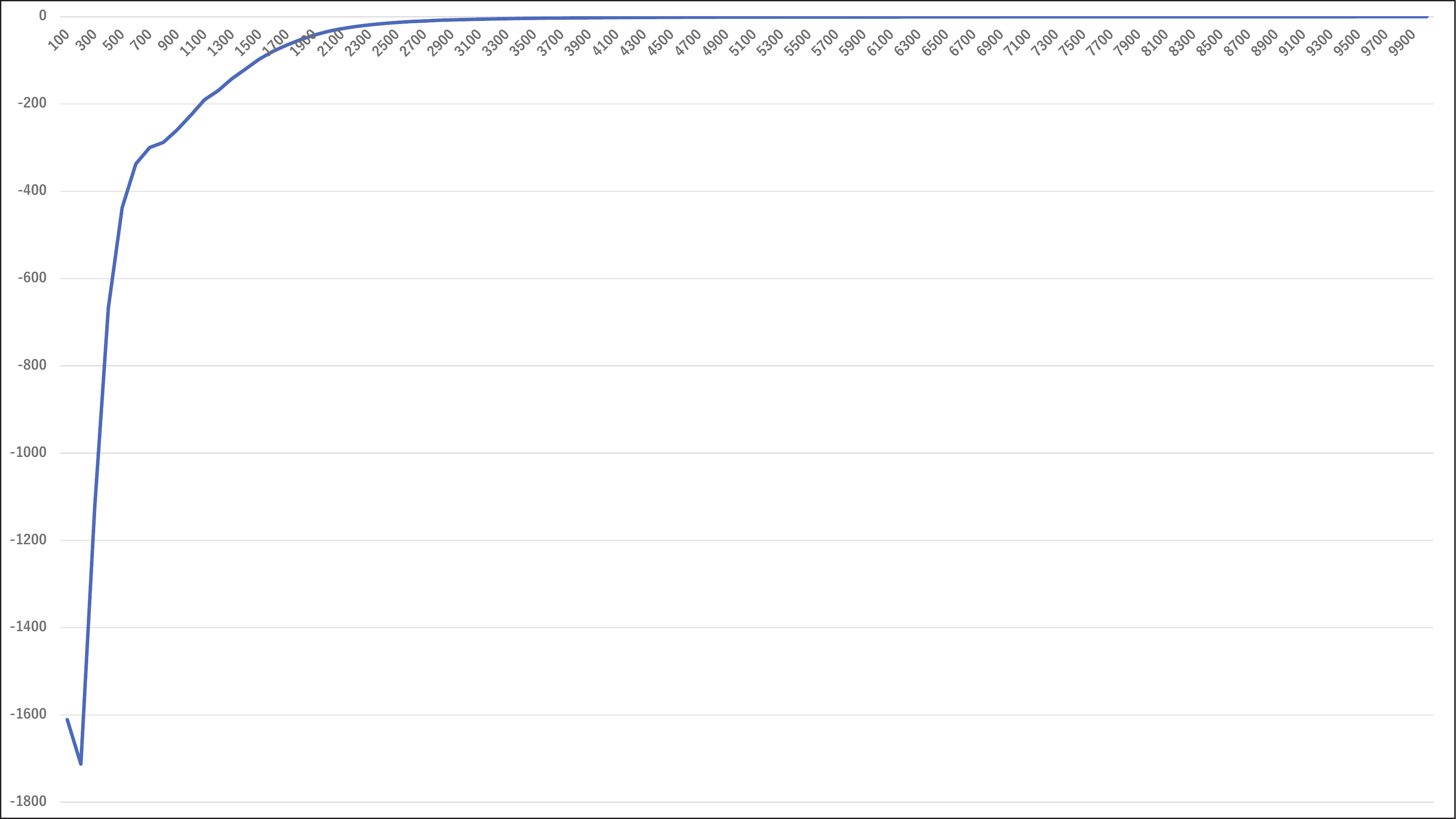}
\subcaption{Nug24}
\label{fig:Nug24}
\end{minipage} &
\begin{minipage}{0.3\columnwidth}
\centering
\includegraphics[width = 50mm,pagebox=cropbox,clip]{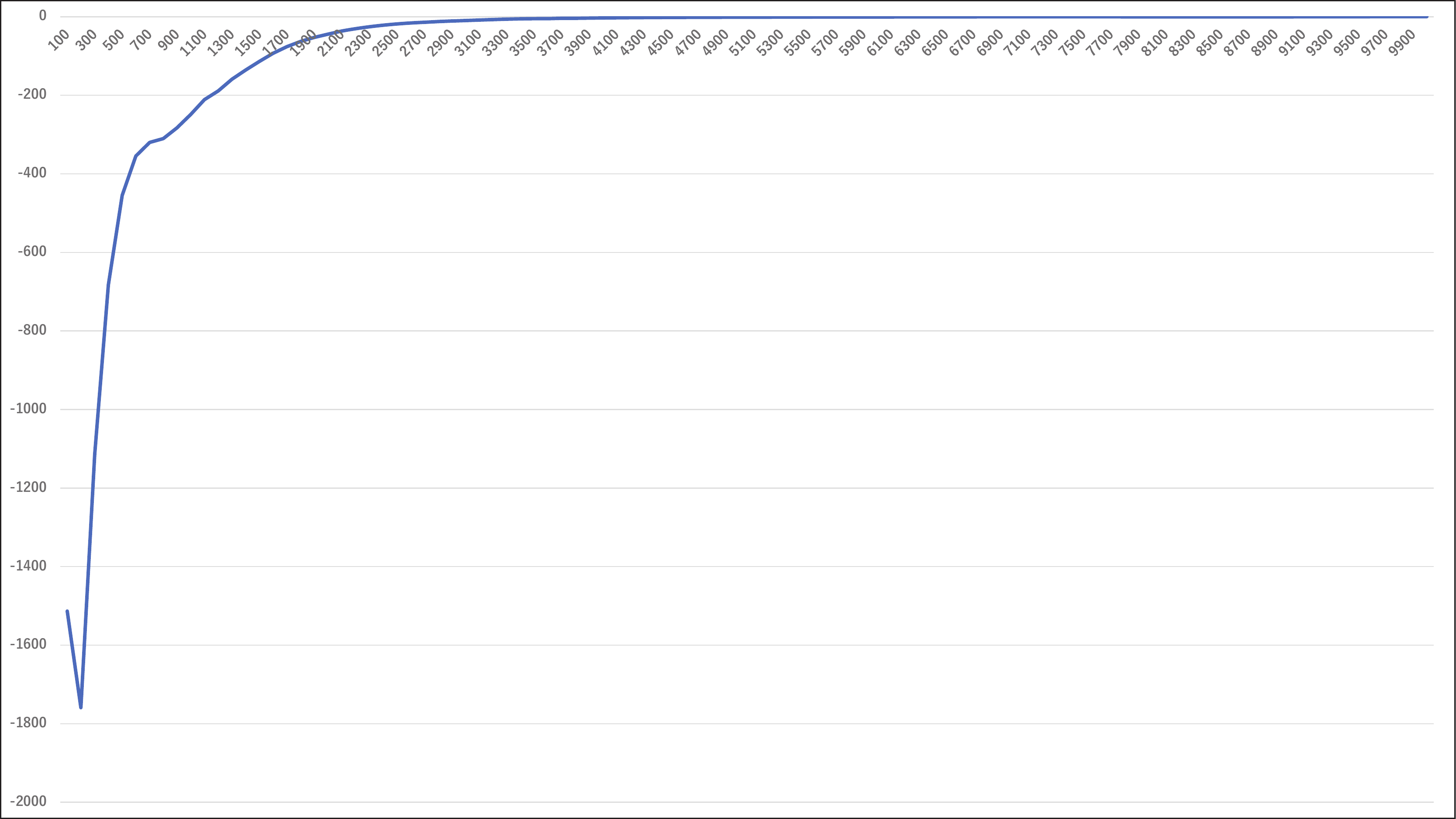}
\subcaption{Nug25}
\label{fig:Nug25}
\end{minipage} \\
%----------------------------------
\begin{minipage}{0.3\columnwidth}
\centering
\includegraphics[width = 50mm,pagebox=cropbox,clip]{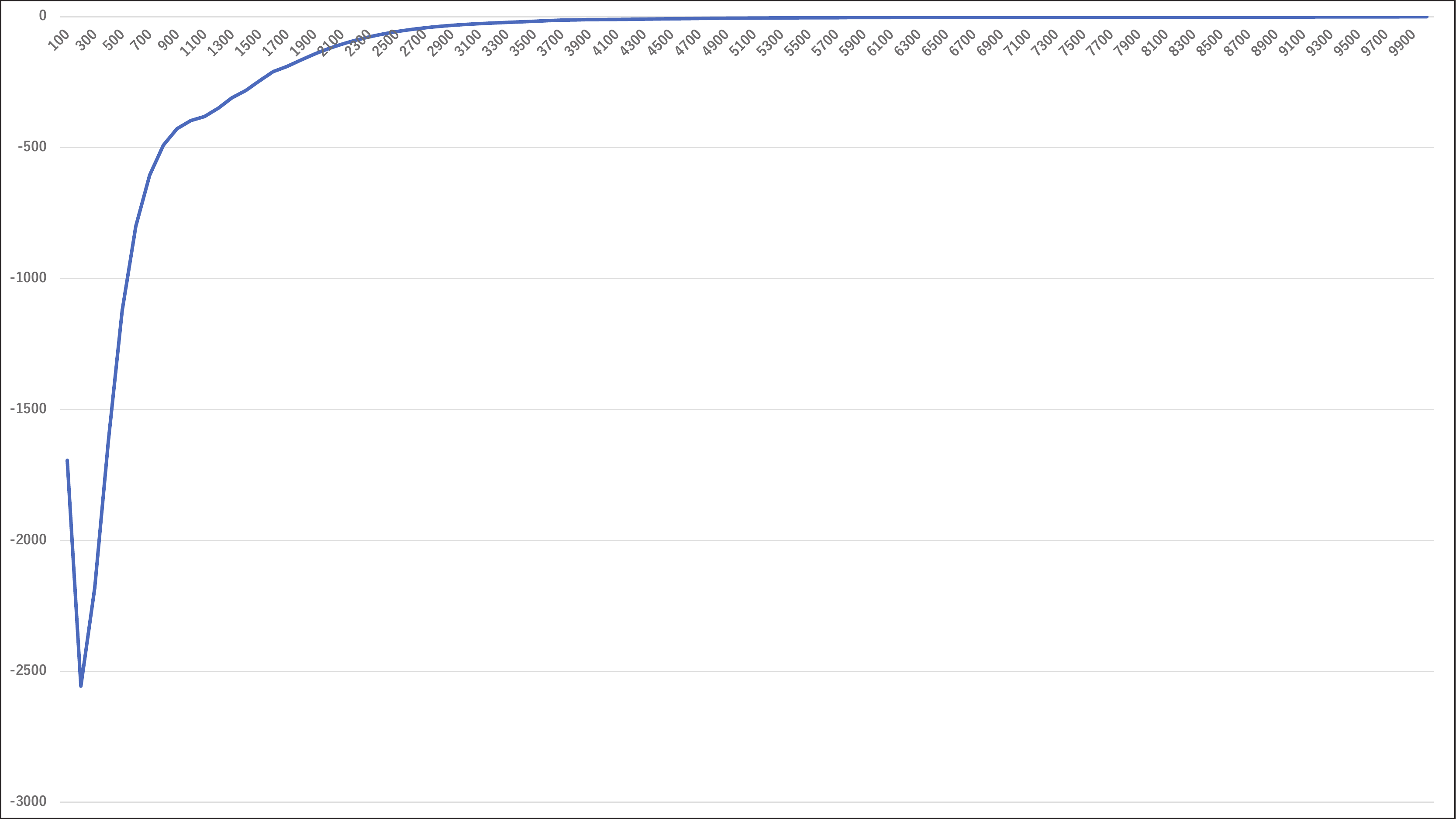}
\subcaption{Nug27}
\label{fig:Nug27}
\end{minipage} &
\begin{minipage}{0.3\columnwidth}
\centering
\includegraphics[width = 50mm,pagebox=cropbox,clip]{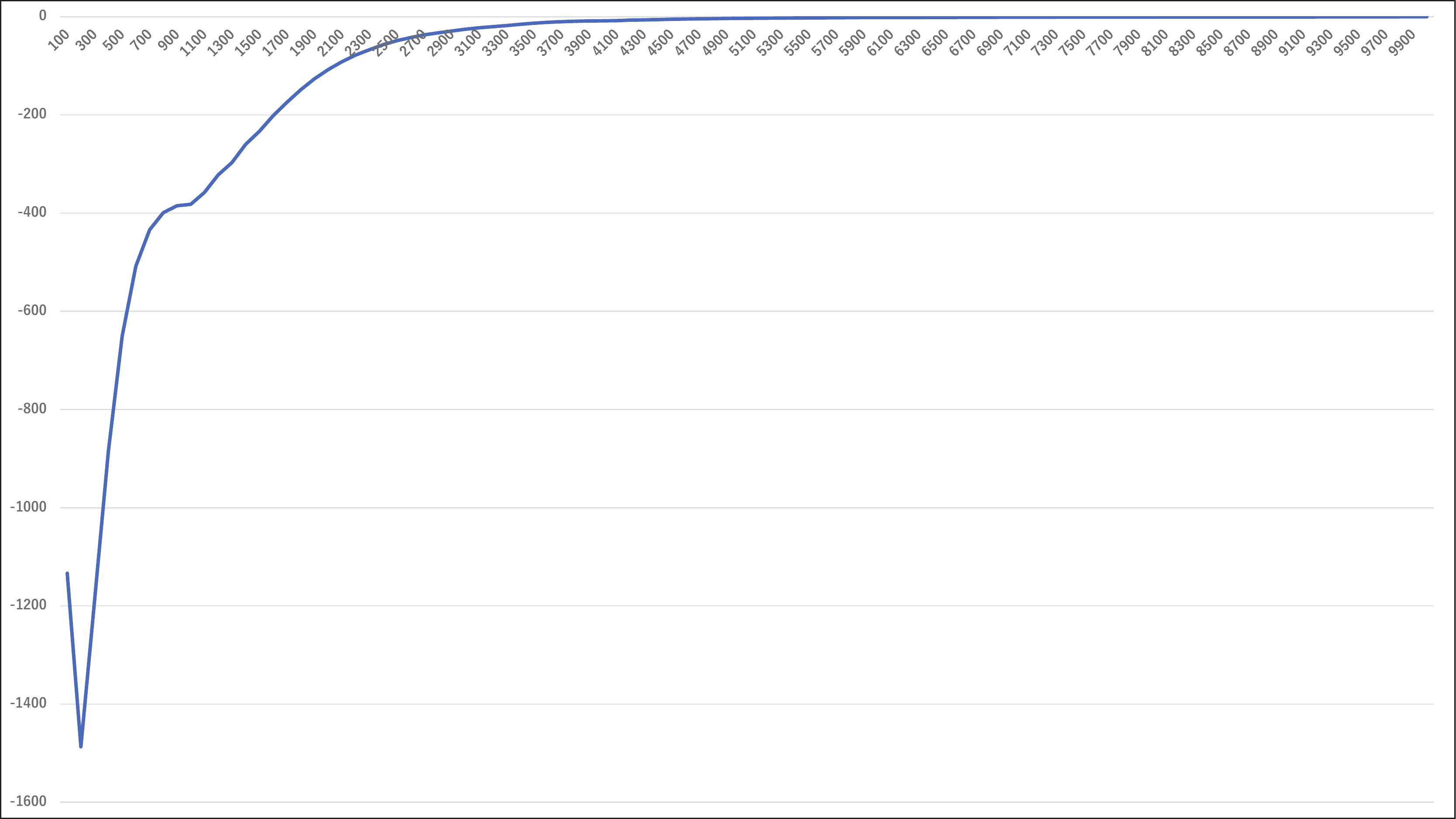}
\subcaption{Nug28}
\label{fig:Nug28}
\end{minipage} &
\begin{minipage}{0.3\columnwidth}
\centering
\includegraphics[width = 50mm,pagebox=cropbox,clip]{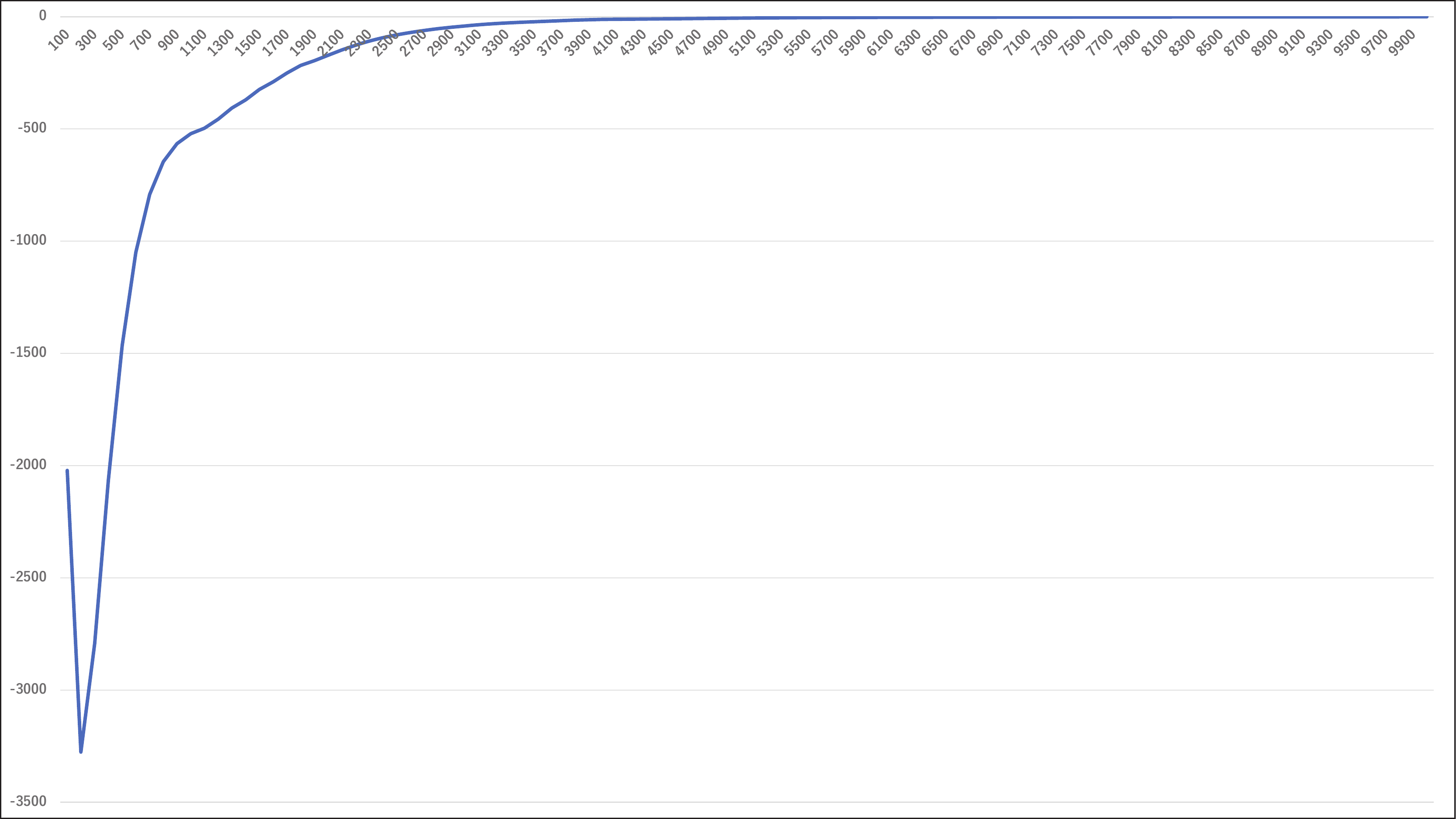}
\subcaption{Nug30}
\label{fig:Nug30}
\end{minipage}
%----------------------------------
\end{tabular}
\caption{Results for   ``Nug'' instances}
\label{fig:Nug}
\end{figure}

\clearpage

\begin{figure}[H]
\begin{tabular}{ccc}
%----------------------------------
\begin{minipage}{0.3\columnwidth}
\centering
\includegraphics[width = 50mm,pagebox=cropbox,clip]{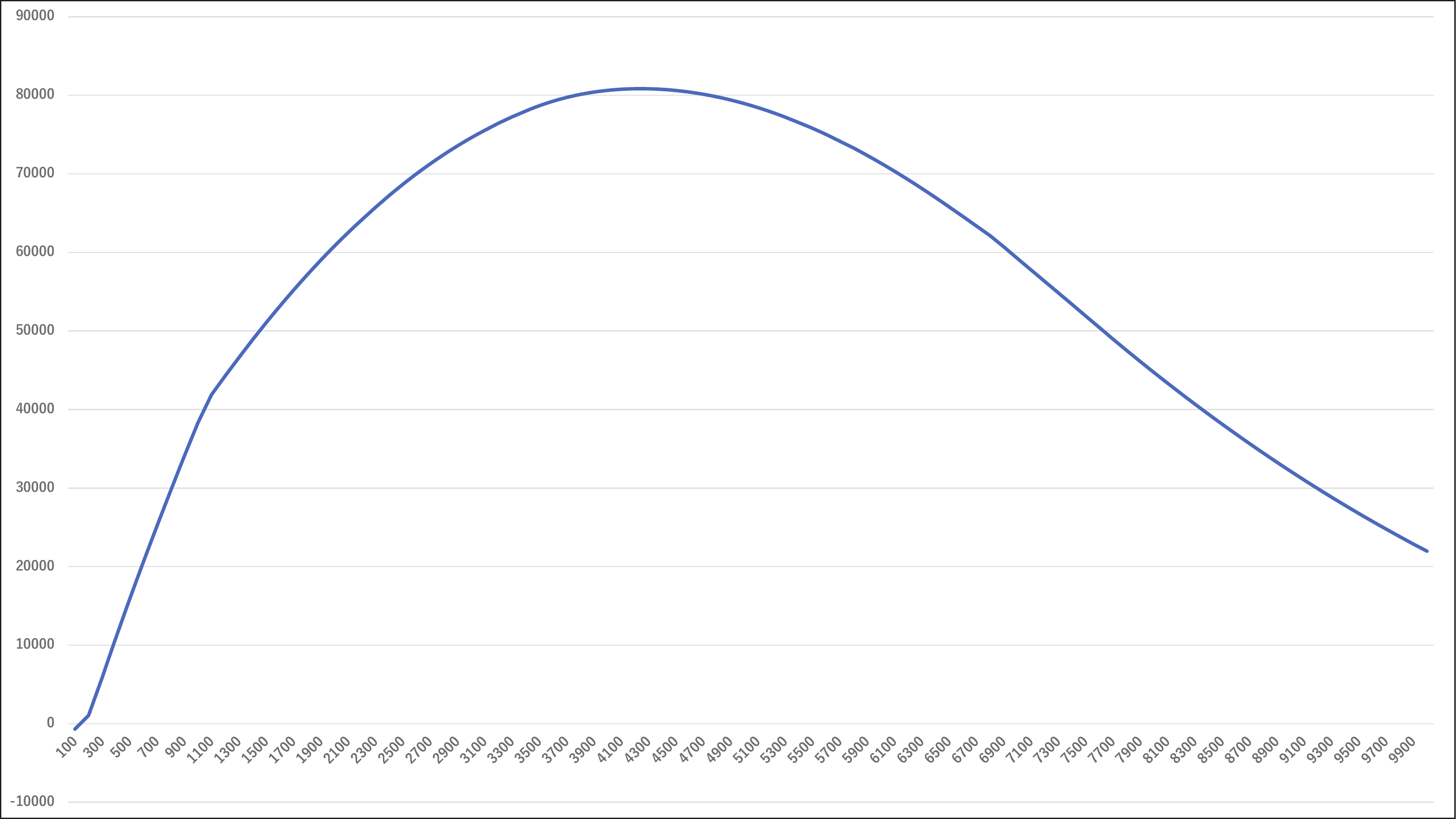}
\subcaption{Tai12a}
\label{fig:Tai12a}
\end{minipage} &
\begin{minipage}{0.3\columnwidth}
\centering
\includegraphics[width = 50mm,pagebox=cropbox,clip]{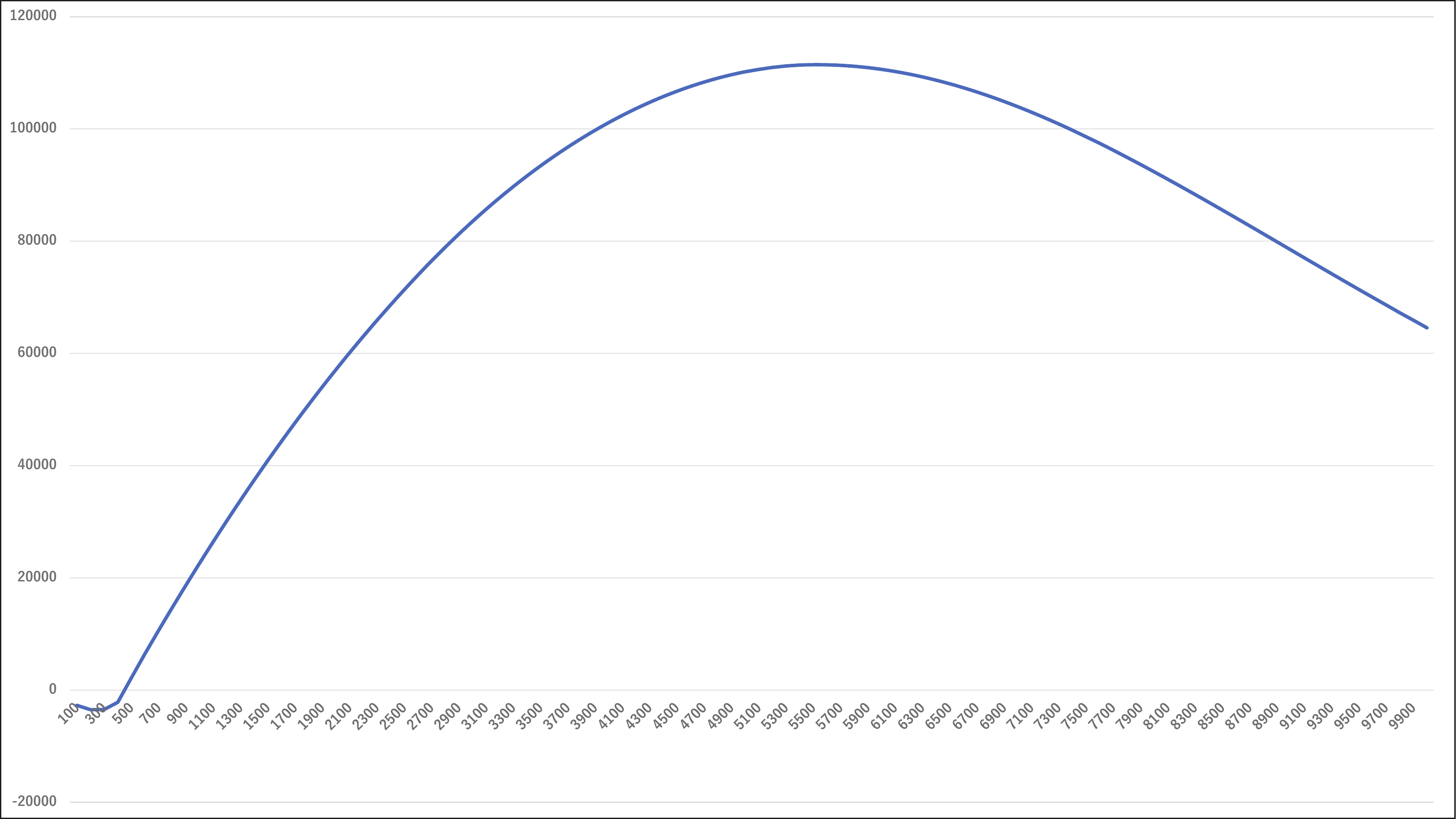}
\subcaption{Tai15a}
\label{fig:Tai15a}
\end{minipage} &
\begin{minipage}{0.3\columnwidth}
\centering
\includegraphics[width = 50mm,pagebox=cropbox,clip]{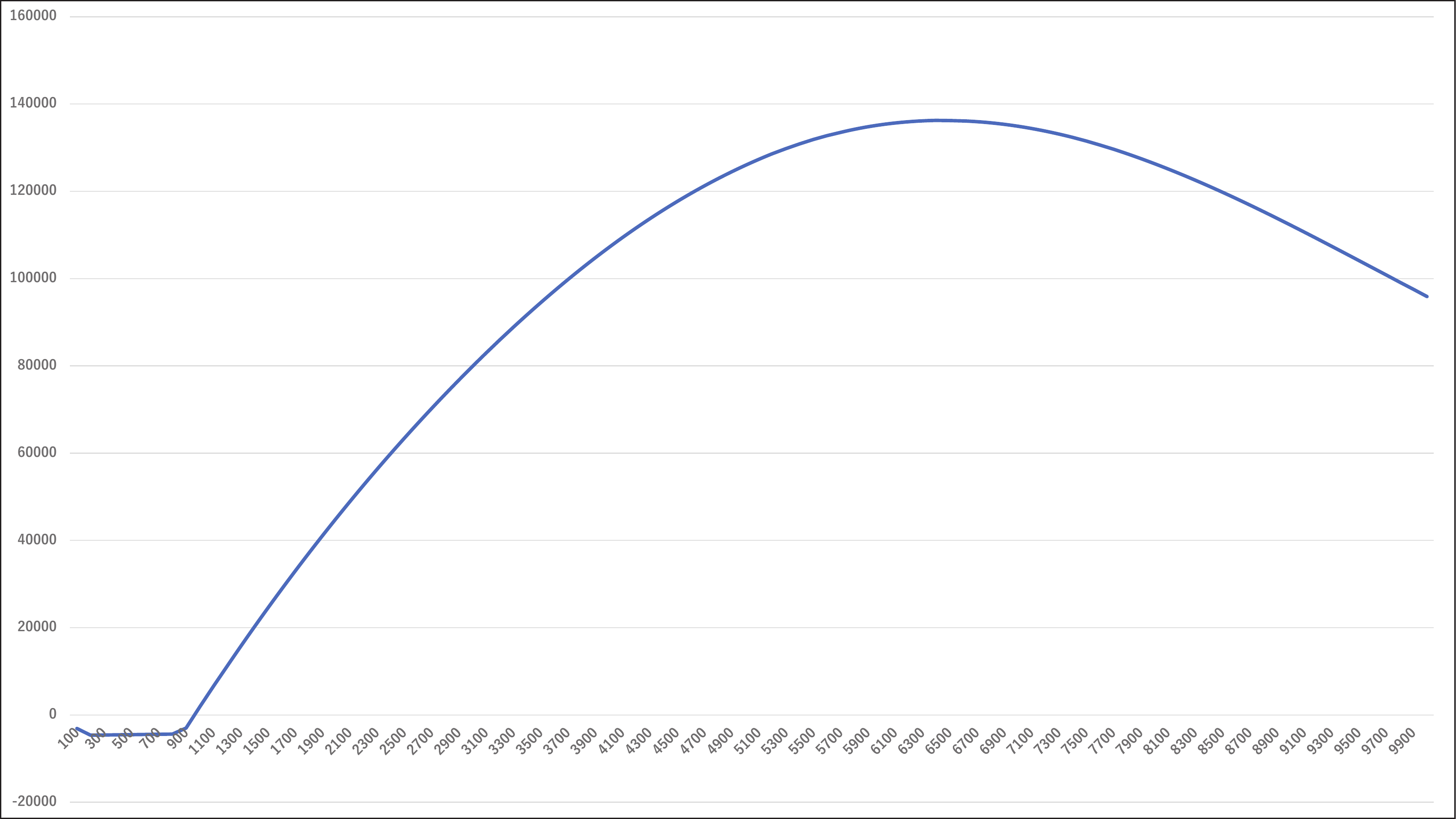}
\subcaption{Tai17a}
\label{fig:Tai17a}
\end{minipage} \\
%----------------------------------
\begin{minipage}{0.3\columnwidth}
\centering
\includegraphics[width = 50mm,pagebox=cropbox,clip]{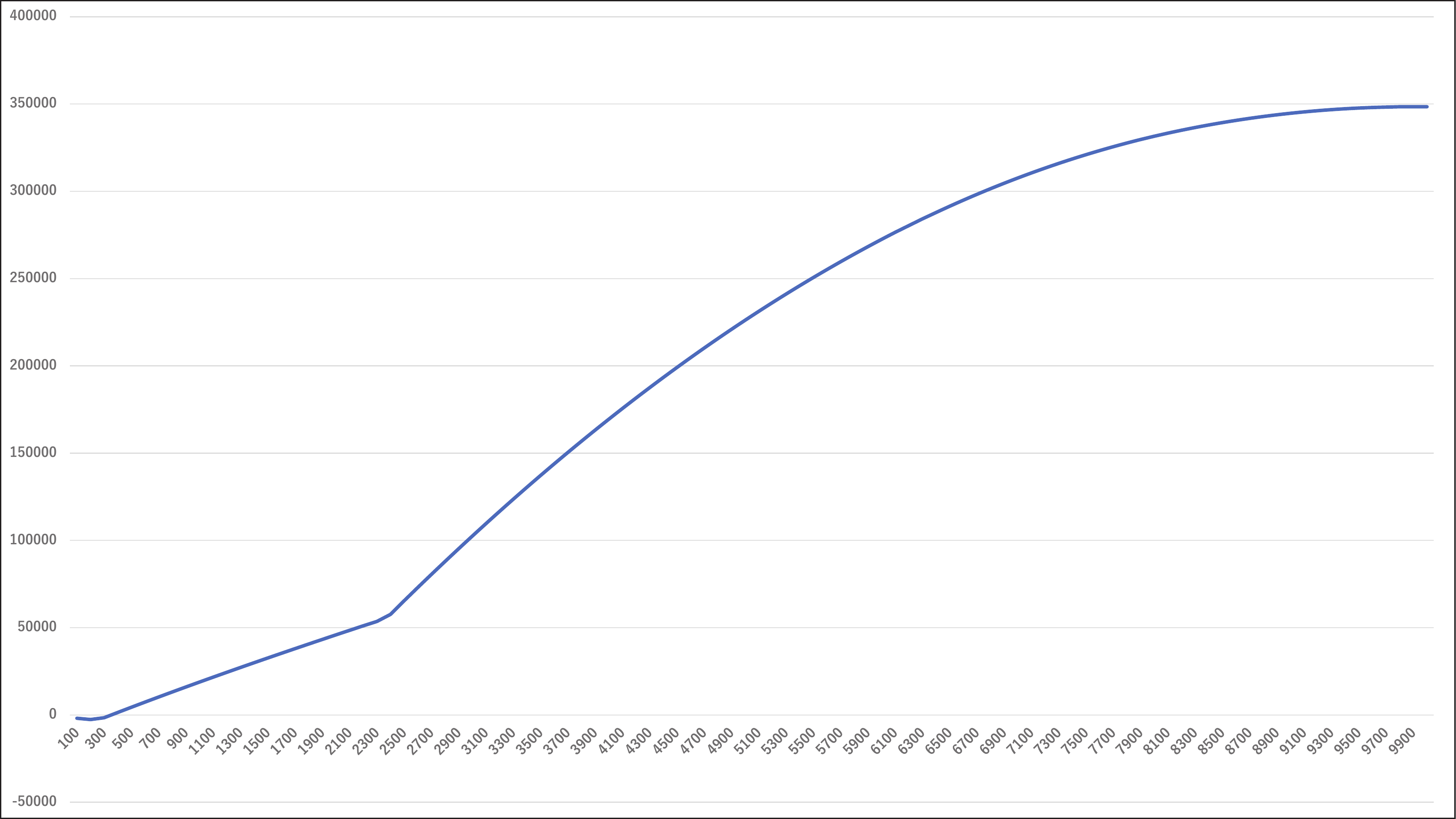}
\subcaption{Tai20a}
\label{fig:Tai20a}
\end{minipage} &
\begin{minipage}{0.3\columnwidth}
\centering
\includegraphics[width = 50mm,pagebox=cropbox,clip]{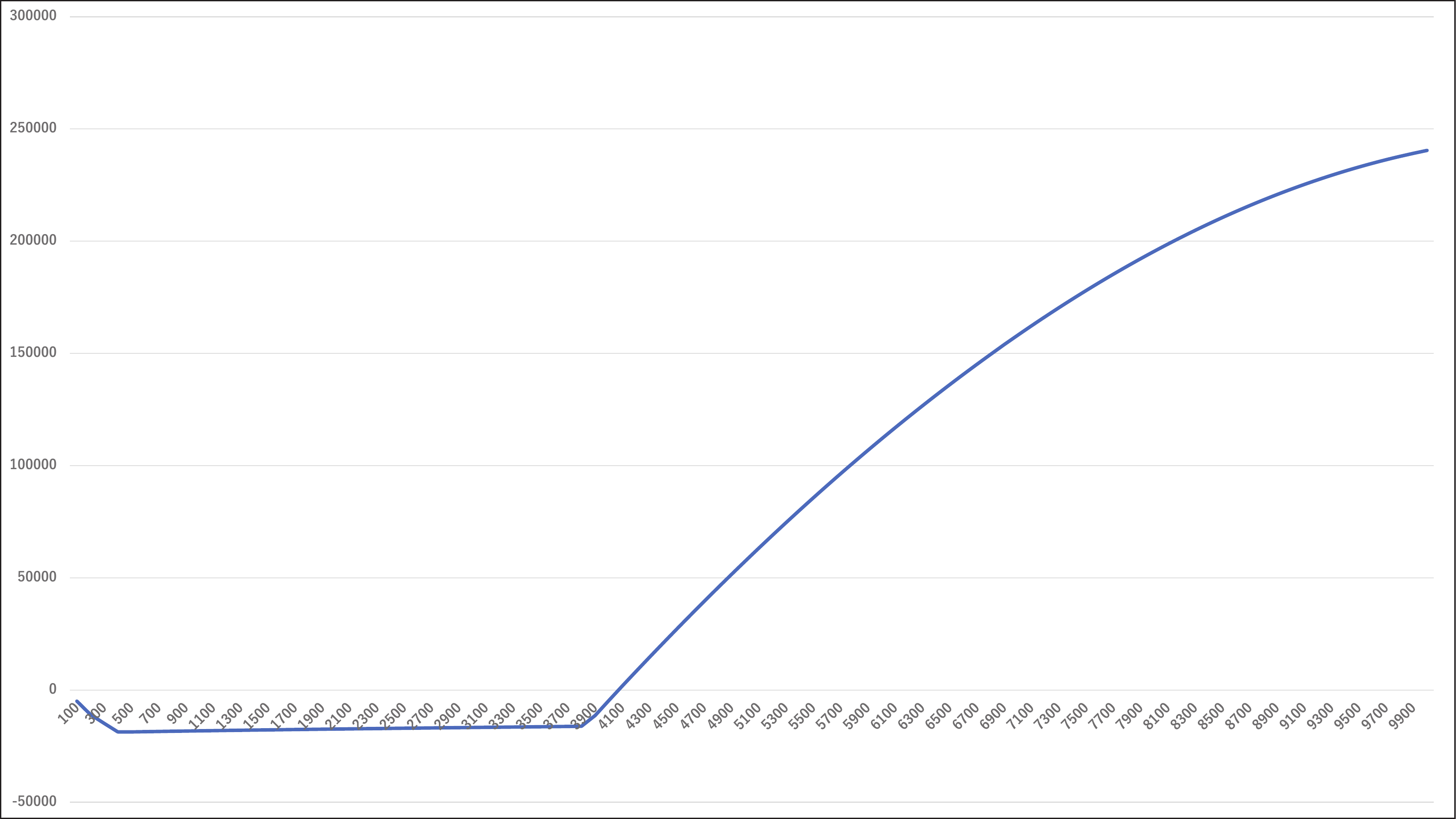}
\subcaption{Tai25a}
\label{fig:Tai25a}
\end{minipage} &
\begin{minipage}{0.3\columnwidth}
\centering
\includegraphics[width = 50mm,pagebox=cropbox,clip]{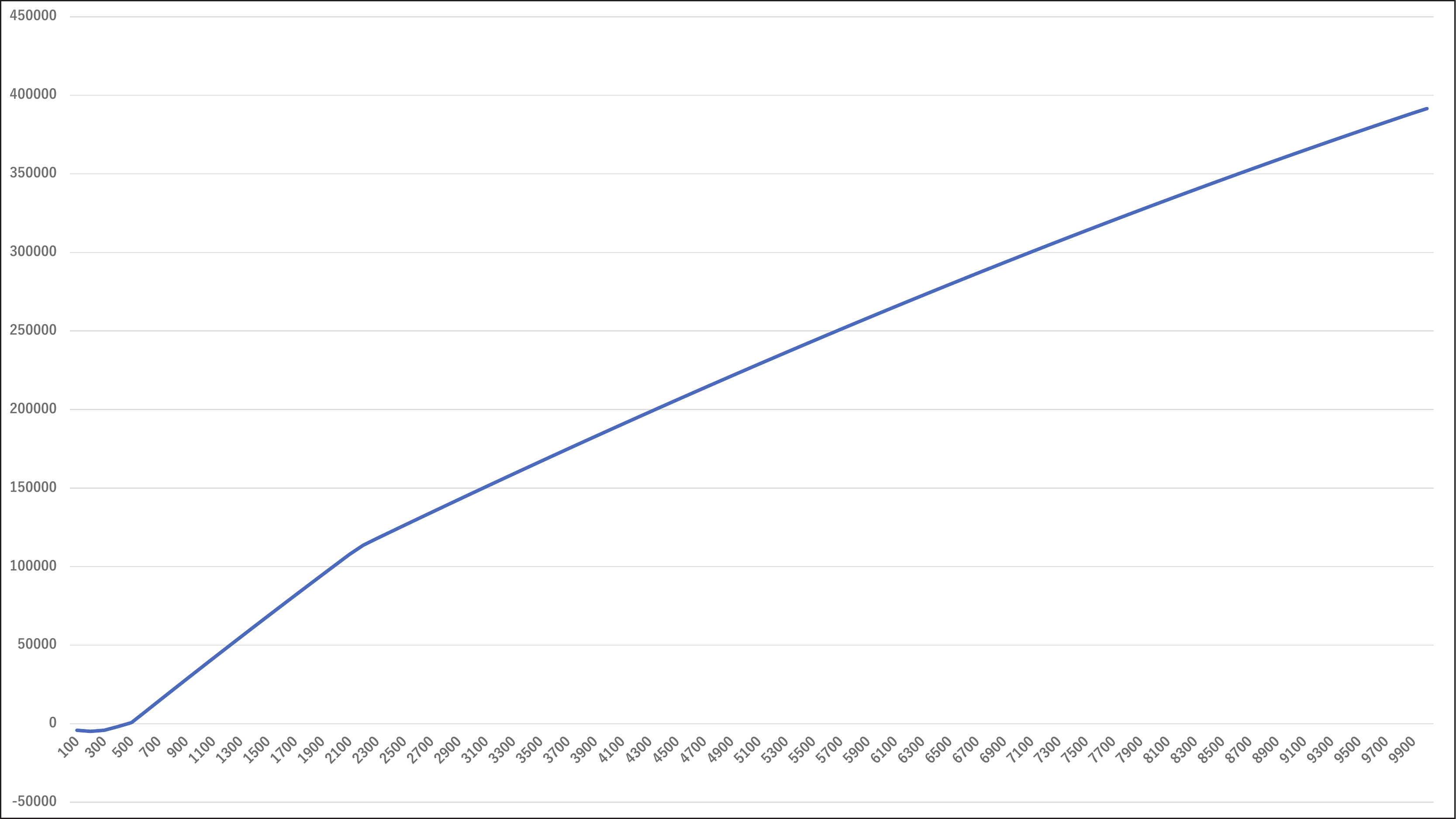}
\subcaption{Tai30a}
\label{fig:Tai30a}
\end{minipage} \\
%----------------------------------
\begin{minipage}{0.3\columnwidth}
\centering
\includegraphics[width = 50mm,pagebox=cropbox,clip]{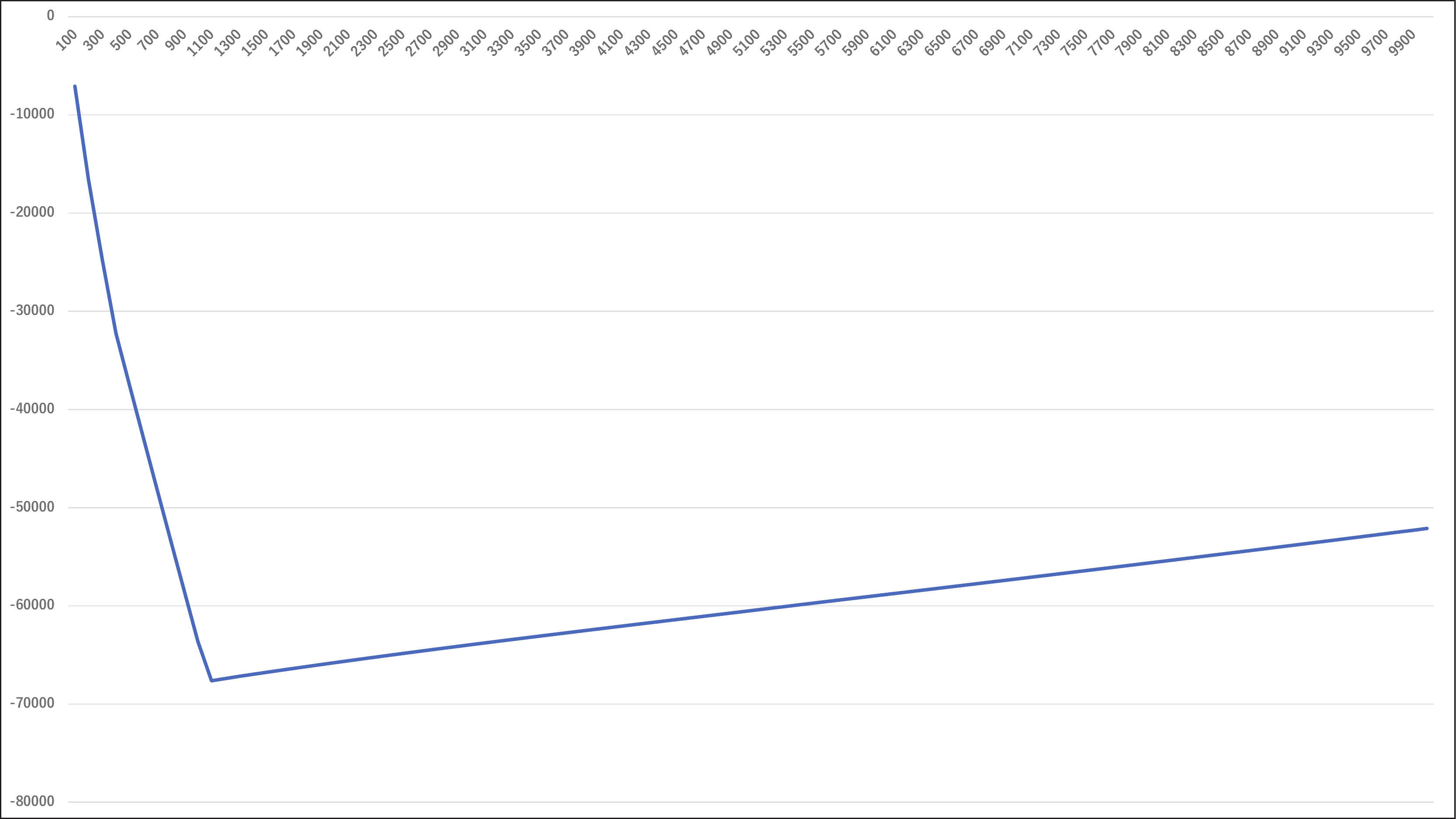}
\subcaption{Tai35a}
\label{fig:Tai35a}
\end{minipage} & & 
%----------------------------------
\end{tabular}
\caption{Results for  ``Tai-a'' instances}
\label{fig:Tai-a}
\end{figure}

\begin{figure}[H]
\begin{tabular}{ccc}
%----------------------------------
\begin{minipage}{0.3\columnwidth}
\includegraphics[width = 50mm,pagebox=cropbox,clip]{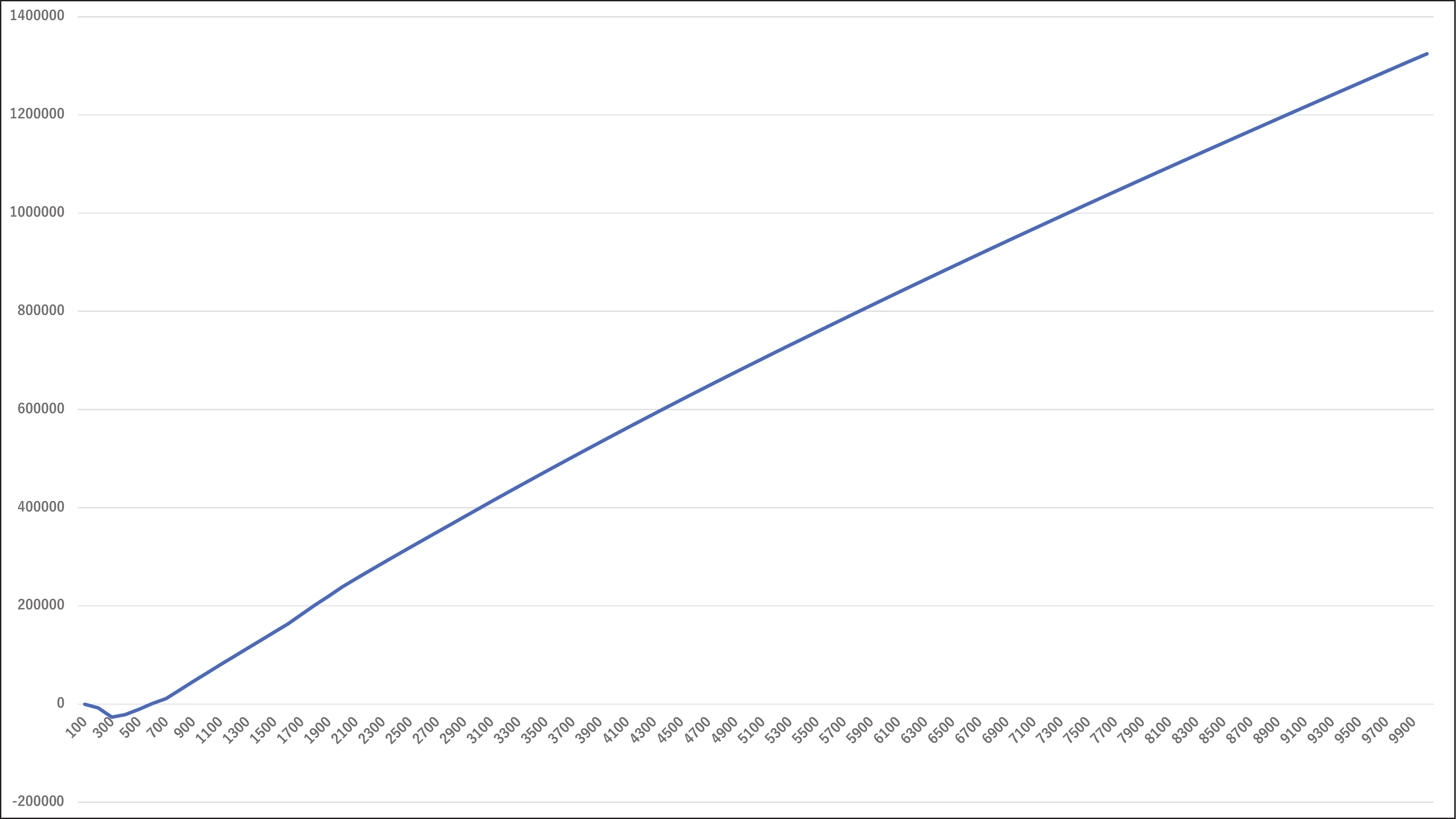}
\subcaption{Els19}
\label{fig:Els19}
\end{minipage} & 
\begin{minipage}{0.3\columnwidth}
\includegraphics[width = 50mm,pagebox=cropbox,clip]{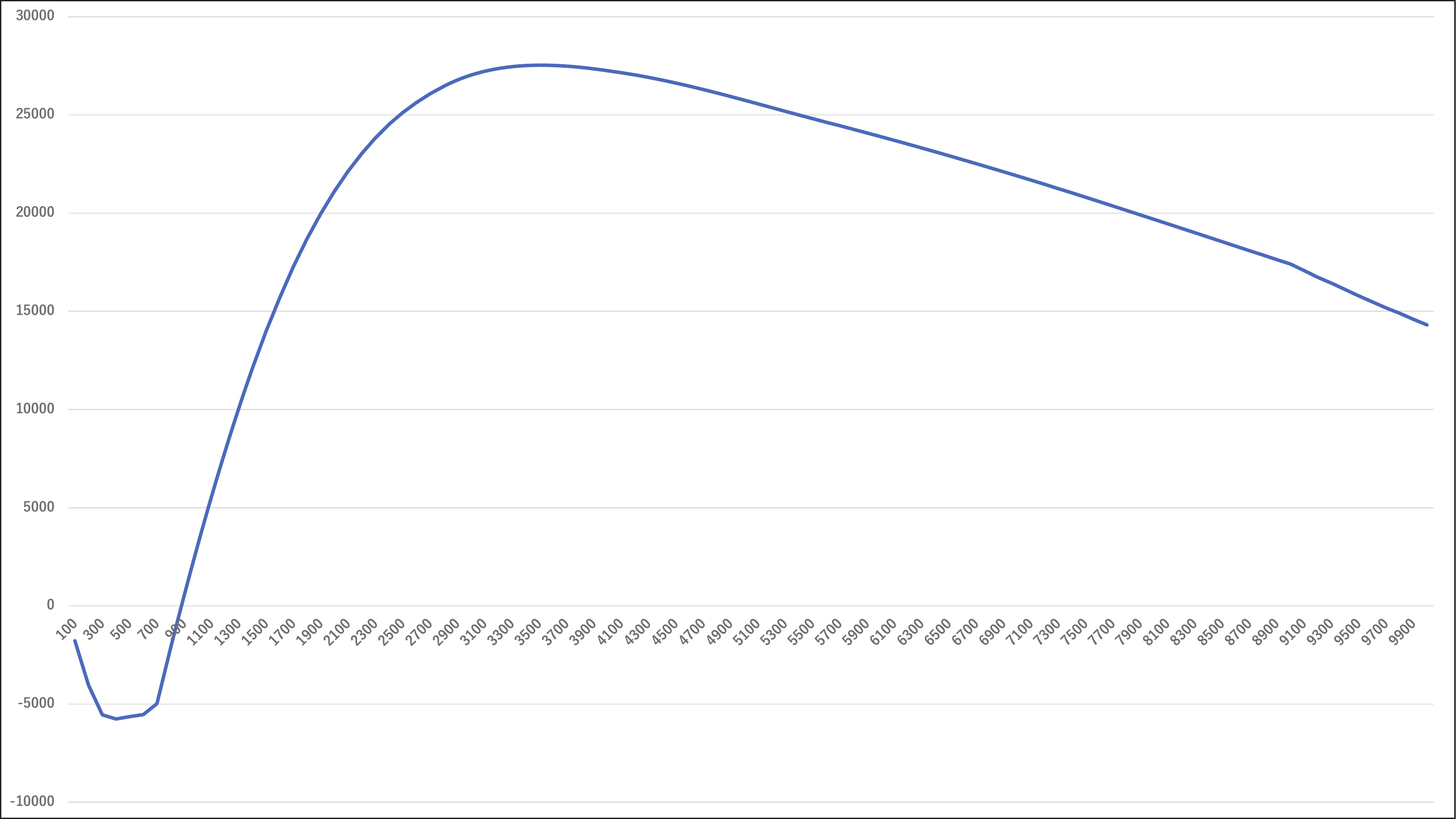}
\subcaption{Tho30}
\label{fig:Tho30}
\end{minipage} &
%----------------------------------
\end{tabular}
\caption{Results for ``Els19'' and ``Tho30'' instances}
\label{fig:ElsandTho}
\end{figure}

\begin{figure}[H]
\centering
\includegraphics[width=12.5cm,height = 4.5cm,pagebox=cropbox,clip]{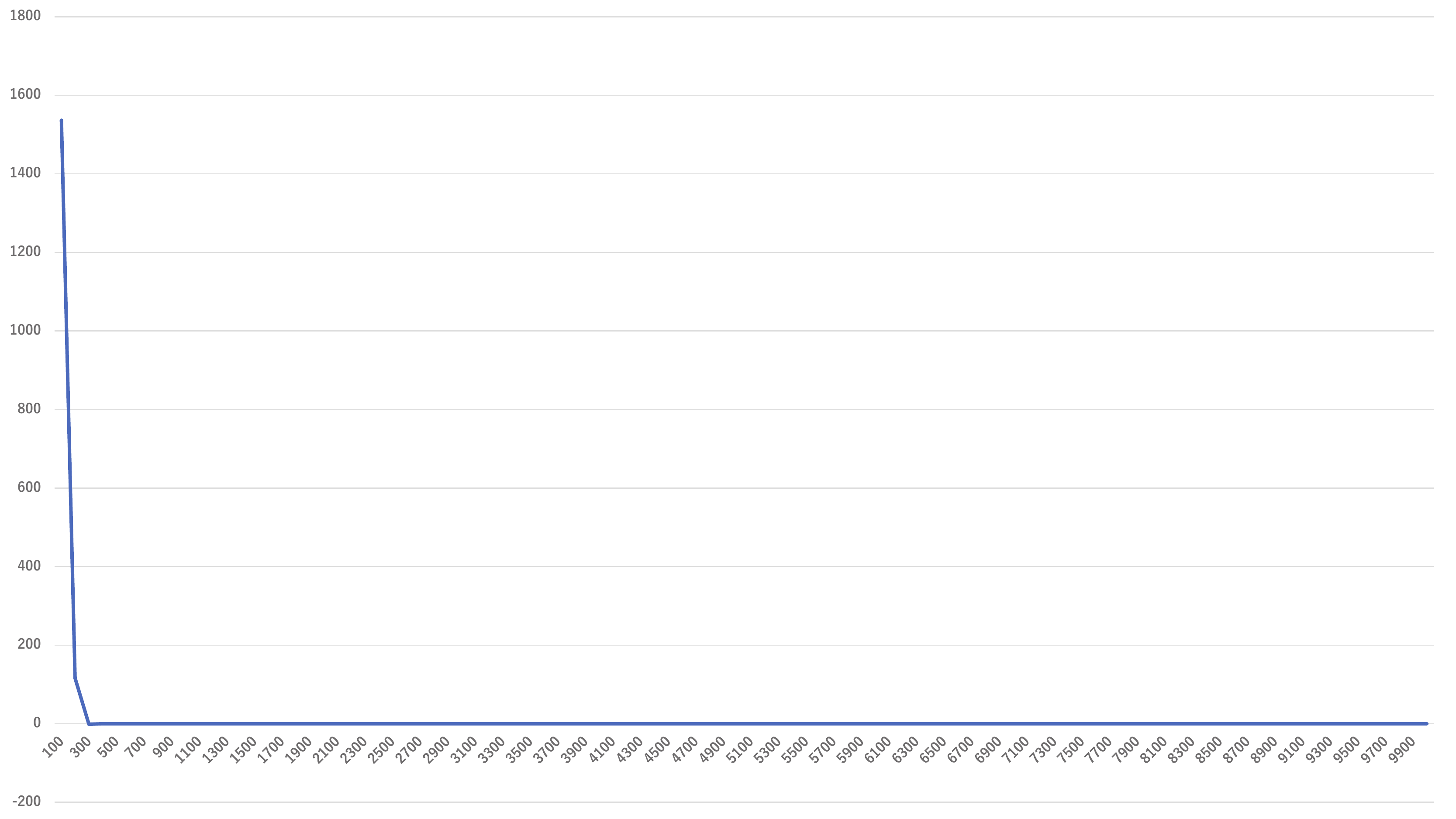}
\caption{Difference between the lower bounds for ``Had12''}
\label{fig:Had12_difference}
\end{figure}

\begin{table}[H]
\begin{center}
\begin{tabular}{lrrrrr} \hline
Prob. &\ \ Opt. & \ \ Centering LB& \ \ Centering \#Iter. & \ \ Standard LB & \ \ Standard \#Iter. \\ \hline
Had12 & 1652&1651.666717 &1400 &1651.918164 &300  \\ \hline
Had14 & 2724&2723.602388 &2200 &2723.683821&500  \\ \hline
Had16 & 3720&3719.512521&2000 &3719.322802&400  \\ \hline
Had18 & 5358&5357.509444&5400 &5357.523159&1800  \\ \hline
Had20 & 6922&6921.514997&3800 &6921.599327&1400 \\ \hline
\end{tabular}
\caption{Comparison of results for ``Had'' instances}
\label{tb:tableHad}
\end{center}
\end{table}

\begin{table}[H]
\begin{center}
\begin{tabular}{lrrr} \hline
Prob. &\ \ Slow Down LB & \ \ Small Diff. \#Iter.& \ \ LB at 10000 Iter.\\ \hline
Nug12 &567&2000&567.99 \\ \hline
Nug14 &1009&2700&1010.11 \\ \hline  
Nug15 &1140&3000&1140.56\\ \hline
Nug16a &1598&3200&1599.27\\ \hline
Nug16b &1217&2900&1218.25\\ \hline
Nug17 &1706&3400&1707.10 \\ \hline
Nug18 &1892&3300&1893.53 \\ \hline
Nug20 &2505&4200&2506.33 \\ \hline
Nug21 &2380&5100&2381.89 \\ \hline
Nug22 &3525&8300&3528.53 \\ \hline
Nug24 &3399&5100&3401.05 \\ \hline
Nug25 &3624&5500&3625.85 \\ \hline
Nug27 &5125&9300&5129.36 \\ \hline
Nug28 &5023&7300&5025.56 \\ \hline
Nug30 &5947&8700&5949.43 \\ \hline
\end{tabular}
\caption{Comparison of results for ``Nug'' instances (1)}
\label{tb:tableNug1}
\end{center}
\end{table}

\begin{table}[H]
\begin{center}
\begin{tabular}{lrrrrr} \hline
Prob. & \ \ Opt. & \ \ Centering LB & \ \ Centering \#Iter. & \ \ Standard LB & \ \ Standard \#Iter. \\ \hline
Nug12 & 578& 567.0016543&2100 &567.0206964&500  \\ \hline
Nug14 & 1014&1009.018699&2600 &1009.047161&600  \\ \hline  
Nug15 & 1150&1140.052527&3500 &1140.04133&1100  \\ \hline
Nug16a & 1610&1598.013623&3000 &1598.175632&1200  \\ \hline
Nug16b & 1240&1217.009463&2700 &1217.175222&1000  \\ \hline
Nug17 & 1732&1706.06183&3400 &1706.139849&1400  \\ \hline
Nug18 & 1930&1892.09605&3000 &1892.124008&1200  \\ \hline
Nug20 & 2570&2505.029389&3900 &2505.033302&1400  \\ \hline
Nug21 & 2438&2380.015468&4200 &2380.053905&1800  \\ \hline
Nug22 & 3596&3525.102227&5000 &3525.137032&2400  \\ \hline
Nug24 & 3488&3399.109253&4300 &3399.234062&1900 \\ \hline
Nug25 & 3744&3624.029764&4700 &3624.147168&1900 \\ \hline
Nug27 & 5234&5125.1364&5800 &5125.054708&2500  \\ \hline
Nug28 & 5166&5023.01236&5700 &5023.162624&2900  \\ \hline
Nug30 & 6124&5947.069915&6600 &5947.159729& 1700 \\ \hline
\end{tabular}
\caption{Comparison of results for ``Nug'' instances (2)}
\label{tb:tableNug2}
\end{center}
\end{table}

\end{document}